\documentclass[12pt,reqno
  ]{amsart}
\usepackage{amssymb,amsmath,amsthm,secdot,bbm,bm}%stmaryrd
\usepackage[english]{babel}
\usepackage[LGR,T1]{fontenc} % spell out all text encodings used
\usepackage{lmodern}
\usepackage[
bookmarks=true,         %generates bookmarks for all entries in the "Table of Contents"
bookmarksnumbered=true, %includes chapter-/header-/section-/subsection-/... -
colorlinks=true, pdfstartview=FitV, linkcolor=blue, citecolor=blue,
urlcolor=blue]{hyperref}
\usepackage[abbrev,msc-links,alphabetic]{amsrefs}
\usepackage[innermargin=1.4in,outermargin=1.4in,bottom=1.5in,marginparwidth=1.1in,marginparsep=3mm]{geometry}
\usepackage[shortlabels] {enumitem}
\usepackage{xcolor}
\usepackage[babel=true,kerning=true]{microtype}
\theoremstyle{plain}
\newtheorem{theorem}{Theorem}[section]

\newtheorem{conjecture}[theorem]{Conjecture}
\newtheorem{corollary}[theorem]{Corollary}

\newtheorem{lemma}[theorem]{Lemma}

\newtheorem{proposition}[theorem]{Proposition}

\theoremstyle{remark}

\newtheorem{condition}[theorem]{Condition}

\theoremstyle{definition}
\newtheorem{remark}[theorem]{Remark}
\newtheorem*{remark*}{Remark}
\newtheorem{definition}[theorem]{Definition}
\newtheorem{example}[theorem]{Example}

\numberwithin{equation}{section}

  \newcommand{\R}{{\mathbb{R}}}
  \newcommand{\Q}{{\mathbb{Q}}}
  \newcommand{\Rd}{{\mathbb{R}^d}}
  \newcommand\E{\mathbb{E}}

  \newcommand{\caa}{{\mathcal A}}
  \newcommand{\cbb}{{\mathcal B}}

  \newcommand{\cee}{{\mathcal E}}
  \newcommand{\cff}{{\mathcal F}}
  \newcommand{\cgg}{{\mathcal G}}
  \newcommand{\chh}{{\mathcal H}}
  \newcommand{\cii}{{\mathcal I}}
  \newcommand{\cjj}{{\mathcal J}}

  \newcommand{\C}{{\mathcal  C}}

  \newcommand{\css}{{\mathcal S}}
  \newcommand{\coo}{{\mathcal O}}
  \newcommand{\cmm}{{\mathcal M}}
  
  \newcommand{\cvv}{{\mathcal V}}

  \newcommand{\lt}{\left}
  \newcommand{\rt}{\right}

  \newcommand{\bj}{{\bm{j}}}

  \newcommand{\fgg}{\mathfrak{G}}

  \renewcommand{\P}{{\mathbb P}}
  \renewcommand{\(}{\lt(}
  \renewcommand{\)}{\rt)}

  \newcommand{\wei}[1]{{\langle#1 \rangle}}
  \newcommand{\bra}[1]{{[#1]}}
  \DeclareMathOperator*{\esssup}{ess\,sup}

\title{A stochastic sewing lemma and applications}
\thanks{Supported by the Leverhulme Trust through Martin Hairer's leadership award.}
\author[K. L\^e]{Khoa L\^e}
\address{Department of Mathematics\\
 South Kensington Campus\\
 Imperial College London\\ 
 London, SW7 2AZ, United Kingdom}
\email{le@math.tu-berlin.de}
\subjclass[2000]{Primary 60H10; Secondary 60H05, 60L20}
\keywords{sewing lemma; Doob-Meyer decomposition; rough paths; regularization by noise; stochastic differential equations; fractional Brownian motion; additive functional; chaos expansion}

\begin{document}
\begin{abstract} We introduce a stochastic version of Gubinelli's sewing lemma (\cite{MR2091358}),
providing a sufficient condition for the convergence in moments of some random Riemann sums. Compared with the deterministic sewing lemma, adaptiveness is required and the regularity restriction is improved by a half.
The limiting process exhibits a Doob-Meyer-type decomposition. Relations with It\^o calculus are established.
To illustrate further potential applications, we use the stochastic sewing lemma in studying stochastic differential equations driven by Brownian motions or fractional Brownian motions with irregular drifts.
\end{abstract}
\maketitle
\tableofcontents
\section{Introduction} % (fold)
\label{sec:introduction}
  The sewing lemma was introduced by Gubinelli in \cite[Proposition 1]{MR2091358}. It generalizes earlier works of Young \cite{young} and Lyons \cite{MR1654527}, provides a sufficient condition ensuring the convergence of some abstract Riemann sums. This point of view was later highlighted in the work of Feyel and de La Pradelle \cite[Lemma 2.1]{MR2261056}, in which the lemma was called sewing lemma. Known for its use in deriving estimates for rough integrals (see for instance \cite[Chapter 4]{MR3289027}), the sewing lemma is one of the essential tools in Lyons' rough path theory (\cite{MR1654527}).

  The current article introduces a stochastic version of the sewing lemma, Theorem \ref{lem:Sew1} below. 
  It relaxes the regularity assumption of the original sewing lemma by a half but instead requires a certain adaptiveness of the considered increment processes.
  In a context of multidimensional parameter spaces, the sewing lemma is called reconstruction theorem and is introduced by Hairer \cite[Theorem 3.23]{MR3274562}.  
  Needless to say, the reconstruction theorem also plays a fundamental role in the theory of regularity structures. 
  However, it is not clear at the moment of writing if a stochastic reconstruction theorem is available.

  We will describe the stochastic sewing lemma in detail in Section \ref{sec:stochastic_sewing_lemma}. While its proof is reminiscent of \cite{MR2261056}, the new observation that we bring in is the use of the Doob's decomposition (\cite{MR0058896}).
  This approach naturally leads to a unique decomposition of the constructed process into a martingale and a remainder (Theorem \ref{cor:1}). 
  Such result is reminiscent of the classical Doob-Meyer decomposition, except that the remainder is not necessary a process of finite variation.
  Relations between the stochastic sewing lemma and It\^o calculus are discussed in Examples \ref{ex1}, \ref{ex2} and \ref{ex3}. 
  In these examples, we show that It\^o integrations, quadratic variations of certain martingales and It\^o formulas can be formulated and obtained by means of the stochastic sewing lemma.
  These examples suggest that the essential elements of the stochastic sewing lemma have deep connections with the foundations of stochastic analysis. Therefore, formulating these elements as a single instrumental lemma provides new insights and leads to new applications.
  To illustrate this point, we have included a few non-trivial applications, which are briefly described below.

  (i) Suppose that $\{f_t\}_{t\ge0}$ is a family of distributions with a certain negative regularity index and $X$ is a Markov process whose transition semigroup maps each $f_t$ to a bounded continuous function. 
  In Section \ref{sec:additive_functionals}, we provide a robust definition for the additive functional $\int_0^\cdot f_s(X_s)ds$ which extends the classical integration in the case $f$ is continuous.   

  (ii) We consider the stochastic differential equation
  \begin{equation}\label{intro:SDE}
    X_t^x=x+\int_0^t b(s,X_s^x)ds+W_t\,, \quad\forall t\in[0,T]\,,
  \end{equation}
  where $x\in\Rd$, $b\in[ L^\infty([0,T];C^\alpha_b(\Rd))]^d$ for some $\alpha\in(0,1)$ and $W$ is a standard Brownian motion. In \cite{MR2593276}, the authors show that the map $x\mapsto X_t^x$ is differentiable and its derivatives are H\"older continuous in the spatial variables. However, because $b$ is not differentiable, it is difficult to write down an equation for the process $Y:=\nabla X^x$ rigorously. We explain in Section \ref{sec:stochastic_flows} that $Y$ satisfies a system of Young-type differential equations. As a consequence, we show that $t\mapsto\nabla X^x_t$ is H\"older continuous for every fixed $x$.

  (iii) In Section \ref{sec:chaos_expansion}, we study weak solutions of the stochastic differential equation \eqref{intro:SDE} with drift $b\in [L^q([0,T];\C^\nu(\Rd))]^d$, for some suitable $q\in[1,\infty]$ and $\nu\in(-1,0)$. Here, $\C^\nu(\Rd)$ is the Besov-H\"older space. 
  Depending on each situation, existence and uniqueness in law of weak solutions to \eqref{intro:SDE} can be derived. 
  We will not ponder on this problem in the article, but rather refer readers to various examples in the literature, \cite{MR1988703,MR1964949,MR3652414,MR3500267,MR3785598,MR3581216}. 
  Starting from a pathwise solution $(W,X)$ defined on a complete probability space, we derive truncated Wiener chaos expansions for $\phi(X_t)$, where $\phi$ is a regular deterministic test function. Consequently, we obtain a criterion to determine if $(W,X)$ is indeed a strong solution. 
  Verifying this criterion, however, is beyond the scope of the article. 
  This result extends previous works of Krylov and Veretennikov in \cite{veretennikov1976explicit,veretennikov1981strong} who considered the cases when the drifts are measurable functions.

  (iv) The stochastic sewing lemma is also capable in situations where Markov properties are not apparent at the first sight. In Section \ref{sec:sdes_driven_by_fractional_brownian_motions}, we consider the stochastic differential equation
  \begin{equation}\label{intro:fbm}
    X_t=x+\int_0^tb(r,X_r)dr+B^H_t\,, \quad t\in[0,T]\,,
  \end{equation}
  where $x\in\Rd$, $b$ is a Borel function in $[L^q([0,T];L^p(\Rd))]^d$, $p,q\in[1,\infty]$. Here $B^H$ is a standard fractional Brownian motion with Hurst parameter $H\in(0,\frac12)$. We obtain weak existence and uniqueness in law for \eqref{intro:fbm} under the condition
  \begin{equation*}
    H\frac dp+\frac1q<\frac12\,.
  \end{equation*}
  In addition, we show that pathwise uniqueness and strong existence hold if
  \begin{equation*}
    H\frac dp+\frac1q<\frac12-H\,.
  \end{equation*}
  The former result relies on Girsanov transformation, while the later is obtained by means of the stochastic sewing lemma. The results of Section \ref{sec:sdes_driven_by_fractional_brownian_motions} extend earlier works of Nualart and Ouknine in \cite{MR1934157,MR2073441} and Ba\~nos, Nilssen and Proske in \cite{banos2015strong}.

  (v) In Section \ref{sec:averaging_along_fractional_brown}, we investigate the averaging effect of fractional Brownian motions. Namely, for a given distribution $f$ in $L^q([0,T];\C^\nu(\R)) $, $\nu\in\R$, the random field
  \begin{equation*}
    (t,x)\to\int_0^t f_r(B^H_r+x)dr
  \end{equation*}
  can be defined and has a joint-H\"older continuous (in the sense of \cite{hule2012}) version. This type of regularity plays a central role in Catellier and Gubinelli's study on path-by-path uniqueness for stochastic differential equations driven by fractional Brownian motions with distributional drifts (\cite{MR3505229}). To obtain joint-H\"older continuity properties for such random fields, the method of \cite{MR3505229} involves Fourier transforms, moment estimates and chaining arguments. Here, we obtain these properties by means of the stochastic sewing lemma and the multiparameter Garsia-Rodemich-Rumsey inequality of Hu and L\^e in \cite{hule2012}.
  In comparison with \cite{MR3505229}, our approach provides explicit regularity exponents in space and time simultaneously.

  After the appearance of the first manuscript of this article, Hairer and Li in \cite{hairer2019averaging} have used the stochastic sewing lemma introduced herein to study averaging dynamics of slow and fast systems where the slow system is driven by fractional Brownian motion with Hurst parameter $H>\frac12$. The stochastic sewing lemma is also used by Butkovsky, Dareiotis and Gerencs\'er in \cite{butkovsky2019approximation} to obtain convergence rate of the Euler-Maruyama scheme for stochastic differential equations driven by fractional Brownian motions with irregular drifts.

  We conclude the introduction with some notation which are used throughout the article.
  For every $\nu\le 0$, $\C^\nu(\Rd)$ denotes the Besov-H\"older space $\cbb^{\nu}_{\infty,\infty}(\Rd)$. 
  For each integer $k\ge0$,  $C^k_b(\Rd)$ denotes the classical space of bounded functions with bounded continuous derivatives up to order $k$. The space of all bounded uniformly continuous functions on $\Rd$ is denoted by $BUC(\Rd)$.
  Let $\css(\Rd)$ be the space of Schwartz functions on $\Rd$. The dual of $\css$ is the space of all tempered distributions $\css'(\Rd)$. 
  The notation $\lesssim $ means $\le C$ for some multiplicative non-negative constant $C$, whose value can change from one line to another.

% section introduction (end)
\section{Stochastic sewing lemma} % (fold)
\label{sec:stochastic_sewing_lemma}
  Hereafter, $d\ge1$ is a dimension, $(\Omega,\cff,\P)$ is a complete probability space, $m\ge2$ is a fixed number, $L_m$ denotes $[L^m(\Omega,\cff,\P)]^d$. 
  Let us begin with the following observation which will be employed several times. 
  Often the case, one would like to estimate moments of a sum of the form
  \begin{equation*}
    S=\sum_{i=1}^n Z_i\,,
  \end{equation*}
  where $Z_i$'s are some random variables in $L_m$.
  Without any additional structure, one at least uses triangle inequality to obtain
  \begin{equation*}
    \|S\|_{L_m}\le\sum_{i=1}^n\|Z_i\|_{L_m}\,.
  \end{equation*} 
  Indeed, this kind of estimate is used to obtain the deterministic sewing lemma (\cite{MR2261056}). 
  Suppose in addition that there is an increasing sequence of $\sigma$-algebras $\cgg_i\subset\cff$ such that $Z_1,\cdots,Z_{i-1}\in\cgg_i$ for every $i$. 
  Then, using the so-called Doob's decomposition (\cite{MR0058896}), we can write
  \begin{equation}\label{S}
    S=\sum_{i=1}^n\E^{\cgg_i}Z_i+\sum_{i=1}^n(Z_i-\E^{\cgg_i}Z_i)=:S_1+S_2\,.
  \end{equation}
  Hereafter, $\E^\cgg$ denotes the expectation conditioned on a given $\sigma$-algebra $\cgg$.
  Estimating $S_1$ by triangle inequality gives
  \begin{align}\label{est:S1}
    \|S_1\|_{L_m}\le \sum_{i=1}^n\|\E^{\cgg_i}Z_i\|_{L_m}\,.
  \end{align}
  $S_2$ is a sum of martingale differences and can be estimated using Burkholder-Davis-Gundy (BDG) inequality (\cite{MR0400380}) and Minkowski inequality, 
  \begin{align}\label{est:S2}
    \|S_2\|_{L_m}
    \le \kappa_{m,d} \lt\|\sum_{i=1}^n|Z_i-\E^{\cgg_i}Z_i|^2\rt\|_{L_{m/2}}^{\frac12}
    \le \kappa_{m,d} \(\sum_{i=1}^n\|Z_i-\E^{\cgg_i}Z_i\|^2_{L_m}\)^{\frac12}\,,
  \end{align}
  where $\kappa_{m,d}$ is the constant in BDG inequality in $L_m$.  
  Hence, we have shown that
  \begin{align}\label{est:S}
    \|S\|_{L_m}\le \sum_{i=1}^n\|\E^{\cgg_i}Z_i\|_{L_m}+ \kappa_{m,d}\(\sum_{i=1}^n\|Z_i-\E^{\cgg_i}Z_i\|^2_{L_m}\)^{\frac12}\,.
  \end{align}
  In some cases, it is more convenient to estimate the second sum on the right-hand side further by mean of triangle inequality and contraction property of conditional expectation. This yields the following inequality
  \begin{align}\label{est:SS}
    \|S\|_{L_m}\le \sum_{i=1}^n\|\E^{\cgg_i}Z_i\|_{L_m}+ 2\kappa_{m,d}\(\sum_{i=1}^n\|Z_i\|^2_{L_m}\)^{\frac12}\,.
  \end{align}
  
  The decomposition \eqref{S} and inequalities \eqref{est:S1}-\eqref{est:SS} are certainly well-known. 
  They appear, for instance, in Davie's \cite[pg. 18]{MR2377011} and in Delarue and Diel's \cite{MR3500267} in an attempt to identify the distributional drift of a diffusion. In the current article, we apply the identity \eqref{S} and inequalities \eqref{est:S1}, \eqref{est:S2} in the sewing lemma  of \cite{MR2091358,MR2261056}. This approach yields a stochastic version of the sewing lemma, Theorem \ref{lem:Sew1} below. 

  Before stating the result, let us introduce some notation which are used throughout the article. Let $\{\cff_t\}_{t\ge0}$ be a filtration on $(\Omega,\cff,\P)$. We always assume that $\cff_0$ contains $\P$-null sets, which ensures that modifications of $\{\cff_t\}$-adapted processes are still $\{\cff_t\}$-adapted. 
  Let $S,T$ be fixed non-negative numbers such that $S<T$. We denote by $[S,T]^2_\le$ the simplex $\{(s,t)\in[S,T]^2:s\le t\}$. 
  Let  $(A_{s,t})_{S\le s\le t\le T}$ be a two-parameter stochastic process with values in $\Rd$. This means that $A_{s,t}$ is a random variable in $\Rd$ for each $(s,t)$ in $[S,T]^2_\le$.
  For every $S\le s\le u\le t\le T$, we define
  \[
    \delta A_{s,u,t}=A_{s,t}-A_{s,u}-A_{u,t}.
  \]
  We say that $A$ is \textit{adapted} to $\{\cff_t\}$ if $A_{s,t}$ is $\cff_t$-measurable for every $(s,t)\in[S,T]^2_\le$; $A$ is $L_m$\textit{-integrable} if $A_{s,t}$ belongs to $L_m$ for every $(s,t)\in[S,T]^2_\le$.
  Similarly, for a one-parameter stochastic process $(\varphi_t)_{S\le t\le T}$ in $\Rd$, we say that $\varphi$ is adapted to $\{\cff_t\}$ if $\varphi_t$ is $\cff_t$-measurable for every $t\in[S,T]$; $\varphi$ is $L_m$-integrable if $\varphi_t$ belongs to $L_m$ for every $t\in[S,T]$.
  We now state our main result. Its proof will be presented later in Section \ref{sub:proofs}.
  \begin{theorem}[Stochastic sewing lemma]\label{lem:Sew1}
    Let $m\ge2$ be a real number and $(A_{s,t} )_{S\le s\le t\le T}$ be a two-parameter stochastic process with values in $\R^d$ which is $L_m$-integrable and adapted to $\{\cff_t\}$. Suppose that there are non-negative constants $\Gamma_1,\Gamma_2$ and positive constants $\varepsilon_1,\varepsilon_2$ such that 
    \begin{align}
      &\label{con:dA1}\|\E^{\cff_s}\delta A_{s,u,t} \|_{L_m}\le \Gamma_1 |t-s|^{1+\varepsilon_1} 
      \quad\textrm{for every} \quad S\le s\le u\le t\le T,
    \end{align}
    and
    \begin{equation}
      \label{con:dA2} \|\delta A_{s,u,t}-\E^{\cff_s} \delta A_{s,u,t}\|_{L_m}\le \Gamma_2|t-s|^{\frac12+\varepsilon_2}
      \quad\textrm{for every} \quad S\le s\le u\le t\le T.
    \end{equation}
    Then, there exists a unique (up to modifications) stochastic process $(\caa_t)_{S\le t\le T}$ with values in $\Rd$ satisfying the following properties
    \begin{enumerate}[wide,itemsep=2pt,labelindent=1pt,label={\upshape(\ref{lem:Sew1}\alph*)}]
      % \item \label{cl:a} the map $\caa:[S,T]\to L_m$ is continuous and $\caa_S=0$,
      \item \label{cl:a} $\caa_S=0$, $\caa$ is $\{\cff_t\}$-adapted and $L_m$-integrable,
      \item\label{cl:b} there are non-negative constants $C_1,C_2$ such that
      \begin{equation}\label{est:A1}
        \|\caa_t-\caa_s-A_{s,t}\|_{L_m}\le C_1|t-s|^{1+\varepsilon_1}+C_2|t-s|^{\frac12+\varepsilon_2}
      \end{equation}
      and
      \begin{equation}\label{est:A2}
        \|\E^{\cff_{s}}(\caa_t-\caa_s-A_{s,t})\|_{L_m}\le C_1|t-s|^{1+\varepsilon_1}
      \end{equation}
      for every $S\le s\le t\le T$.
     \end{enumerate}    
    The least constants $C_1,C_2$ are at most $ {\Gamma_1}({1-2^{-\varepsilon_1}})^{-1}$ and ${\kappa_{m,d}}\Gamma_2({1-2^{-\varepsilon_2}})^{-1}$ respectively.

    Furthermore, for every fixed $t\in[S,T]$ and any partition $\pi=\{S=t_0<t_1<\cdots<t_N=t\}$ of $[S,t]$, define the Riemann sum
    \[
      A^\pi_{t}:=\sum_{i=0}^{N-1}A_{t_i,t_{i+1}}.
    \]
    Then $\{A^\pi_t\}_\pi$ converges to $\caa_t$ in $L_m$ as the mesh size $|\pi|:=\max_i|t_{i+1}-t_i|$ goes to $0$.
  \end{theorem}
  \begin{remark*}
    Up to a modification of the constants $\Gamma_1,\Gamma_2$, the two conditions \eqref{con:dA1} and \eqref{con:dA2} are equivalent to the two conditions \eqref{con:dA1} and
    \begin{equation}\label{con:dA2'}
      \|\delta A_{s,u,t}\|_{L_m}\le \Gamma_2|t-s|^{\frac12+\varepsilon_2}
      \quad\textrm{for every}\quad S\le s\le t\le T.
    \end{equation}  
    We have favored \eqref{con:dA2} over \eqref{con:dA2'} because in this form, it is easier to derive estimates for the martingale decomposition in Theorem \ref{cor:1} (below) from Theorem \ref{lem:Sew1}.
  \end{remark*}
  We also note that no continuity assumption on the map $A:[S,T]^2_\le\to L_m$ is imposed in Theorem \ref{lem:Sew1}. In most known references on the sewing lemma, continuity of $A$ (at least on the diagonal) is usually assumed. An extension of sewing lemma without any continuity assumption on $A$ is due to Yaskov's \cite{MR3860015}. We will indeed use part of his arguments in the proof of Theorem \ref{lem:Sew1} (see Lemma \ref{lem.Yaskov} below). This allows us to drop all regularity assumption on the map $A:[S,T]^2_\le\to L_m$. 
  This fact complements the original ideas of Gubinelli in \cite{MR2091358} that the (sewing) map $A\mapsto \left((s,t)\mapsto\caa_t-\caa_s-A_{s,t}\right) $ is actually a function of $\delta A$, and hence, depends solely on the properties of $\delta A$. 
  Another illustration of this observation appears in Remark \ref{rmk.pfSSL} below, in which we explain that the integrability and adaptiveness of $A$ assumed in Theorem \ref{lem:Sew1} can be replaced by those of $\delta A$.

  With an additional assumption on $A$ (namely \eqref{con:dA3} below), it is possible to decompose $\caa$ into the sum of a martingale $\cmm$ and a remainder process $\cjj$. 
  Such a decomposition is similar to the well-known Doob-Meyer's one. However, in our case, the process $\cjj$ need not be of bounded variation. To ensure that the decomposition $\caa=\cmm+\cjj$ is unique, the bounded variation property is replaced by qualitative bounds on the increments of the process $\cjj$ centered about the process $(s,t)\mapsto\E^{\cff_s}A_{s,t}$ (see \eqref{est:J} below).
  One can also give other different characterizations of such decomposition and of $\cmm$ and $\cjj$ themselves. These findings are formalized in detail in the next theorem.
  \begin{theorem}\label{cor:1}
    Suppose that the hypotheses of Theorem \ref{lem:Sew1} holds. In addition, we assume that there are constants $\Gamma_3\ge0$ and $\varepsilon_3>0$ such that
    \begin{equation}
      \label{con:dA3}\|\E^{\cff_s}A_{u,t} -\E^{\cff_u} A_{u,t} \|_{L_m}\le \Gamma_3 |t-s|^{\frac12+\varepsilon_3} \quad\textrm{for every}\quad S\le s\le u\le t\le T.
    \end{equation}
    Then, there exist stochastic processes $\cmm,\cjj$ and non-negative constants $C_1,C_2,C_3$  satisfying the following properties
    \begin{enumerate}[wide,itemsep=2pt,labelindent=1pt,label={\upshape(\ref{cor:1}\alph*)}]
      \item\label{cl:amj} $\cmm,\cjj$ are $
      \{\cff_t\}$-adapted, $L_m$-integrable and $\caa_t=\cmm_t+\cjj_t$ a.s. for every $t\in[S,T]$,
      \item\label{cl:m} $(\cmm_s)_{S\le s\le T}$ is an $\{\cff_t\}$-martingale with $\cmm_S=0$,
      \item \label{cl:est.m} for any $S\le s\le t\le T$,
      \begin{equation}
        \label{est:M} \| \cmm_t-\cmm_s-A_{s,t}+\E^{\cff_{s}}A_{s,t}\|_{L_m}\le C_2|t-s|^{\frac12+\varepsilon_2}+C_3|t-s|^{\frac12+\varepsilon_3}\,,
      \end{equation}
      \item\label{cl:est.j} for any $S\le s\le t\le T$,
        \begin{align}
          \label{est:J} \| \cjj_t-\cjj_s-\E^{\cff_{s}} A_{s,t}\|_{L_m}\le C_1|t-s|^{1+\varepsilon_1}+C_3|t-s|^{\frac12+\varepsilon_3}\,,
        \end{align}
      \item \label{cl:est.j'} for any $S\le s\le t\le T$,
      \begin{equation}
        \label{est:J'} \| \E^{\cff_{s}}(\cjj_t-\cjj_s-A_{s,t})\|_{L_m}\le C_1|t-s|^{1+\varepsilon_1}\,.
      \end{equation}
     \end{enumerate} 
    Given $A$, we have the following characterizations.
    \begin{enumerate}[resume*]% [(\ref{cor:1}a),resume]
      \item\label{ch1} The pair of processes $(\cmm,\cjj)$ is characterized uniquely by the set of properties \ref{cl:amj}, \ref{cl:m}, \ref{cl:est.m} or, alternatively by the set of properties \ref{cl:amj}, \ref{cl:m}, \ref{cl:est.j}. 
      \item\label{ch2} The process $\cmm$ is characterized uniquely by \ref{cl:m} and \ref{cl:est.m}. 
      \item\label{ch3} The process $\cjj$ is characterized uniquely by \ref{cl:est.j}, \ref{cl:est.j'} and the fact that $\cjj$ is $\{\cff_t\}$-adapted and $\cjj_S=0$.
    \end{enumerate} 
    
    The least constants $C_1,C_2,C_3$ are at most $ {\Gamma_1}({1-2^{-\varepsilon_1}})^{-1}$, ${ \kappa_{m,d}\Gamma_2}({1-2^{-\varepsilon_2}})^{-1}$ and ${\kappa_{m,d}\Gamma_3}({1-2^{-\varepsilon_3}})^{-1}$ respectively.

    Furthermore, for every fixed $t\in[S,T]$ and any partition $\pi=\{S=t_0<t_1<\cdots<t_N=t\}$ of $[S,t]$, define the Riemann sums
    \begin{equation*}
        M^\pi_{t}:=\sum_{i=0}^{N-1}\(A_{t_i,t_{i+1}}-\E^{\cff_{t_i}} A_{t_i,t_{i+1}}\)
        \quad\textrm{and}\quad
        J^\pi_{t}:=\sum_{i=0}^{N-1}\E^{\cff_{t_i}}A_{t_i,t_{i+1}}\,.
    \end{equation*}
    Then $\{M^\pi_{t}\}_\pi$ and $\{J^\pi_{t}\}_\pi $  
    converge to $\cmm_t$ and $\cjj_t$ respectively in $L_m$ as $|\pi|$ goes to 0.
  \end{theorem}
  The above result can be regarded as a consequence of Theorem \ref{lem:Sew1}. Indeed, we define two-parameter processes $J$, $M$ by
  \begin{equation}\label{def.JM}
    J_{s,t}=\E^{\cff_s}A_{s,t}
    \quad\textrm{and}\quad
    M_{s,t}=A_{s,t}-\E^{\cff_s}A_{s,t}\,.
  \end{equation}
  It is evident that $J$ and $M$ are $\{\cff_t\}$-adapted and $L_m$-integrable. 
  In addition, for every $s\le u\le t$, we have
  \begin{align*}
    &\delta J_{s,u,t}=\E^{\cff_s}\delta A_{s,u,t}+\left(\E^{\cff_s}A_{u,t}-\E^{\cff_u}A_{u,t}\right),
    \\&\delta M_{s,u,t}=\left(\delta A_{s,u,t}-\E^{\cff_s}\delta A_{s,u,t}\right)- \left(\E^{\cff_s}A_{u,t}-\E^{\cff_u}A_{u,t}\right),
  \end{align*}
  and hence 
  \begin{equation*}
    \E^{\cff_s}\delta J_{s,u,t}=\E^{\cff_s}\delta A_{s,u,t}\,,
    \quad
    \delta J_{s,u,t}-\E^{\cff_s}\delta J_{s,u,t}=\E^{\cff_s}A_{u,t}-\E^{\cff_u}A_{u,t},
  \end{equation*}
  and 
  \begin{equation}\label{id.EM}
    \E^{\cff_s}\delta M_{s,u,t}=0.
  \end{equation}
  The conditions \eqref{con:dA1}, \eqref{con:dA2} and \eqref{con:dA3} on $A$ imply that
  \begin{equation}\label{est.EdJ}
    \|\E^{\cff_s}\delta J_{s,u,t}\|_{L_m}\le \Gamma_1|t-s|^{1+\varepsilon_1},
  \end{equation}
  \begin{equation}\label{est.dJ}
    \|\delta J_{s,u,t}-\E^{\cff_s}\delta J_{s,u,t}\|_{L_m}\le \Gamma_3|t-s|^{\frac12+\varepsilon_3},
  \end{equation}
  \begin{equation}\label{est.dMfake}
    \|\delta M_{s,u,t}-\E^{\cff_s}\delta M_{s,u,t}\|_{L_m}\le \Gamma_{2,3}|t-s|^{\frac12+\varepsilon_2\wedge \varepsilon_3},
  \end{equation}
  for some constant $\Gamma_{2,3}$ which depends on $\Gamma_2,\Gamma_3,S,T$.
  From here, by applying Theorem \ref{lem:Sew1} to the two-parameter processes $J$ and $M$ respectively, the processes $\cjj$ and $\cmm$ can be constructed which are $\{\cff_t\}$-adapted and $L_m$-integrable. From \eqref{est:A1}, \eqref{est:A2} and \eqref{est.EdJ}, \eqref{est.dJ} we can derive the estimates \eqref{est:J}, \eqref{est:J'}. From \eqref{est:A1}, \eqref{est:A2} and \eqref{id.EM}, \eqref{est.dMfake},  we can derive the following estimates 
  \begin{equation}\label{tmp.M2}
    \|\cmm_t-\cmm_s-A_{s,t}+\E^{\cff_s}A_{s,t}\|_{L_m}\le C|t-s|^{\frac12+ \varepsilon_2\wedge \varepsilon_3}.
  \end{equation}
  and
  \begin{equation}\label{tmp.M1}
    \|\E^{\cff_s}(\cmm_t-\cmm_s-A_{s,t}+\E^{\cff_s}A_{s,t} )\|_{L_m}\le0
  \end{equation}
  for every $s\le t$. Obviously, $\E^{\cff_s}(A_{s,t}-\E^{\cff_s}A_{s,t} )=0$, so \eqref{tmp.M1} implies that $\cmm$ is an $\{\cff_t\}$-martingale. 
  % To see that $\caa=\cmm+\cjj$, we define $\bar\caa=\cmm+\cjj$.
  Unfortunately, the estimate \eqref{tmp.M2} does not imply \eqref{est:M} with the optimal constants stated in Theorem \ref{lem:Sew2}. 
  This is due to the use of \eqref{est.dMfake}. A better estimate for $\delta M$ which should be employed here is the following
  \begin{equation}\label{est.dM}
    \|\delta M_{s,u,t}-\E^{\cff_s}\delta M_{s,u,t}\|_{L_m}\le \Gamma_2|t-s|^{\frac12+\varepsilon_2}+\Gamma_3|t-s|^{\frac12+\varepsilon_3}.
  \end{equation}
  However, the right-hand sides of \eqref{est.dM} and \eqref{con:dA2} are of different forms, preventing us from applying Theorem \ref{lem:Sew1} directly. This is a minor issue and will be resolved in Section \ref{sub:proofs} once we have better understanding on the proof of Theorem \ref{lem:Sew1}.

  In most applications considered herein, it is more convenient to verified a simpler condition than \eqref{con:dA2}. 
  This is described in the next result. 
  \begin{theorem}\label{lem:Sew2}
    Let $m\ge2$ be a real number and $(A_{s,t} )_{S\le s\le t\le T}$ be a stochastic process with values in $\R^d$ which is $L_m$-integrable and adapted to $\{\cff_t\}$. Assume that there are constants $\Gamma_1,\Gamma_4\ge0$ and $\varepsilon_1,\varepsilon_4>0$ such that  \eqref{con:dA1} holds and
    \begin{equation}\label{con:A}
      \|A_{s,t}\|_{L_m}\le \Gamma_4 |t-s|^{\frac12+\varepsilon_4}\quad\textrm{for every } (s,t)\in[S,T]^2_\le.
    \end{equation}
    Then, there exists a unique (up to modifications) stochastic process $(\caa_t)_{S\le t\le T}$ with values in $\Rd$  satisfying the following properties
    \begin{enumerate}[wide,itemsep=2pt,labelindent=0pt,label={\upshape(\ref{lem:Sew2}\alph*)}]
      % \item\label{25a} the map $\caa:[S,T]\to L_m$ is continuous and $\caa_S=0$,
      \item\label{25a} $\caa_S=0$ and for every $t\in[S,T]$, $\caa_t$ is $\cff_t$-measurable and $L_m$-integrable,
      \item\label{25b} there are non-negative constants $C_1,C_4$ such that for every $(s,t)\in[S,T]^2_\le $, $\caa$ satisfies the inequality
      \begin{equation}\label{est:A1'}
        \|\caa_t-\caa_s\|_{L_m}\le C_1|t-s|^{1+\varepsilon_1}+C_4|t-s|^{\frac12+\varepsilon_4} 
        % \quad\forall(s,t)\in[S,T]^2_\le
      \end{equation}
      % for every $(s,t)\in[S,T]^2_\le$
      and inequality \eqref{est:A2}.
    \end{enumerate}
    The least constants $C_1,C_4$ are at most $\Gamma_1(1-2^{-\varepsilon_1})^{-1}$ and $\Gamma_4(6 \kappa_{m,d}(1-2^{-\varepsilon_4})^{-1}+1)$ respectively.
  \end{theorem}
  \begin{proof}
    This result is a direct application of Theorem \ref{lem:Sew1}.
    We note that condition \eqref{con:A} implies condition \eqref{con:dA2} with $\Gamma_2=6 \Gamma_4$ and $\varepsilon_2=\varepsilon_4$. Hence, Theorem \ref{lem:Sew1} is applied to construct the process $\caa$ satisfying \ref{cl:a} and \ref{cl:b}. The estimate \eqref{est:A1} and condition \eqref{con:A} imply \eqref{est:A1'}.
  \end{proof}
  \begin{remark}\label{rmk.0m}
    We note that if $A$ satisfies the hypotheses of Theorem \ref{lem:Sew2}, then it also satisfies the hypotheses of Theorem \ref{cor:1}. This is because condition \eqref{con:A} also implies condition \eqref{con:dA3} with $\Gamma_3=2 \Gamma_4$ and $\varepsilon_3=\varepsilon_4$. Consequently, the process $\caa$ in Theorem \ref{lem:Sew2} is uniquely decomposed into $\cmm+\cjj$ where $\cmm,\cjj$ are the processes in Theorem \ref{cor:1}. 
    It turns out that the martingale part of $\caa$, namely $\cmm$, has to be $0$. This can be seen by directly checking that $\cmm\equiv0$ satisfies \eqref{est:M} and hence, by uniqueness, property \ref{ch2}, the martingale part of $\caa$ has to vanish. 
    Consequently, for each $t\in[S,T]$, the Riemann sums $\{A^\pi_t\}_\pi$ and $\{J^\pi_t\}_\pi$ both converge to $\caa_t$ in $L_m$ while $\{M^\pi_t\}_\pi$ converges to $0$ in $L_m$.
  \end{remark}
  
    In Theorem \ref{lem:Sew1} and its proof (in Section \ref{sub:proofs}), it is seen that for each $t\in[S,T]$, $\caa_t$ is the limit in $L_m$ of certain Riemann sums. Conversely, starting from the estimates \eqref{est:A1} and \eqref{est:A2}, one can show the convergence of Riemann sums in $L_m$, which is the content of the following result. The deterministic counterpart of this result can be found, for instance, in \cite[Corollary 2.4]{MR2261056} and \cite[Corollary 1]{MR2091358}.
  \begin{proposition}\label{prop.rie}
    Let $m\ge2$ be a real number, $(\varphi_t)_{S\le t\le T}$ and $(\mu_{s,t})_{S\le s\le t\le T}$ be $\{\cff_t\}$-adapted and $L_m$-integrable processes with values in $\Rd$. Suppose that there are constants $C\ge0$ and $\varepsilon>0$ such that
    \begin{equation}\label{con.phi2}
      \|\varphi_t- \varphi_s- \mu_{s,t}\|_{L_m}\le C|t-s|^{\frac12+\varepsilon}
    \end{equation}
    and
    \begin{equation}\label{con.phi1}
      \|\E^{\cff_s}(\varphi_t- \varphi_s- \mu_{s,t})\|_{L_m}\le C|t-s|^{1+\varepsilon}
    \end{equation}
    for every $(s,t)\in[S,T]^2_\le$. 
    For every fixed $(s,t)\in[S,T]^2_\le$ and any partition $\pi=\{s=t_0<t_1<\cdots<t_N=t\}$ of $[s,t]$, define the Riemann sum
    \begin{equation*}
      \mu^\pi_{s,t}:=\sum_{i=0}^{N-1}\mu_{t_i,t_{i+1}}\,.
    \end{equation*}
    Then $\{\mu^\pi_{s,t}\}_\pi$ converges to $ \varphi_t- \varphi_s$ in $L_m$ as the mesh size $|\pi|$ goes to 0.
  \end{proposition}
  \begin{proof}
    We write
    \begin{align*}
      \varphi_t-\varphi_s-\mu^\pi_{s,t}
      =\sum_i\(\varphi_{t_{i+1}}-\varphi_{t_i}-\mu_{t_i,t_{i+1}}\)
    \end{align*}
    and apply inequality \eqref{est:SS} to obtain that
    \begin{align*}
      \|\varphi_t-\varphi_s-\mu^\pi_{s,t}\|_{L_m}&\le\sum_i\|\E^{\cff_{t_i}}(\varphi_{t_{i+1}}-\varphi_{t_i}-\mu_{t_i,t_{i+1}})\|_{L_m}
      \\&\quad+2 \kappa_{m,d} \(\sum_i\|\varphi_{t_{i+1}}-\varphi_{t_i}-\mu_{t_i,t_{i+1}}\|^2_{L_m}\)^{\frac12}\,.
    \end{align*}
    In conjunction with \eqref{con.phi2} and \eqref{con.phi1}, the above inequality implies that $$\|\varphi_t-\varphi_s-\mu^\pi_{s,t}\|_{L_m}\lesssim |\pi|^{\varepsilon}.$$
    This means that $\lim_{|\pi|\to0}\mu^\pi_{s,t}=\varphi_t-\varphi_s$ in $L_m$.
  \end{proof}
  For later purposes, it is convenient to view the resulting processes in Theorems \ref{lem:Sew1}, \ref{cor:1} and \ref{lem:Sew2} as operators whose input is the increment process $A$. This leads to the following convention.
  \begin{definition}\label{def.I}
    Let $t$ be in $[S,T]$ and let $\pi=\{S=t_0<\cdots<t_N=T\}$ denote a partition of $[S,t]$. Let $A^\pi_t, M^\pi_t$ and $J^\pi_t$ be the Riemann sums defined in Theorems \ref{lem:Sew1} and \ref{cor:1}. We define $\cii_t[A], \cmm_t[A]$ and $\cjj_t[A]$ respectively by $\lim_{|\pi|\downarrow0}A^\pi_t$, $\lim_{|\pi|\downarrow0}M^\pi_t$ and $\lim_{|\pi|\downarrow0}J^\pi_t$ whenever these limits exists in probability.
    % In the contexts of Theorems \ref{lem:Sew1}, \ref{cor:1} and \ref{lem:Sew2}, we denote $\cii[A]=\caa$, $\cmm[A]=\cmm$ and $\cjj[A]=\cjj$.
  \end{definition}
  Theorems \ref{lem:Sew1}, \ref{cor:1} and \ref{lem:Sew2} provides sufficient conditions for the well-posedness of $\cii[A],\cmm[A]$ and $\cjj[A]$. 
  In addition, from the discussion succeeding Theorem \ref{cor:1}, we see that under hypotheses of Theorem \ref{cor:1}, $\cjj[A],\cmm[A]$ and $\cii[A]$ are related by
  \begin{equation*}
    \cjj[A]=\cjj=\cii\lt[(s,t)\mapsto\E^{\cff_s}A_{s,t} \rt]
  \end{equation*}
  and
  \[  
    \cmm[A] =\cmm=\cii\lt[(s,t)\mapsto A_{s,t}- \E^{\cff_s}A_{s,t} \rt].
  \]
  It is also possible to write $\cjj,\cmm$ as actions of $\cii$ on $\caa$. To see this, we need the following observation.
  \begin{proposition}\label{prop:AA}
    Let $m\ge2$ be a real number and $(A)_{S\le s\le t\le T}$ be a two-parameter process in $\Rd$ which is $\{\cff_t\}$-adapted and $L_m$-integrable. Suppose that there are non-negative constants $\Gamma_4,\Gamma_5$ and  positive constants $\varepsilon_4,\varepsilon_5$ such that \eqref{con:A} holds and
    \begin{equation}\label{con:AA2}
      \|\E^{\cff_s} A_{s,t}\|_{L_m}\le C_5|t-s|^{1+\varepsilon_5}
    \end{equation}
    for every $(s,t)$ in $[S,T]^2_\le$. Then $A$ satisfies all conditions of Theorem \ref{lem:Sew1} and $\cii[ A]=0$.
  \end{proposition}
  \begin{proof}
    % It is straightforward to verify that $ A$ satisfies conditions \eqref{con:dA1} and \eqref{con:dA2}. Hence, $\cii[ A]$ is well-defined.
    Using \eqref{con:A} and \eqref{con:AA2}, we can directly verify \eqref{con:dA1}, \eqref{con:dA2} and that the zero process satisfies \ref{cl:a} and \ref{cl:b}. By uniqueness of Theorem \ref{lem:Sew1}, this means $\cii_t[ A]=0$ a.s. for every $t\in[S,T]$. 
  \end{proof}
  \begin{corollary}\label{cor:AA}
    Suppose that the hypotheses of Theorem \ref{cor:1} are satisfied.
    Let $\caa,\cmm,\cjj$ be the processes in Theorem \ref{cor:1}. For every $s,t\in[S,T]$, define
    \begin{equation*}
      A^{(1)}_{s,t}=\E^{\cff_s}(\caa_t-\caa_s)
      \quad\textrm{and}\quad
      A^{(2)}_{s,t}=\caa_t-\caa_s-\E^{\cff_s}(\caa_t-\caa_s)\,.
    \end{equation*}
    % \KL{this result is not used at the moment}
    Then $A^{(1)}, A^{(2)}$ satisfy the hypotheses of Theorem \ref{lem:Sew1} and $\cii[A^{(1)}]=\cjj=\cjj[A]$ and $\cii[A^{(2)}]=\cmm=\cmm[A]$.
  \end{corollary}
  \begin{proof}
    Let $J,M$ be the processes defined in \eqref{def.JM}.
    This result amounts to verifying conditions \eqref{con:A} and \eqref{con:AA2} of Proposition \ref{prop:AA} for $A^{(1)}-J$ and $A^{(2)}-M$. In both cases, \eqref{con:A} and \eqref{con:AA2} are consequences of \eqref{est:A1} and \eqref{est:A2}.
  \end{proof}

  We give an example hinting that the restriction $m\ge2$ is necessary in Theorem \ref{lem:Sew1}.
  \begin{example}\label{ex.Poi}
    Let $(N_t,\cff_t)_{t\ge0}$ be a Poisson point process on $\R$ with intensity $\lambda>0$ (see \cite[Definition 3.3]{MR1121940}).  It is well-known that $\mathrm{Var\,}(N_t-N_s)=\E(N_t-N_s)=\lambda(t-s)$ and that $\(N_t- \lambda t,\cff_t\)_{t\ge0}$ is a square integrable martingale. Define $A_{s,t}=N_t-N_s- \lambda(t-s)$.
    Then $\delta A_{s,u,t}=0$ and hence we can apply Theorem \ref{lem:Sew1}, with $m=2$ for simplicity, to $A$ to find a unique process $\caa$ satisfying \ref{cl:a} and \ref{cl:b}. Since $\delta A=0$, we find that $C_1$ and $C_2$ can be taken to be $0$ in \eqref{est:A1} and \eqref{est:A2}. This implies that $\caa_t=N_t- \lambda t$ for every $t$.

    On the other hand, we also have
    \begin{equation}\label{tmp.346}
      \E^{\cff_s}(A_{s,t})=0
      \quad\textrm{and}\quad
      \|A_{s,t}\|_{L_n}\le (2\lambda)^{\frac1n}(t-s)^{\frac1n}
    \end{equation}
    for every $s\le t$ and $n\in[1,2]$. In the above, the identity is obvious while the inequality is derived by interpolating
    \[
      \E|A_{s,t}|^2=\lambda(t-s)
      \quad\textrm{and}\quad
      \E|A_{s,t}|\le 2 \lambda(t-s).
    \] 
    Note that when $n\in[1,2)$, $\frac1n>\frac12$. 
    This means that Theorem \ref{lem:Sew1} does not hold for $m=n\in[1,2)$. For if it were true, \eqref{tmp.346}  and Proposition \ref{prop:AA} (a consequence of Theorem \ref{lem:Sew1}) would imply that $\caa\equiv0$. In other words, we would have $N_t- \lambda t=0$ for all $t$, which is a contradiction.
  \end{example}

  \subsection{Relations with It\^o calculus} % (fold)
  \label{sub:relations_with_ito}
    
  % subsection relations_with_ito (end)
  To see that the conditions \eqref{con:dA1}, \eqref{con:dA2} and \eqref{con:dA3} are natural, let us see how the stochastic sewing lemma is related to It\^o calculus. This is established through the following examples.
  \begin{example}\label{ex1}
    Let $B$ be a standard Brownian motion in $\Rd$ with respect to a filtration $\{\cff_t\}$. Assume that $\{\cff_0\}$ contains $\P$-null sets.
    We wish to define the It\^o integral $\int_0^T f(B_s)\otimes dB_s$, where $f:\R^d\to\R^d$ is a H\"older continuous function with exponent $\tau\in(0,1]$. We define $A_{s,t}=f(B_s)\otimes (B_t-B_s)$. Then for every $s\le u\le t$,
    \begin{equation*}
      \delta A_{s,u,t}=-[f(B_u)-f(B_s)]\otimes(B_t-B_u)\,.
    \end{equation*}
    It follows that for every $m\ge2$,
    \begin{equation*}
      \|\delta A_{s,u,t}\|_{L_m}\le \|f\|_{C^\tau}|t-s|^{\frac12+\frac{\tau}2}
      \quad\textrm{and}\quad
      \E^{\cff_s}\delta A_{s,u,t}=0\,.
    \end{equation*}
    In other words, $A$ satisfies conditions \eqref{con:dA1} and \eqref{con:dA2}, respectively with $\Gamma_1=0$, $\Gamma_2=\|f\|_{C^\tau}$ and $\varepsilon_2=\frac \tau2$. By Theorem \ref{lem:Sew1}, we can define
    \begin{equation*}
      \int_0^Tf(B_s)\otimes dB_s:=\lim_{\max_i|t_{i+1}-t_i|\downarrow0} \sum_i f(B_{t_i})\otimes (B_{t_{i+1}}-B_{t_i})
      \quad\textrm{in} \quad L_m,
    \end{equation*}
    where $\{t_i\}$ is any partition of $[0,T]$. 
    The estimates \eqref{est:A1} and \eqref{est:A2}, respectively, become
    \begin{align*}
      \|\int_s^tf(B_r)\otimes dB_r-f(B_s)\otimes(B_t-B_s)\|_{L_m}\le  \kappa_{m,d}(1-2^{-\frac \tau2})^{-1}\|f\|_{C^\tau}|t-s|^{\frac12+\frac \tau2}
    \end{align*}
    and
    \begin{align*}
      \|\E^{\cff_s}\int_s^tf(B_r)\otimes dB_r-\E^{\cff_s}f(B_s)\otimes(B_t-B_s)\|_{L_m}\le 0\,.
    \end{align*}
    Since $\E^{\cff_s}f(B_s)\otimes(B_t-B_s)=0$, the previous estimate implies that $\int_0^\cdot f(B_r)\otimes dB_r$ is a martingale.

    Alternatively, we can also see that $\int_0^\cdot f(B_r)\otimes dB_r$ is a martingale from Theorem \ref{cor:1}. Indeed, we note that $\E^{\cff_s}A_{s,t}=0$ for every $s\le t$. In particular, \eqref{con:dA3} is satisfied with $\Gamma_3=0$. By Theorem \ref{cor:1}, we have the decomposition $\int_0^tf(B_s)\otimes dB_s=\cmm_t+\cjj_t $. We observe that the process $\cjj\equiv0$ satisfies \eqref{est:J}. Hence, by \ref{ch1}, $\int_0^\cdot f(B_s)\otimes dB_s=\cmm$, which is a martingale. 

    Lastly, we point out an interesting implication of the uniqueness part of Theorem \ref{lem:Sew1}. The It\^o integral $\int_0^\cdot f(B)\otimes dB$ is the unique $\{\cff_t\}$-adapted process $\varphi:[0,T]\to L_m$ such that $\varphi_0=0$ and
    \begin{gather}
      \|\varphi_t-\varphi_s-f(B_s)\otimes(B_t-B_s)\|_{L_m}\lesssim|t-s|^{\frac12+\varepsilon}\,,
      \label{c1}\\\|\E^{\cff_s}(\varphi_t-\varphi_s)\|_{L_m}\lesssim|t-s|^{1+\varepsilon}
      \label{c2}
    \end{gather}
    for every $s\le t$ for some $\varepsilon>0$.
  \end{example} 
  \begin{example}[Quadratic variation]\label{ex2} Let $M$ be a $L_4$-integrable $\{\cff_t\}$-martingale in $\Rd$. Assume that $\cff_0$ contains $\P$-null sets and $M$ satisfies 
  \begin{equation}\label{con.regM}
    \|M_{s,u}\otimes M_{u,t}\|_{L_2}\le C|t-s|^{\frac12+\varepsilon}\quad\forall s\le u\le t
  \end{equation}
  for some constants $\varepsilon>0$ and $C>0$. 
  Here, we adopt the notation $M_{s,t}=M_t-M_s$.
  We consider $A_{s,t}=M_{s,t}\otimes M_{s,t} $, which is a random element in $\Rd\otimes\Rd$ for every fixed $s\le t$. It is straightforward to verify that for every $s\le u\le t$
  \begin{equation*}
    \delta A_{s,u,t}=M_{s,u}\otimes M_{u,t}+M_{u,t}\otimes M_{s,u}\,.
  \end{equation*}
  It follows that $\|\delta A_{s,u,t}\|_{L_2}\le 2C|t-s|^{\frac12+\varepsilon} $ and $\E^{\cff_s}\delta A_{s,u,t}=0 $. By Theorem \ref{lem:Sew1}, there exists a unique adapted process, denoted by $\bra{M}:[0,T]\to [L^2(\Omega)]^{d\times d}$, such that  $\bra{M}_0=0$,
  \begin{equation}\label{eqn:quadM}
    \|\bra{M}_{s,t}-M_{s,t}\otimes M_{s,t}\|_{L_2}\lesssim |t-s|^{\frac12+\varepsilon}
    \quad\textrm{and}\quad
    \E^{\cff_s}\bra{M}_{s,t} =\E^{\cff_s}\(M_{s,t}\otimes M_{s,t}\)
  \end{equation}
  for every $s\le t$. Here, we use the notation $\bra{M}_{s,t}=\bra{M}_t-\bra{M}_s$.
  In addition, we have
  \begin{equation*}
    \bra{M}_t=\lim_{\max_i|t_{i+1}-t_i|\downarrow0}\sum_i M_{t_i,t_{i+1}}\otimes M_{t_i,t_{i+1}}
    \quad\textrm{in}\quad L_2
  \end{equation*}
  for every partition $\{t_i\}$ of $[0,t]$ which shows that $\bra{M}$ (defined via Theorem \ref{lem:Sew1} as above) is indeed the quadratic variation of $M$ (see for instance \cite[Chapter 1, (2.3)]{MR1725357}).

  Observing that $\E^{\cff_s}(M_{s,t}\otimes M_{s,t})=\E^{\cff_s}(M_t\otimes M_t-M_s\otimes M_s)$, the identity in \eqref{eqn:quadM} implies that $M_\cdot\otimes M_\cdot-\bra{M}_\cdot$ is a martingale.
  
  Finally, for $L_4$-integrable martingales satisfying \eqref{con.regM}, the quadratic variation $\bra{M}$ is the unique adapted process $\varphi:[0,T]\to [L^2(\Omega)]^{d\times d}$ such that
  \begin{gather}
    \|\varphi_t- \varphi_s-M_{s,t}\otimes M_{s,t}\|_{L_2}\lesssim |t-s|^{\frac12+\varepsilon}\,,\label{q1}
    \\\|\E^{\cff_s}(\varphi_t- \varphi_s-M_{s,t}\otimes M_{s,t})\|_{L_2}\lesssim |t-s|^{1+\varepsilon}\label{q2}
  \end{gather}
  for every $s\le t$ and for some $\varepsilon>0$. 

  We illustrate the usefulness of the previous characterization in computing the quadratic variations by considering two specific cases of Brownian motion and compensated Poisson process.

  In the case when $M=B$ is a standard Brownian motion in $\Rd$, we have $\|B_{s,t}\otimes B_{s,t}\|_{L_2}\lesssim\|B_{s,t}\|_{L_4}\lesssim |t-s|$ and $\E^{\cff_s}(B_{s,t}\otimes B_{s,t})=(t-s)I_d$, where $I_d$ is the $d\times d$-identity matrix. Hence, from \eqref{q1} and \eqref{q2}, the quadratic variation of $B$ is the unique adapted process $\varphi:[0,T]\to [L^2(\Omega)]^{d\times d}$ such that
  \[
    \|\varphi_t- \varphi_s\|_{L_2}\lesssim |t-s|^{\frac12+\varepsilon}
    \quad\textrm{and}\quad
    \|\E^{\cff_s}(\varphi_t- \varphi_s-(t-s)I_d)\|_{L_2}\lesssim |t-s|^{1+\varepsilon}
  \]
  for every $s\le t$ and for some $\varepsilon>0$. It is evident that the process  $\varphi_t=tI_d$ satisfies the above inequalities. By uniqueness, $\bra{B}_t=tI_d$. 
  
  Consider the case when $M=\bar N$ is a compensated Poisson point process on $\R$ with intensity $\lambda$. This means that $\bar N_t=N_t- \lambda t$ where $(N_t,\cff_t)_{t\ge0}$ is a Poisson point process on $\R$ with intensity $\lambda>0$ (see \cite[Definition 3.3]{MR1121940}). By direct calculations, we have $\E(\bar N_{s,u}\bar N_{u,t})^2=\lambda^2(u-s)(t-u)$, $\E^{\cff_s}(\bar N_{s,t}\bar N_{s,t})=\lambda(t-s)$ and $\E(N_{s,t}-\bar N_{s,t}\bar N_{s,t})^2=2\lambda^2(t-s)^2$.
  We see that condition \eqref{con.regM} is satisfied and hence, from \eqref{q1} and \eqref{q2}, the quadratic variation of $\bar N$ is the unique adapted process $\varphi:[0,T]\to L^2(\Omega)$ such that
  \[
    \|\varphi_t- \varphi_s- N_{s,t}\|_{L_2}\lesssim |t-s|^{\frac12+\varepsilon}
    \quad\textrm{and}\quad
    \|\E^{\cff_s}(\varphi_t- \varphi_s- \lambda(t-s))\|_{L_2}\lesssim |t-s|^{1+\varepsilon}
  \]
  for every $s\le t$ and for some $\varepsilon>0$. It is evident that the process  $\varphi_t=N_t$ satisfies the above inequalities. By uniqueness, $\bra{\bar N}_t=N_t$. 
  \end{example}
  \begin{example}[It\^o formula] \label{ex3}Let $M$ be a martingale such that
  \begin{equation}\label{con.M6}
    \|M_{s,t}\|_{L_6}\le C|t-s|^{\frac13+\varepsilon} \quad\forall (s,t)\in[0,1]^2_\le
  \end{equation}
  for some constants $\varepsilon>0$ and $C\ge0$. We recall the notation $M_{s,t}=M_t-M_s$ used in the previous example. From Cauchy-Schwarz inequality, it is easy to see that \eqref{con.M6} implies condition \eqref{con.regM}. Hence the quadratic variation of $M$, denoted by $[M]$, is well-defined by Example \ref{ex2} and satisfies
  \begin{equation}\label{est.Mbra}
    \|[M]_{s,t}\|_{L_2}\lesssim|t-s|^{\frac23+2 \varepsilon}.
  \end{equation}
  Let $f$ be a function in $C^3_b(\Rd)$. Let us explain how to obtain the following It\^o's formula, which is well-known, by stochastic sewing lemma
    \begin{equation*}
      f(M_1)-f(M_0)=\int_0^1\wei{\nabla f(M_s),dM_s}+\frac12\int_0^1 \wei{\nabla^2 f(M_s),d\bra{M}_s}\,.
    \end{equation*}
  In the right-hand side above, $\wei{\cdot,\cdot}$ denotes the inner product, the former integral is an It\^o integral and the later integral is a Young integral in $L^1(\Omega)$.

  We consider $A_{s,t}=f(M_t)-f(M_s)$. It is obvious that $f(M_1)-f(M_0)=\cii_1[A]$, where $\cii$ is defined in Definition \ref{def.I}. By Taylor expansion, we have
    \begin{align*}
      A_{s,t}&=\wei{\nabla f(M_s),M_{s,t}}+\frac12\wei{\nabla^2f(M_s),M_{s,t}\otimes M_{s,t}}+\coo(|M_{s,t}|^3)
      \\&=: A^{(1)}_{s,t}+A^{(2)}_{s,t}+A^{(3)}_{s,t}\,.
    \end{align*}
    where $\coo(h)$ is some quantity such that $\|\coo(h)\|_{L_2}\lesssim h$. Obviously, $\cii[A]=\cii[A^{(1)}]+\cii[A^{(2)}]+\cii[A^{(3)}]$ whenever each term on the right-hand side exists, hence, it suffices to show that $\cii[A^{(1)}]=\int_0^\cdot \wei{\nabla f(M_s),dM_s}$, $\cii[A^{(2)}]=\frac12\int_0^\cdot\wei{\nabla^2 f(M_s),d[M]_s}$ and $\cii[A^{(3)}]=0$.
    Reasoning as in Example \eqref{ex1}, it is straightforward to see that $\cii[A^{(1)}]=\int_0^\cdot \wei{\nabla f(M_s),dM_s}$. To compute $\cii[A^{(2)}]$, we write
    \begin{align*}
      A^{(2)}_{s,t}=\frac12\wei{\nabla^2f(M_s),[M]_{s,t}}+\frac12\wei{\nabla^2f(M_s),[M]_{s,t}-M_{s,t}\otimes M_{s,t}}=:A^{(4)}_{s,t}+A^{(5)}_{s,t}.
    \end{align*}
    From $\delta A^{(4)}_{s,u,t}=-\frac12\wei{\nabla^2f(M_u)-\nabla^2f(M_s),[M]_{u,t}}$, applying Cauchy-Schwarz inequality, \eqref{con.M6} and \eqref{est.Mbra}, we obtain
    \[
      \|\delta A^{(4)}_{s,u,t}\|_{L_1}\lesssim|\nabla^3f|_\infty\|M_{s,u}\|_{L_2}\|[M]_{u,t}\|_{L_2}
      \lesssim|\nabla^3f|_\infty|t-s|^{1+3 \varepsilon}.
    \]
    Applying the sewing lemma (\cite[Lemma 2.1]{MR2261056}), we see that $\cii[A^{(4)}]=\int_0^\cdot f(M_s)d[M]_s$.
    From \eqref{eqn:quadM}, we see that
    \[
      \|A^{(5)}_{s,t}\|_{L_2}\lesssim|\nabla^2f|_\infty|t-s|^{\frac12+\bar\varepsilon}
      \quad\textrm{and}\quad
      \E^{\cff_s}A^{(5)}_{s,t}=0
      \]
    for some $\bar \varepsilon>0$. By Proposition \ref{prop:AA}, $\cii[A^{(5)}]=0$ and hence $\cii[A^{(2)}]=\int_0^\cdot f(M_s)d[M]_s$. 
    Finally, we note that $\|A^{(3)}_{s,t}\|_{L_2}\lesssim\|M_{s,t}\|_{L_6}^3\lesssim|t-s|^{1+3 \varepsilon}$ and apply Proposition \ref{prop:AA} once again to see that $\cii[A^{(3)}]=0$.
  \end{example}

  \subsection{Proofs} % (fold)
  \label{sub:proofs}
  
  % subsection proofs (end)
  We present herein the proofs of Theorems \ref{lem:Sew1} and \ref{cor:1}.
  We first show in Lemma \ref{lem.dyest} two uniform estimates for the Riemann sums along dyadic partitions, which are straightforward  applications of inequalities \eqref{est:S1} and \eqref{est:S2}. 
  In Lemma \ref{lem.ANA}, similar uniform estimates for Riemann sums along \textit{any} partitions are obtained using a dyadic allocation procedure adapted from \cite[Lemma 2.2]{MR3860015}. The procedure is described in Lemma \ref{lem.Yaskov}, which asserts that every partition can be arranged into dyadic partitions so that the estimates for Riemann sums along dyadic partitions can be carried over more or less in the same fashion.

  \begin{lemma}\label{lem.dyest}
    Let $(A_{s,t})_{S\le s\le t\le T}$ be a two-parameter process in $\Rd$ which is $L_m$-integrable and adapted to $\{\cff_t\}$. We assume that there exist constants $\Gamma_1,\Gamma_2,\Gamma_3\ge0$ and $\varepsilon_1,\varepsilon_2,\varepsilon_3>0$ such that \eqref{con:dA1} holds and
    \begin{equation}\label{con.1032}
      \|\delta A_{s,u,t}-\E^{\cff_s}\delta A_{s,u,t}\|_{L_m}\le \Gamma_2|t-s|^{\frac12+\varepsilon_2}+\Gamma_3|t-s|^{\frac12+\varepsilon_3}
    \end{equation}
    for every $S\le s\le u\le t\le T$.
    For each $(s,t)$ in $[S,T]^2_\le$ and each integer $n\ge1$, let $\pi^n_{s,t}=\{s=t^n_0<t^n_1<\cdots<t^n_{2^n}=t\}$ be the dyadic partition of $[s,t]$ with mesh size $|\pi^n_{s,t}|=|t^n_i-t^n_{i+1}|=2^{-n}(t-s)$. 
    Define
    \begin{equation}\label{def.apin}
      A^n_{s,t}=\sum_{i=0}^{2^n-1}A_{t^n_i,t^n_{i+1}}\,.
    \end{equation}
    Then for every $n\ge1$, we have
    \begin{equation}
      \|A^n_{s,t}-A_{s,t}\|_{L_m}
      \le \frac{\Gamma_1}{1-2^{-\varepsilon_1}}|t-s|^{1+\varepsilon_1}+\frac{ \kappa_{m,d}\Gamma_2}{1-2^{-\varepsilon_2}}|t-s|^{\frac12+\varepsilon_2}+\frac{ \kappa_{m,d}\Gamma_3}{1-2^{-\varepsilon_3}}|t-s|^{\frac12+\varepsilon_3},
      \label{est.dyes1}
    \end{equation}
    and
    \begin{align}
      \|\E^{\cff_{s}}(A^n_{s,t}-A_{s,t})\|_{L_m}
      \le \frac{\Gamma_1}{1-2^{-\varepsilon_1}}|t-s|^{1+\varepsilon_1}\,.
      \label{est.dyes2}
    \end{align}
  \end{lemma}
  \begin{proof}
    For each $k,i$, let $u^k_i$ be the midpoint of $[t^k_i,t^k_{i+1}]$. We have
    \begin{align}\label{tmp:Ann}
      A^{k}_{s,t}-A^{k+1}_{s,t}=\sum_{i=0}^{2^k-1}\delta A_{t^k_i,u^k_i,t^k_{i+1}}
      =I_1+I_2
    \end{align}
    where
    \begin{align}\label{def:I12}
      I_1=\sum_{i=0}^{2^k-1}\E^{\cff_{t^k_i}}\delta A_{t^k_i,u^k_i,t^k_{i+1}}
      \,\textrm{ and }\, I_2=\sum_{i=0}^{2^k-1}\(\delta A_{t^k_i,u^k_i,t^k_{i+1}}-\E^{\cff_{t^k_i}}\delta A_{t^k_i,u^k_i,t^k_{i+1}}\)\,.
    \end{align}
    Observe that $A^{k}_{s,t}-A^{k+1}_{s,t}, I_1,I_2$ correspond respectively to $S, S_1, S_2$ in \eqref{S}.
    We use \eqref{est:S1} and \eqref{con:dA1} to estimate $I_1$, which yields
    \begin{align}
      \|I_1\|_{L_m}
      &\le \sum_{i=0}^{2^k-1}\|\E^{\cff_{t^k_i}}\delta A_{t^k_i,u^k_i,t^k_{i+1}}\|_{L_m}
      \nonumber\\& \le \Gamma_1|t-s|^{1+\varepsilon_1}\sum_{i=0}^{2^k-1}2^{-k(1+\varepsilon_1)}
      =\Gamma_1|t-s|^{1+\varepsilon_1}2^{-k\varepsilon_1}\,.
      \label{tmp.i1}
    \end{align}
    Similarly, we use \eqref{est:S2}, \eqref{con.1032} and Minkowski inequality to estimate $I_2$, which gives
    \begin{align}
      &\|I_2\|_{L_m}
      \le \kappa_{m,d} \(\sum_{i=0}^{2^k-1}\|\delta A_{t^k_i,u^k_i,t^k_{i+1}}-\E^{\cff_{t^k_i}}\delta A_{t^k_i,u^k_i,t^k_{i+1}}\|_{L_{m}}^2 \)^{\frac12}
      \nonumber\\&\le \kappa_{m,d}\( \Gamma_2^2 |t-s|^{1+2 \varepsilon_2}\sum_{i=0}^{2^k-1}2^{-k(1+2 \varepsilon_2)} \)^{\frac12}+\kappa_{m,d}\( \Gamma_3^2 |t-s|^{1+2 \varepsilon_3}\sum_{i=0}^{2^k-1}2^{-k(1+2 \varepsilon_3)} \)^{\frac12}
      \nonumber\\&=\kappa_{m,d} \Gamma_2|t-s|^{\frac12+\varepsilon_2}2^{-k \varepsilon_2}+\kappa_{m,d} \Gamma_3|t-s|^{\frac12+\varepsilon_3}2^{-k \varepsilon_3}\,.
      \label{tmp.i2}
    \end{align}
    Hence, we have shown that
    \begin{multline}\label{tmp.an}
      \|A^{k}_{s,t}-A^{k+1}_{s,t}\|_{L_m}\le \Gamma_1|t-s|^{1+\varepsilon_1}2^{-k\varepsilon_1}
      \\+\kappa_{m,d} \Gamma_2|t-s|^{\frac12+\varepsilon_2}2^{-k \varepsilon_2}+\kappa_{m,d} \Gamma_3|t-s|^{\frac12+\varepsilon_3}2^{-k \varepsilon_3}\,.
    \end{multline}
    Using triangle inequality, we have
    \begin{align*}
      \|A^n_{s,t}-A_{s,t}\|_{L_m}
      &\le\sum_{k=0}^\infty\|A^{k}_{s,t}-A^{k+1}_{s,t}\|_{L_m}.
      % \\&\le \frac{\Gamma_1}{1-2^{-\varepsilon_1}}|t-s|^{1+\varepsilon_1}+\frac{2 \kappa_{m,d}\Gamma_2}{1-2^{-\varepsilon_2}}|t-s|^{\frac12+\varepsilon_2}\,.
    \end{align*}
    Taking into account the estimate \eqref{tmp.an} and summing up the resulting geometric series in $k$, we obtain \eqref{est.dyes1} from the above inequality.
    Similarly, using triangle inequality, we have
    \[
      \|\E^{\cff_s}(A^n_{s,t}-A_{s,t})\|_{L_m}
      \le\sum_{k=0}^\infty\|\E^{\cff_s}(A^{k}_{s,t}-A^{k+1}_{s,t})\|_{L_m}.
    \]
    To estimate the right-hand side above, we use \eqref{tmp:Ann}, \eqref{tmp.i1} and the fact that $\E^{\cff_{s}}I_2=0$ to obtain that
    \begin{equation*}
      \|\E^{\cff_{s}}(A^{k}_{s,t}-A^{k+1}_{s,t})\|_{L_m} =\|\E^{\cff_{s}}I_1\|_{L_m}\le \Gamma_1|t-s|^{1+\varepsilon_1}2^{-n\varepsilon_1}\,.
    \end{equation*}
    Combining the previous inequalities together and sum up the resulting geometric series in $k$ yields \eqref{est.dyes2}.
  \end{proof}

  To obtain similar estimates to \eqref{est.dyes1} and \eqref{est.dyes2} for arbitrary partitions, the following result, which is adapted from Yaskov's \cite[Lemma 2.2]{MR3860015}, is needed.
  \begin{lemma}[Dyadic allocation] \label{lem.Yaskov}
    Let $(A_{s,t})_{S\le s\le t\le T}$ be a two-parameter process in $\Rd$ which is $L_m$-integrable and adapted to $\{\cff_t\}$. We assume that $A_{s,s}=0$ for every $s\in[S,T]$.
    Let $s,t$ be in $[S,T]$ with $s<t$. For each $n$, let $\{t^n_i\}_{i=0}^{2^n}$ be the dyadic partition of $[s,t]$ of mesh size $2^{-n}(t-s)$, that is $t^n_i=s+i2^{-n}(t-s)$ for each $n,i$.
    Then, for every $N\ge0$ and every  $s\le t_0<\cdots <t_N\le t$, there exist a positive integer $n_0$ and random variables $R^n_i$, $i=0,\cdots,2^n-1$, $n\ge0$, such that
    \begin{enumerate}[wide,itemsep=2pt,labelindent=0pt,label={\upshape(\roman*)}]
      \item\label{p.i} $R^n_i=0$ for every $n\ge n_0$ and every $i$;
      \item\label{p.ii} for each $n,i$, there exist four (not necessarily distinct) points $s^{n,i}_1\le s^{n,i}_2\le s^{n,i}_3\le s^{n,i}_4$ in $[t^n_i,t^n_{i+1}]$ so that
      \begin{equation}\label{form.Rni}
        R^n_i= A_{s^{n,i}_1,s^{n,i}_2}+A_{s^{n,i}_2,s^{n,i}_3}+A_{s^{n,i}_3,s^{n,i}_4}-A_{s^{n,i}_1,s^{n,i}_4}
      \end{equation}
      \item\label{p.iii} the following identity holds 
      \begin{equation}\label{id.yaskov}
        \sum_{i=0}^{N-1} A_{t_i,t_{i+1}}-A_{t_0,t_N}=\sum_{n\ge0}\sum_{i=0}^{2^n-1}R^n_i.
      \end{equation}
    \end{enumerate}   
  \end{lemma}
  \begin{proof}
    For each collection $\pi=\{s_i\}_{i=0}^K$ we define
    \[
      I(\pi)=\sum_{i=0}^{K-1} A_{s_i,s_{i+1}}-A_{s_0,s_K}
      \quad\textmd{if}\quad K\ge1
    \] 
    and $I(\pi)=0$ if either $K=0$ or $\pi$ is empty.
    For any two finite collections $\pi_1,\pi_2$, define
    \begin{equation}\label{def.dI}
      \delta I(\pi_1,\pi_2)=I(\pi_1\cup \pi_2)-I(\pi_1)-I(\pi_2).
    \end{equation}
    Put $\pi^0_0=\{t_i\}_{i=0}^N$, which is a subset of $[t^0_0,t^0_1]$. The main idea of the proof is to allocate the elements of $\pi^0_0$ into the dyadic subintervals of $[s,t]$ while keeping track of the resulting changes in $I(\pi^0_0)$ during the process.
    For each $n\ge1$, define
    \[
      \pi^n_{2^n-1}=\pi^0_0\cap[t^n_{2^n-1},t^n_{2^n}]
      \quad\textrm{and}\quad
      \pi^n_i=\pi^0_0\cap[t^n_i,t^n_{i+1})
      \quad\textmd{for}\quad i=0,\dots,2^n-2.
    \]
    For each $n\ge0$ and $i=0,\dots,2^n-1$, define
    \begin{equation}\label{def.Rni}
      R^n_i:=\delta I(\pi^{n+1}_{2i},\pi^{n+1}_{2i+1})= I(\pi^n_i)-I(\pi^{n+1}_{2i})-I(\pi^{n+1}_{2i+1})
    \end{equation}
    where the second identity comes from the fact that $\pi^n_i=\pi^{n+1}_{2i}\cup \pi^{n+1}_{2i+1}$. We verify that the random variables $\{R^n_i\}_{n,i}$ satisfy \ref{p.i}-\ref{p.iii}.

    Since $\pi^0_0$ is a finite set, there exists a finite integer $n_0\ge1$ so that $[t^n_i,t^n_{i+1}]\cap \pi^0_0$ contains at most one point for every $n\ge n_0$ and every $i=0,\dots,2^{n}-1$. Hence, when $n\ge n_0$, we have $I(\pi^{n}_i)=0$ and $R^n_i=0$ for every $i$. This shows \ref{p.i}. 

    % Since points in $\pi^n_i$ are $\cff_{t^n_{i+1}}$-measurable, we deduce from \eqref{def.Rni} that $R^n_i$ is $\cff_{t^n_{i+1}}$-measurable. 
    If either $\pi^{n+1}_{2i}$ or $\pi^{n+1}_{2i+1}$ is empty, then $R^n_i=0$ and \eqref{form.Rni} is satisfied with $s^{n,i}_j=t^n_i$ for $j=1,2,3,4$. In the case when $\pi^{n+1}_{2i}$ and $\pi^{n+1}_{2i+1}$ are not empty, we define 
    \[
      s^{n,i}_1=\min \pi^{n+1}_{2i},\, s^{n,i}_2=\max \pi^{n+1}_{2i},\, s^{n,i}_3=\min \pi^{n+1}_{2i+1}\, \mbox{ and } \,s^{n,i}_4=\max \pi^{n+1}_{2i+1}.
    \] Here, $\min$ (respectively $\max$) of a nonempty finite set $F$ is the smallest (respectively largest) element of $F$.  We derive \eqref{form.Rni} from \eqref{def.Rni} and the definition of $I$ at the beginning of the proof.
    Hence, \ref{p.ii} is verified.

    Lastly, to show \ref{p.iii}, we apply \eqref{def.Rni} recursively to see that
    \begin{align*}
      I(\pi^0_0)
      &=I(\pi^1_0)+I(\pi^1_1)+R^0_0
      \\&=\sum_{i=0}^{2^n-1}I(\pi^n_i)+\sum_{k=0}^{n-1}\sum_{i=0}^{2^k-1}R^k_i
      \quad\textrm{for every}\quad n\ge1.
    \end{align*}
    Since $I(\pi^n_i)=0$ as soon as $n\ge n_0$, the previous identity implies \eqref{id.yaskov}. This completes the proof.
  \end{proof}
  When combined with estimates \eqref{con:dA1} and \eqref{con:dA2} on $\delta A$, the previous lemma yields uniform estimates on Riemann sums along arbitrary partitions.
  \begin{lemma}\label{lem.ANA}
    Let $A$ be the process in Theorem \ref{lem:Sew1}. Then for every $N\ge0$ and every $S\le t_0<\cdots<t_N\le T$, we have
    \begin{equation}\label{est.AN1}
      \|\E^{\cff_{t_0}}\left(\sum_{i=0}^{N-1}A_{t_i,t_{i+1}}-A_{t_0,t_N}\right)\|_{L_m}\le \frac{2\Gamma_1}{1-2^{-\varepsilon_1}} |t_N-t_0|^{1+\varepsilon_1}
    \end{equation}
    and
    \begin{align}\label{est.AN2}
      \|\sum_{i=0}^{N-1}A_{t_i,t_{i+1}}-A_{t_0,t_N}\|_{L_m}\le \frac{2 \Gamma_1}{1-2^{-\varepsilon_1}} |t_N-t_0|^{1+\varepsilon_1}+\frac{2 \kappa_{m,d}\Gamma_2}{1-2^{-\varepsilon_2}} |t_N-t_0|^{\frac12+\varepsilon_2}.
    \end{align}
  \end{lemma}
  \begin{proof}
    Put $s=t_0$ and $t=t_N$. Applying Lemma \ref{lem.Yaskov}, we can find random variables $R^n_i$, $i=0,\cdots, 2^n-1$, $n\ge0$ which satisfy properties \ref{p.i}-\ref{p.iii} stated there. 

    Let $n\ge0$ be fixed. Define $\cgg^n_{2^n}=\cff_{t^n_{2^n}}$
    and $\cgg_i^n=\cff_{s_1^{n,i}}$ for each $i=0,\dots,2^n-1$. We recall that $s^{n,i}_1$ is defined in \ref{p.ii}. The sequence $\{\cgg_i^n\}_{i=0}^{2^n}$ forms a filtration such that $R^n_i$ is $\cgg^n_{i+1}$-measurable for every $i=0,\cdots,2^n-1$.
    The formula \eqref{form.Rni} can be written as
    \[
      R^n_i=-\delta A_{s^{n,i}_1,s^{n,i}_2,s^{n,i}_3}-\delta A_{s^{n,i}_1,s^{n,i}_3,s^{n,i}_4}
    \]
    Applying the conditions \eqref{con:dA1} and \eqref{con:dA2}, we obtain from the previous identity that
    \begin{equation}\label{tmp.rni}
      \|\E^{\cgg^n_i}R^n_i\|_{L_m}\le 2 \Gamma_1|t^n_{i+1}-t^n_{i}|^{1+\varepsilon_1}
      \quad\textrm{and}\quad
      \|R^n_i-\E^{\cgg^n_i}R^n_i\|_{L_m}\le 2 \Gamma_2|t^n_{i+1}-t^n_{i}|^{\frac12+\varepsilon_2}.
    \end{equation}
    For each $n\ge0$, we apply inequality \eqref{est:S} to obtain that
    \begin{align*}
      \|\sum_{i=0}^{2^n-1}R^n_i\|_{L_m}
      &\le\sum_{i=0}^{2^n-1}\|\E^{\cgg^n_i}R^n_i\|_{L_m}+ \kappa_{m,d}\left(\sum_{i=0}^{2^n-1}\|R^n_i-\E^{\cgg^n_i}R^n_i\|_{L_m}^2\right)^{\frac12}.
    \end{align*}
    The bounds in \eqref{tmp.rni} are applied to estimate each sum on the right-hand side above
    \begin{align}\label{tmp.R1}
      \sum_{i=0}^{2^n-1}\|\E^{\cgg^n_i}R^n_i\|_{L_m}\le 2\Gamma_1|t-s|^{1+\varepsilon_1}\sum_{i=0}^{2^n-1}2^{-n(1+\varepsilon_1)}
      =2\Gamma_1|t-s|^{1+\varepsilon_1}2^{-n\varepsilon_1}
    \end{align}
    and
    \begin{align*}
      \sum_{i=0}^{2^n-1}\|R^n_i-\E^{\cgg^n_i}R^n_i\|_{L_{m}}^2 
      \le  4\Gamma_2^2 |t-s|^{1+2 \varepsilon_2}\sum_{i=0}^{2^n-1}2^{-n(1+2 \varepsilon_2)} 
      = \left(2\Gamma_2|t-s|^{\frac12+\varepsilon_2}2^{-n \varepsilon_2}\right)^2.
    \end{align*}
    Combining the previous three estimates yields
    \begin{align}\label{tmp.R2}
      \|\sum_{i=0}^{2^n-1}R^n_i\|_{L_m}
      \le 2\Gamma_1|t-s|^{1+\varepsilon_1}2^{-n\varepsilon_1}+2\kappa_{m,d}\Gamma_2|t-s|^{\frac12+\varepsilon_2}2^{-n \varepsilon_2}.
    \end{align}

    From \eqref{id.yaskov}, applying triangle inequality and contraction property of conditional expectation (noting that $\cff_{t_0}\subset\cgg^n_i$) , we see that
    \begin{align*}
      \|\E^{\cff_{t_0}}\left(\sum_{i=0}^{N-1}A_{t_i,t_{i+1}}-A_{t_0,t_N}\right)\|_{L_m}
      \le \sum_{n\ge0}\sum_{i=0}^{2^n-1}\|\E^{\cgg^n_i}R^n_i\|_{L_m}
    \end{align*}
    and
    \begin{align*}
      \|\sum_{i=0}^{N-1}A_{t_i,t_{i+1}}-A_{t_0,t_N}\|_{L_m}
      \le \sum_{n\ge0}\|\sum_{i=0}^{2^n-1}R^n_i\|_{L_m}.
    \end{align*}
    We apply inequalities \eqref{tmp.R1} and \eqref{tmp.R2} to the previous two estimates and compute the resulting geometric sums in $n$ to obtain \eqref{est.AN1} and \eqref{est.AN2}.
  \end{proof}
  \begin{remark}\label{rmk.23}
     Of course, one could consider a more general situation where conditions \eqref{con:dA1} and \eqref{con:dA2} are replaced by
    \begin{equation}
      \|\E^{\cff_s}\delta A_{s,u,t}\|_{L_m}\le \rho_1(s,t)
      \quad\textrm{and}\quad
      \|\delta A_{s,u,t}-\E^{\cff_s}\delta A_{s,u,t}\|_{L_m}\le \rho_2(s,t)
    \end{equation}
    for every $S\le s\le u\le t\le T$, where $\rho_1,\rho_2$ are some non-negative functions with certain properties. As in Lyons' rough path theory (\cite{MR2314753,MR2604669}), we expect that such an extension is necessary to enlarge the scope of applications of the stochastic sewing lemma and possibly improve some of the results stated in the current paper. However, we will not consider such general situation on this occasion and defer it for later publications.
   \end{remark}
  \begin{proof}[Proof of Theorem \ref{lem:Sew1}] We divide the proof into three steps. In the first step, a process $\caa$ is constructed as limit of some Riemann sums using Lemma \ref{lem.ANA}. The second step verifies the properties \ref{cl:a} and \ref{cl:b} using Lemma \ref{lem.dyest}. In the last step, uniqueness of $\caa$ is shown using \ref{cl:a} and \ref{cl:b}.

    \textit{Step 1.} Let $t$ be fixed but arbitrary in $[S,T]$. We show that the Riemann sums
    \[
      A_t^\pi=\sum_{i=0}^{n-1} A_{t_i,t_{i+1}}
    \]
    over partitions $\pi=\{t_i\}_{i=0}^n$ of $[0,t]$ has a limit in $L_m$, denoted by $\caa_t$, as the mesh size $|\pi|:=\max_i|t_{i+1}-t_i|$ shrinks to $0$. 
    Let $\pi'=\{s_i\}_{i=0}^{n'}$ be another partition of $[0,t]$ and define $\pi''=\pi\cup \pi'$. We denote the points in $\pi''$ by $\{u_i\}_{i=0}^{n''}$ with $u_0\le u_1\le\cdots\le u_{n''}$ and some integer $n''\le n+n'$.  Then we have
    \begin{align*}
      A_t^{\pi''}-A_t^\pi=\sum_{i=0}^{n-1}Z_i
      \quad\textrm{where}\quad
      Z_i=\sum_{j:t_i\le u_j<t_{i+1}}A_{u_j,u_{j+1}}-A_{t_i,t_{i+1}}.
    \end{align*}
    Applying inequality \eqref{est:S}, we have
    \begin{align*}
      \|A_t^{\pi''}-A_t^\pi\|_{L_m} 
      &\lesssim \sum_{i=0}^{n-1}\|\E^{\cff_{t_i}}Z_i\|_{L_m}+\left(\sum_{i=0}^{n-1}\|Z_i\|_{L_m}^2\right)^{\frac12}.
    \end{align*}
    Lemma \ref{lem.ANA} is applied to obtain that
    \begin{align*}
      \|\E^{\cff_{t_i}}Z_i\|_{L_m}\lesssim|t_{i+1}-t_i|^{1+\varepsilon_1}
      \quad\textrm{and}\quad
      \|Z_i\|_{L_m}\lesssim|t_{i+1}-t_i|^{\frac12+\varepsilon_2}.
    \end{align*}
    This implies that
    \[
      \sum_{i=0}^{n-1}\|\E^{\cff_{t_i}}Z_i\|_{L_m}\lesssim \sum_{i=0}^{n-1}|t_{i+1}-t_{i}|^{1+\varepsilon_1}\lesssim |\pi|^{\varepsilon_1}
    \]
    and similarly $\sum_{i=0}^{n-1}\|Z_i\|_{L_m}^2\lesssim |\pi|^{\varepsilon_2}$.
     Hence, we have
    \begin{align*}
      \|A_t^\pi-A_t^{\pi''}\|_{L_m} 
      &\lesssim |\pi|^{\varepsilon_1}+|\pi|^{\varepsilon_2}.
    \end{align*}
    The same argument is applied to $A_t^{\pi'}-A_t^{\pi''}$ which gives $\|A_t^{\pi'}-A_t^{\pi''}\|_{L_m}\lesssim |\pi'|^{\varepsilon_1}+|\pi'|^{\varepsilon_2}$. Hence, by triangle inequality, 
    \begin{align*}
      \|A_t^\pi-A_t^{\pi'}\|_{L_m}\le\|A_t^\pi-A_t^{\pi''}\|_{L_m}+\|A_t^{\pi'}-A_t^{\pi''}\|_{L_m}
      \lesssim |\pi|^{\varepsilon_1}+|\pi|^{\varepsilon_2}+|\pi'|^{\varepsilon_1}+|\pi'|^{\varepsilon_2}.
    \end{align*}
    This implies that $\{A_t^\pi\}_\pi$ is Cauchy in $L_m$ and hence $\caa_t:=\lim_{|\pi|\to0}A_t^\pi$ is well-defined in $L_m$.

    \textit{Step 2.} We show that the process $(\caa_t)_{S\le t\le T}$ defined in the previous step satisfies \ref{cl:a} and \ref{cl:b}.

    The condition \eqref{con:dA2} with $s=u=t$ implies that $A_{s,s}=0$ for every $s\in[S,T]$. Hence, it is evident that $\caa_S=0$. The fact that $A$ is $\{\cff_t\}$-adapted implies that $\caa$ is $\{\cff_t\}$-adapted. Obviously, $\caa_t$, being a limit in $L_m$, belongs to $L_m$ for each $t$. This shows \ref{cl:a}.

    Let $(s,t)$ be fixed but arbitrary in $[S,T]^2_\le$. For each integer $n\ge1$, define $A^n_{s,t}$ as in \eqref{def.apin}. Applying Lemma \ref{lem.dyest}, we obtain that
    \begin{equation}
      \|A^n_{s,t}-A_{s,t}\|_{L_m}
      \le \frac{\Gamma_1}{1-2^{-\varepsilon_1}}|t-s|^{1+\varepsilon_1}+\frac{ \kappa_{m,d}\Gamma_2}{1-2^{-\varepsilon_2}}|t-s|^{\frac12+\varepsilon_2}\,,
      \label{tmp2:estA}
    \end{equation}
    and
    \begin{align}
      \|\E^{\cff_{s}}(A^n_{s,t}-A_{s,t})\|_{L_m}
       \le \frac{\Gamma_1}{1-2^{-\varepsilon_1}}|t-s|^{1+\varepsilon_1}\,.
      \label{tmp2:estA2}
    \end{align}
    From construction of $\caa$ in the previous step, we see that
    \begin{equation}\label{tmp.aa}
      \caa_t-\caa_s=\lim_{n\to\infty}A^n_{s,t} \quad\textrm{in}\quad L_m.
    \end{equation}
    Hence, passing through the limit $n\to\infty$ in \eqref{tmp2:estA} and \eqref{tmp2:estA2}, we obtain \eqref{est:A1} and \eqref{est:A2} respectively with constants $C_1= \Gamma_1(1-2^{-\varepsilon_1})^{-1}$ and $C_2=\kappa_{m,d}\Gamma_2(1-2^{-\varepsilon_2})^{-1}$. This shows \ref{cl:b}.
   
    \textit{Step 3. Uniqueness:} We show that properties \ref{cl:a} and \ref{cl:b} uniquely determine the process $\caa$. Suppose $\bar\caa$ is (another) $\{\cff_t\}$-adapted, $L_m$-integrable process on $[S,T]$ which satisfies $\bar\caa_S=0$ and for every $S\le s\le t\le T$
    \begin{gather*}
      \|\bar\caa_t-\bar\caa_s-A_{s,t}\|_{L_m}\le \bar C|t-s|^{\frac12+\bar\varepsilon}\,,
      \\\|\E^{\cff_{s}}(\bar\caa_t-\bar\caa_s-A_{s,t})\|_{L_m}\le \bar C|t-s|^{1+\bar\varepsilon}
    \end{gather*}
    for some positive constants $\bar C,\bar \varepsilon$. Then the difference $\tilde \caa=\caa-\bar \caa$ satisfies 
    \begin{equation}\label{tmp:tildeA}
      \|\tilde\caa_t-\tilde\caa_s\|_{L_m}\le C|t-s|^{\frac12+\varepsilon}
      \quad\textrm{and}\quad\|\E^{\cff_{s}}(\tilde\caa_t-\tilde\caa_s)\|_{L_m}\le  C|t-s|^{1+\varepsilon}
    \end{equation}
    for every $S\le s\le t\le T$, for some positive constants $C,\varepsilon$. Let us now fix $t\in[S,T]$, an integer $n\ge1$ and put $t_i=S+i2^{-n}(t-S) $ for each $i=0,\dots,2^{n}$. We write
    \begin{align*}
      \tilde\caa_t=\sum_{i=0}^{2^n-1}\(\tilde\caa_{t_{i+1}}-\tilde\caa_{t_i}\)
    \end{align*}
    and apply inequality \eqref{est:SS} (with $\cgg_i=\cff_{t_i}$) to obtain that
    \begin{align*}
      \|\tilde\caa_t\|_{L_m}\le \sum_{i=0}^{2^n-1}\|\E^{\cff_{t_i}}\(\tilde\caa_{t_{i+1}}-\tilde\caa_{t_i}\)\|_{L_m}+2 \kappa_{m,d}\(\sum_{i=0}^{2^n-1}\|\tilde\caa_{t_{i+1}}-\tilde\caa_{t_i}\|^2_{L_m}\)^{\frac12}\,.
    \end{align*}
    In view of \eqref{tmp:tildeA}, the above inequality yields
    $\|\tilde \caa_t\|_{L_m}\lesssim 2^{-n\varepsilon} $ for all integer $n\ge1$. This means $\caa_t=\bar\caa_t$ a.s. for all $t\in [S,T]$.
  \end{proof}
  \begin{remark}\label{rmk.pfSSL}
    (i) In the second step of the above proof, we can use Lemma \ref{lem.ANA} in place of Lemma \ref{lem.dyest}. However, this will produce estimates \eqref{est:A1} and \eqref{est:A2} with constants $C_1=2 \Gamma_1(1-2^{-\varepsilon_1})^{-1}$ and $C_2=2\kappa_{m,d}\Gamma_2(1-2^{-\varepsilon_2})^{-1}$.

    \noindent(ii) It is evident from the above proof that we can replace the condition that $A$ is $L_m$-integrable by the condition that $\delta A$ is $L_m$-integrable, meaning that $\delta A_{s,u,t}$ is $L_m$-integrable for every $s\le u\le t$. In this case, it is necessary to replace \ref{cl:a} in Theorem \ref{lem:Sew1} by
    \begin{enumerate}[wide,itemsep=2pt,labelindent=1pt,label={(\ref{lem:Sew1}a')}]
      \item  $\caa_S=0$, $\caa$ is $\{\cff_t\}$-adapted and $\caa_t-A_{S,t}$ is $L_m$-integrable for every $t\in[S,T]$,
    \end{enumerate}
    and we have $\{A^\pi_t-A_{S,t}\}_\pi$ converges to $\caa_t-A_{S,t}$ in $L_m$ as $|\pi|\downarrow0$.
    Similarly, the condition that $A$ is $\{\cff_t\}$-adapted can be replaced by the condition that $\delta A$ is $\{\cff_t\}$-adapted, meaning that $\delta A_{s,u,t}$ is $\cff_t$-measurable for every $s\le u\le t$. 
    In Theorem \ref{lem:Sew1}, if one only assumes that $\delta A$ is $\{\cff_t\}$-adapted and $L_m$-integrable, it is necessary to replace \ref{cl:a}  by
    \begin{enumerate}[wide,itemsep=2pt,labelindent=1pt,label={(\ref{lem:Sew1}a'')}]
      \item $\caa_S=0$, $\caa_t-A_{S,t}$ is $\cff_t$-measurable and $L_m$-integrable for every $t\in[S,T]$.
    \end{enumerate}
  \end{remark}
  \begin{proof}[Proof of Theorem \ref{cor:1}]
    We recall the definitions of $J,M$ in \eqref{def.JM}. 
    % From the estimates \eqref{est.EdJ}, \eqref{est.dJ} and \eqref{id.EM}, \eqref{est.dM}, we see that $J,M$ satisfies the hypotheses of Theorem \ref{lem:Sew1}. 
    We have seen from the discussion succeeding \eqref{def.JM} the existence of the processes $\cjj,\cmm$ which are $\{\cff_t\}$-adapted, $L_m$-integrable and satisfy \ref{cl:m}, \ref{cl:est.j}, \ref{cl:est.j'}. 
    In addition, the same arguments as Step 1 in the proof of Theorem \ref{lem:Sew1} imply that for every $t\in[S,T]$
    \begin{equation}\label{tmp.jjmm}
      \cjj_t=\lim_{|\pi|\to0}\sum_{i=0}^{N-1}J_{t_i,t_{i+1}}
      \quad\textrm{and}\quad
      \cmm_t=\lim_{|\pi|\to0}\sum_{i=0}^{N-1}M_{t_i,t_{i+1}}
    \end{equation}
    where the limits are taken in $L_m$ and $\pi=\{0=t_0<\cdots <t_N=t\}$ denotes a partition of $[0,t]$.
    
    We now show the remaining properties stated in Theorem \ref{cor:1}. Let $(s,t)$ be fixed but arbitrary in $[S,T]^2_\le$. 
    For each integer $n\ge1$, let $\pi^n_{s,t}=\{s=t^n_0<t^n_1<\cdots<t^n_{2^n}=t\}$ be the dyadic partition of $[s,t]$ as defined in Lemma \ref{lem.dyest}.
    Define
    \begin{equation*}
      M^n_{s,t}=\sum_{i=0}^{2^n-1}M_{t^n_i,t^n_{i+1}}\,.
    \end{equation*}
    From \eqref{id.EM}, \eqref{est.dM} and Lemma \ref{lem.dyest}, we obtain that
    \begin{align*}
      \|M^n_{s,t}-M_{s,t}\|_{L_m}
      &\le 
      \frac{\kappa_{m,d}\Gamma_2}{1-2^{-\varepsilon_2}} |t-s|^{\frac12+\varepsilon_2}+\frac{\kappa_{m,d}\Gamma_3}{1-2^{-\varepsilon_3}} |t-s|^{\frac12+\varepsilon_3}.
    \end{align*}
    Using the property \eqref{tmp.jjmm}, the above inequality implies \eqref{est:M} with constants $C_2= \kappa_{m,d}\Gamma_2(1-2^{-\varepsilon_2})^{-1}$ and $C_3= \kappa_{m,d}\Gamma_3(1-2^{-\varepsilon_3})^{-1}$. Hence \ref{cl:est.m} is verified.

    To see that $\caa=\cmm+\cjj$, we note that for any partition $\pi=\{0=t_0<\cdots<t_N=t\}$ of  $[0,t]$, we have
    \[
      \sum_{i=0}^{N-1}A_{t_i,t_{i+1}}=\sum_{i=0}^{N-1}M_{t_i,t_{i+1}}+\sum_{i=0}^{N-1}J_{t_i,t_{i+1}}.
    \]
    In view of \eqref{tmp.aa}, \eqref{tmp.jjmm}, the above identity implies that $\caa_t=\cmm_t+\cjj_t$ a.s. This shows \ref{cl:amj}.

        \textit{Uniqueness:} Suppose $(\cmm,\cjj)$ and $(\bar\cmm,\bar\cjj)$ are two pairs of processes which satisfy \ref{cl:amj}, \ref{cl:m} and \ref{cl:est.m} with the same processes $A$ and $\caa$. Then the difference processes $\tilde\cmm=\cmm-\bar\cmm$ and $\tilde\cjj=\cjj-\bar\cjj$ satisfy
    \begin{equation}\label{tmp.527}
      \tilde\cmm_t=-\tilde\cjj_t \quad \textrm{a.s.} \quad\textrm{and}\quad\|\tilde\cjj_t-\tilde\cjj_s\|_{L_m}\le C|t-s|^{\frac12+\varepsilon}
    \end{equation}
    for every $S\le s\le t\le T$, for some  constants $C\ge0$ and $\varepsilon>0$. The process $\tilde\cmm$, being a square integrable martingale with super-diffusive increments, ought to be a constant process. Alternatively, one could use the same arguments used in Step 3 of the proof of Theorem \ref{lem:Sew1} to conclude that $\tilde\cmm$ is a constant process. Since $\cmm_S=\bar\cmm_S=0$, this implies that $\cmm=\bar\cmm$ and $\cjj=\bar\cjj$.

    If instead, $(\cmm,\cjj)$ and $(\bar\cmm,\bar\cjj)$ satisfy \ref{cl:amj}, \ref{cl:m} and  \ref{cl:est.j} with the same processes $A$ and $\caa$, the difference processes $\tilde\cmm$ and $\tilde\cjj$ also satisfy \eqref{tmp.527}. Hence, we also have $\cmm=\bar\cmm$ and $\cjj=\bar\cjj$. This shows the claim \ref{ch1}.

    Since $\cmm$ is obtained by applying Theorem \ref{lem:Sew1} to $(s,t)\mapsto M_{s,t}=A_{s,t}-\E^{\cff_s}A_{s,t}$, the characterization of $\cmm$ in \ref{ch2} follows from the uniqueness part of Theorem \ref{lem:Sew1}. One can show the characterization of $\cjj$ in  \ref{ch3} analogously. 
  \end{proof}
  
% section stochastic_sewing_lemma (end)
\section{Distributive functionals of Markov processes} % (fold)
\label{sec:additive_functionals}
  In probability theory and in the study of stochastic differential equations with distributional drifts (\cite{MR1964949,MR3500267,MR3652414,MR3785598}), one encounters the integral of the form $\int_0^\cdot f(X_r)dr$ where $X$ is a solution to a stochastic differential equation (for instance a Brownian motion) and $f$ is a distribution. 
  To define and study such integrals, previous methods  involve explicit moment computations (as in the study of local times, \cite{MR2250510}) or the so-call Zvonkin transformation, also known as the It\^o-Tanaka trick (commonly used in \cite{MR1964949,MR3581216,MR3500267}). In the present section, a new method is proposed which is based on Theorem \ref{lem:Sew2}.
  This method relies on a certain smoothing effect of the transition probability kernel associated to $X$ and does not require explicit moment computations nor It\^o calculus to work.
  The main results are described below in Propositions \ref{prop:defA} and \ref{prop:Lip}.

  We proceed with the detail. Let us fix a time horizon $T>0$ and let $(\Omega,\cff,\P)$  be a complete probability space equipped with a filtration $\{\cff_t\}$ such that $\cff_0$ contains $\P$-null sets.
  For each $m\ge1$ and $\tau\in(0,1]$, the space $C^\tau_T L_m=C^\tau([0,T];[L^m(\Omega,\cff,\P)]^d)$ contains all stochastic processes $\psi$ with values in $\R^d$ such that
  \begin{equation*}
    \|\psi\|_{C^\tau_T L_m} := \sup_{t\in[0,T]}\|\psi_t\|_{L_m}+\sup_{s,t\in[0,T];\,s\neq t}\frac{\|\psi_t- \psi_s\|_{L_m}}{|t-s|^\tau}<\infty\,.
  \end{equation*}
  Similarly, the space $C_TL_m=C([0,T];[L^m(\Omega,\cff,\P)]^d)$ contains all stochastic processes $\psi$ with values in $\R^d$ such that $\psi:[0,T]\to L_m$ is continuous. The norm of $\psi$ in $C_TL_m$ is defined by $\|\psi\|_{C_T L_m} := \sup_{t\in[0,T]}\|\psi_t\|_{L_m}$.

  Let $X=\{X_t\}_{t\ge0}$ be a Markov process in $\R^d$ with respect to $\{\cff_t\}$ with Lebesgue-measurable sample paths.
  Let $(\chh,\|\cdot\|_\chh)$ be a normed vector space, which contains $C_b(\R^d)$ and is a subset of $\css'(\R^d)$. Let $U=\{U_{s,t}\}_{0\le s\le t}$ be a two-parameter family of bounded operators on $\chh$. The following condition relates $X,U$ and $\chh$.
  \begin{condition}\label{con:set1}
    $(X,U,\chh)$ satisfies: 
    \begin{enumerate}[label={\upshape(\roman*)}]
      % \item\label{set1a} for every $0\le s< t$, $U_{s,t}$ maps $\chh$ to $C_b(\R^d)$,
      \item\label{set1c} there are finite numbers $\|U\|_{\chh,\nu}\ge0$ and $\nu\le 0$ such that for every $h\in\chh$ and every $0\le s< t\le T$, $U_{s,t}h$ is a bounded continuous function on $\Rd$ and satisfies
      \begin{equation}\label{con:U1}
        \|U_{s,t}h\|_{\infty}\le \|U\|_{\chh,\nu}\|h\|_\chh |t-s|^{\frac{\nu}2}\,,
      \end{equation}
      where $\|\cdot\|_\infty$ is the supremum norm,
      \item\label{set1b} for every $s\ge0$, $x\in\Rd$ and $h\in\chh$, the map $[s,T]\ni r\to U_{s,r}h(x)\in\R$ is Lebesgue measurable,
      % \item\label{set1group} for every $s\le r< t$, $x\in\Rd$ and $h\in\chh$, $U_{s,r}[U_{r,t}h](x)=U_{s,t}h(x)$,
      \item\label{set1EU} for every $s\le r< t$, $x\in\Rd$ and $h\in\chh$, $\E^{\cff_s}[U_{r,t}h](X_r)=U_{s,t}h(X_s)$.    
     \end{enumerate}
  \end{condition}
  The first two conditions \ref{set1c},\ref{set1b} are purely analytic which are necessary for integrations. The last condition \ref{set1EU} 
  % relates the propagator $U$ with the Markov process $X$. Heuristically,  it 
  enforces that $U$ is the propagator which corresponds to the transition function of $X$ (see \cite[pg. 156]{MR838085}).

  We think of $\chh$ as a space containing irregular functions, measures or distributions with a certain negative regularity index. To be more specific, we briefly recall the definition of H\"older-Besov spaces.
  For each $\phi\in\css(\Rd)$, $x\in\Rd$ and $\lambda>0$, $\phi_x^\lambda$ is defined by $\phi_x^\lambda(y)=\lambda^{-d}\phi(\lambda^{-1}(y-x))$.
  For each integer $r\ge0$, let $\mathfrak{B}^r$ be the set of all smooth functions $\phi$ on $\Rd$ which are supported on the unit ball $\{x\in\Rd:|x|\le 1\}$ and satisfy $\|\phi\|_{C^r_b(\Rd)}\le1$.
  \begin{definition}\label{def.Besov}
    Let $\gamma\le0$ be a real number and $r>|\gamma|$ be an integer. The homogeneous H\"older-Besov space $\C^\gamma(\Rd)$ consists of all tempered distributions $g$ such that 
      \[
        \|g\|_{\C^\gamma}:=\sup_{\lambda\in(0,1]}\sup_{x\in\Rd} \sup_{\phi\in \mathfrak{B}^r} \lambda^{-\gamma} |g(\phi^\lambda_x)|<\infty.
      \]
   \end{definition}
  We refer to \cite{MR1228209} for more detail on the Besov-H\"older spaces. 
  As a toy example of a tuple $(X,U,\chh)$ which satisfies Condition \ref{con:set1}, one may take $(X,\{\cff_t\})$ as a standard Brownian motion in $\Rd$, $U_{s,t}=P_{s,t}$ which is the operator associated to the Gaussian kernel $(2 \pi(t-s))^{-\frac d2}e^{-|x-y|^2/(2(t-s))}$, and $\chh$ as the Besov-H\"older space $\C^{\nu}(\Rd)$. Since $P_{s,t}$ has a smooth kernel whenever $s<t$, the action of $P_{s,t}$ on any distribution is a smooth function. Combining with known Schauder estimates of heat kernels (see e.g. \cite[Lemma 2.8]{MR3358965} or \cite[Prop. 2.4]{MR3785598}), we see that for every $g\in\C^\nu$ and $s<t$, $P_{s,t}g$ is continuous on $\Rd$ and satisfies
  \begin{equation}\label{P.schau}
    \|P_{s,t}g\|_\infty\le C|t-s|^{\frac \nu2}\|g\|_{\C^\nu}
  \end{equation}
  for some positive constant $C$ independent from $s,t$ and $g$.
  It follows that Condition \ref{con:set1} is satisfied for every $\nu\le 0$. Other examples will appear in subsequent sections. 
  
  For each integer $k\ge0$, let $\cee_T$ and $\cee_T^k$ be respectively the sets of measurable finite-valued functions from $[0,T]$ to $\chh$ and $C_b^k(\Rd)$. A function $f$ in $\cee_T$ has the following form
  \begin{equation}\label{eqn:fcee}
    f_r=\sum_{i\in F}1_{I_i}(r)h_i \quad\forall r\in[0,T]\,,
  \end{equation}
  where $F$ is a finite set, $I_i$'s are Lebesgue-measurable subsets of $[0,T]$ and $h_i$'s are elements in $\chh$. Similarly, a function $f$ in $\cee^k_T$ has the form \eqref{eqn:fcee} with $h_i\in C_b^k(\Rd)$ for every $i$. Functions in $\cee_T$ and $\cee_T^k$ are also called simple functions.

  A function $f:[0,T]\to\chh$ is strongly Lebesgue-measurable if there exists a sequence of simple functions $\{f^n\}$ in $\cee_T$ such that $\lim_n f^n_r=f_r$ in $\chh$ for Lebesgue-a.e. $r\in[0,T]$.
  For each $q\in[1,\infty)$, the Bochner space $L^q([0,T];\chh)$ is the linear space of all (equivalence classes of) strongly Lebesgue-measurable functions $f:[0,T]\to\chh$ for which the Lebesgue integral $\int_{[0,T]}\|f_r\|_\chh^qdr$ is finite.
  We define $L^\infty([0,T];\chh)$ as the linear space of all (equivalence classes of) strongly Lebesgue-measurable functions $f:[0,T]\to\chh$ for which there exists a real number $N\ge0$ such that the set $\{r\in[0,T]:\|f_r\|_\chh>N \}$ has Lebesgue measure $0$. For brevity, we abbreviate $L^q_T\chh$ for $L^q([0,T];\chh)$ for every $q\in[1,\infty]$.
  The norm on $L^q_T\chh$ is defined by
  \[
    \|f\|_{L^q_T\chh}:=\left(\int_{[0,T]}\|f_r\|_\chh^q dr\right)^{\frac1q}
    \quad\textrm{for finite } q
  \]
  and when $q=\infty$ by
  \[
    \|f\|_{L^\infty_T\chh}:=\esssup_{r\in[0,T]}\|f_r\|_\chh.
  \]
  For further properties on Bochner spaces and measurability of Banach-valued functions, we refer the reader to \cite[Chapter 1]{MR3617205}.

  For a given $f\in L^q_T\chh$, measurability and integrability of the map $r\mapsto U_{s,r}f_r$ are discussed in the following two lemmas.
  \begin{lemma}\label{lem.meas}
    % Let $f$ be in $L^q_T\chh$,
    Let $(X,U,\chh)$ satisfy Condition \ref{con:set1}.
    For every $f\in L^q_T\chh$ and $x\in\Rd$, the map $r\mapsto U_{s,r}f_r(x)$ is Lebesgue-measurable from $[s,T]$ to $\R$.  
  \end{lemma}
  \begin{proof}
    Since $f$ is strongly Lebesgue-measurable, we can find a sequence of simple functions $\{f^n\}$ in $\cee_T$ such that $\lim_nf_r^n=f_r$ in $\chh$ for a.e. $r\in[s,T]$. For each $n$, if $f^n_r$ has the form \eqref{eqn:fcee}, then for every $r\in(s,T]$,
    \[
      U_{s,r}f^n_r=\sum_{i\in F}1_{I_i}(r)U_{s,r}h_i.
    \]
    In view of Condition \ref{con:set1}\ref{set1b}, the above identity implies that $U_{s,\cdot}f_{\cdot}^n(x):[s,T]\to\R$ is measurable for every $x\in\Rd$. For each $r\in(s,T]$, using the bound \eqref{con:U1}, we see that $\lim_nU_{s,r}f_r^n(x)=U_{s,r}f_r(x)$  for every $x\in\Rd$. This implies, in particular, that $U_{s,\cdot}f_\cdot(x):[s,T]\to\R$ is Lebesgue-measurable.
  \end{proof}

  \begin{lemma}\label{lem:UL1}
    Suppose that Condition \ref{con:set1} holds. Let $f$ be a function in $L^q_T\chh$, with $q\in[1,\infty]$ is such that 
    \begin{equation}\label{con.1136}
      (\nu,q)=(0,1)
      \quad\textrm{or}\quad
      \frac{\nu}2-\frac1q>-1\,.
    \end{equation}
    Then, for every $s\in[0,T]$ and $x\in\Rd$, the maps $[s,T]\ni r\mapsto U_{s,r}f_r(x)\in\Rd$ and $[s,T]\ni r\mapsto\|U_{s,r}f_r\|_\infty\in\R$ are Lebesgue-measurable and Lebesgue-integrable.
    In addition, for every $(s,t)\in[0,T]^2_\le$, we have
    \begin{equation}\label{est:intUf}
      % \sup_{x\in\Rd}|\int_{[s,t]}U_{s,r}f_r(x)dr|\le  
      \int_{[s,t]}\|U_{s,r}f_r\|_\infty dr\le c(\nu,q)\|U\|_{\chh,\nu}\|f1_{[s,t]}\|_{L^q_{T}\chh}(t-s)^{1+\frac{\nu}2-\frac1q}\,,
    \end{equation}
    where $c(\nu,q)=1$ if $q=1$ and $c(\nu,q)=(1+\frac{\nu}2 \frac{q}{q-1})^{\frac1q-1}$ otherwise.
  \end{lemma}
  \begin{proof}
    From Lemma \ref{lem.meas}, we see that  $[s,T]\ni r\mapsto U_{s,r}f_r(x)\in\Rd$ is Lebesgue-measurable for every $x\in\Rd$. Since for each $r>s$, $U_{s,r}f_r:\Rd\to\R$ is continuous, we have
    \[
      \|U_{s,r}f_r\|_\infty=\sup_{x\in\Q^d}|U_{s,r}f_r(x)|
    \]
    where $\Q$ is the set of rationals on $\R$. This implies that $\|U_{s,\cdot}f_\cdot\|_\infty:[s,T]\to\R$ is Lebesgue-measurable.

    From Condition \ref{con:set1}\ref{set1c}, we have
    \[
      \|U_{s,r}f_r\|_\infty\le\|U\|_{\chh,\nu}|r-s|^{\frac \nu2}\|f_r\|_\chh.
    \]
    When $q=1$, \eqref{con.1136} implies that $\nu=0$, so we have
    \[
      \int_{[s,t]}\|U_{s,r}f_r\|_\infty\le \|U\|_{\chh,\nu}\|f\|_{L^1_T\chh}.
    \]
    When $q>1$, we apply H\"older inequality to see that
    \[
      \int_{[s,t]}|r-s|^{\frac \nu2}\|f_r\|_\chh dr\le c(\nu,q)\|f1_{[s,t]}\|_{L^q_{T}\chh}(t-s)^{1+\frac \nu2-\frac1q}\,.
    \]
    The above argument shows that $\|U_{s,\cdot}f_\cdot\|_\infty$ is integrable and the estimate \eqref{est:intUf} holds.
    Obviously, for every $x\in\Rd$ and $r>s$, $|U_{s,r}f_r(x)|\le\|U_{s,r}f_r\|_\infty$, so that $U_{s,\cdot}f_\cdot(x)$ is also Lebesgue integrable. This completes the proof.
  \end{proof}
  For each $t>0$ and each simple function $f$ in $\cee^0_T$ having the form \eqref{eqn:fcee}, the integration $\int_{[0,t]}f_r(X_r)dr$ is defined by
  \begin{equation*}
    \int_{[0,t]}f_r(X_r)dr:=\sum_{i\in F}\int_{I_i\cap[0,t]}h_i(X_r)dr\,.
  \end{equation*}
  Since $X$ has measurable sample paths, each $h_i(X(\omega)_\cdot):[0,T]\to\R$ is a bounded measurable function and each summand in the right-hand side of the above identity is a Lebesgue integral for a.s. $\omega\in \Omega$.
  For a function $f$ in $L^\infty_T C_b(\Rd)$, one can find (see \cite[Corollary 1.1.21]{MR3617205}) a sequence of simple functions $\{f_n\}$ in $\cee^0_T$ such that for a.e. $r\in[0,T]$
  \[
    \|f_n(r,\cdot)\|_\infty\le\|f(r,\cdot)\|_\infty
    \quad\textrm{and}\quad
    \lim_n\|f_n(r,\cdot)- f(r,\cdot)\|_\infty =0.
  \]
  It follows that $f:[0,T]\times\Rd\to\R$ and $f(\cdot,X_\cdot):[0,T]\to\R$ are Lebesgue measurable. Consequently, the integral $\int_{[0,t]}f(r,X_r)dr$ is well-defined as a Lebesgue integral and
  \[
    \int_{[0,t]}f(r,X_r)dr=\lim_n\int_{[0,t]}f_n(r,X_r)dr
  \]
  by dominated convergence theorem.
  We are going to define a linear map $\caa^X$ on $L^q_T\chh$ such that $\caa^X$ is continuous on $L^q_T\chh$ and for every $t\in[0,T]$, $\caa^X_t[f]= \int_{[0,t]} f_r(X_r)dr$ whenever $f\in L^\infty_TC_b(\Rd)$. 
  The process $t\mapsto\caa^X_t[f]$ is a natural extension to the integral $\int_0^\cdot f_r(X_r)dr$ in situations when $f_r$ is not a function or when the composition $f_r(X_r)$ is a priori ill-posed.
 The main toolbox we employ here is the stochastic sewing lemma described in Section \ref{sec:stochastic_sewing_lemma}, more specifically Theorem \ref{lem:Sew2}.

  Let $f$ be an element in $L^q_T\chh$. For every $s\le t$, we put
  \begin{equation}\label{def.afU}
    A^X_{s,t}[f]=\int_{[s,t]} U_{s,r}f_r(X_s)dr
    \end{equation}
  Condition \ref{con:set1}\ref{set1c} ensures that $U_{s,r}f_r$ is a continuous function for each $s<r$, hence evaluating $U_{s,r}f_r$ at any point makes perfect sense and there is no ambiguity in \eqref{def.afU}.
  Furthermore, the estimate \eqref{est:intUf} implies that
  \begin{equation}\label{est:supa}
    \|A^X_{s,t}[f]\|_{L_m}\le c(\nu,q)\|U\|_{\chh,\nu}\|f1_{[s,t]}\|_{L^q_{T}\chh}(t-s)^{1+\frac \nu2-\frac 1q}\,.
  \end{equation}
  If $f$ is a bounded continuous function on $[0,T]\times\Rd$, we can write
  \begin{equation*}
    A^X_{s,t}[f]=\E^{\cff_s}\int_s^t f_r(X_r)dr\,.
  \end{equation*}
  For every triplet $s\le u\le t$, we have
  \begin{equation*}
    \delta A^X_{s,u,t}[f]=\int_{[u,t]} U_{s,r}f_r(X_s)dr-\int_{[u,t]}U_{u,r}f_r(X_u)dr\,.
  \end{equation*}
  Applying Fubini's theorem (which is eligible due to \eqref{est:intUf}) and Condition \ref{con:set1}\ref{set1EU}, we see that
  \begin{align*}
    \E^{\cff_s}\int_{[u,t]}U_{u,r}f_r(X_u)dr
    =\int_{[u,t]}\E^{\cff_s} U_{u,r}f_r(X_u)dr
    % =\int_{[u,t]}U_{s,u} U_{u,r}f_r(X_s)dr
    =\int_{[u,t]}  U_{s,r}f_r(X_s)dr.
  \end{align*} 
  It follows that
  \begin{equation}\label{id.Amarkov}
    \E^{\cff_s} \delta A^X_{s,u,t}[f]=0.
  \end{equation}
  In view of \eqref{est:supa} and the above identity, we see that the hypotheses of Theorem \ref{lem:Sew2} are satisfied provided that the exponent of $t-s$ in \eqref{est:supa} is at least a half. This leads to the following result.
  \begin{proposition}\label{prop:defA}
    Let $m\ge2$ be fixed and assume that Condition \ref{con:set1} is satisfied. Let $q$ be a number in $(2,\infty]$ such that
    \begin{equation}\label{con:nuq}
      \frac{\nu}2-\frac1q>-\frac12\,.
    \end{equation} 
    There exists a linear map $\caa^X$ from $L^q_T\chh$ to $C^{1+\frac{\nu}2-\frac1q}_TL_m$ which satisfies the following properties.
    \begin{enumerate}[label={\upshape(\alph*)}]
      \item\label{pa} For every $f\in L^q_T\chh$ and every $0\le s\le t\le T$, $\caa^X_t[f]$ is $\cff_t$-measurable,
      \[\E^{\cff_s}(\caa^X_t[f]-\caa^X_s[f])=\int_{[s,t]} U_{s,r}f_r(X_s)dr \quad \textrm{a.s.}\]
      and
      \begin{equation}\label{est:Ast}
        \|\caa^X_t[f]-\caa^X_s[f]\|_{L_m}\lesssim \|U\|_{\chh,\nu}\|f1_{[s,t]}\|_{L^q_T\chh}|t-s|^{1+\frac{\nu}2-\frac1q}\,.
      \end{equation}
       \item\label{pa2} For every $f\in L^q_T\chh$ and every $t$ in $[0,T]$, $\caa^X_t[f]$ is the limit in $L_m$ of the Riemann sum
      \begin{equation*}
        \sum_i \int_{[{t_i},{t_{i+1}}]}U_{t_i,r}f_r(X_{t_i})dr
      \end{equation*}
      as $\max_i|t_{i+1}-t_i|\to0$. Here $\{t_i\}_i$ is any finite partition of $[0,t]$.
    \end{enumerate}
  \end{proposition}
  \begin{proof}
    Let $f\in L^q_T\chh $ be fixed and define $A^X_{s,t}[f]$ as in \eqref{def.afU}. 
    From \eqref{con:nuq} and the estimate \eqref{est:supa}, we infer that $A^X[f]$ satisfies \eqref{con:A} with $\varepsilon_4=\frac12+\frac \nu2-\frac1q$.
    We apply Theorem \ref{lem:Sew2} to $A_{s,t}[f]$ to obtain an $\{\cff_t\}$-adapted and $L_m$-integrable process $\caa^X[f]$ which satisfies properties \ref{pa} and \ref{pa2}. 
    Inequality \eqref{est:Ast} also implies that $\caa^X[f]$ belongs to $C_T^{1+\frac \nu2-\frac1q}L_m$. 
    \end{proof}
  The next result connects $\caa^X[f]$ with the usual integration $\int_0^\cdot f(r,X_r)dr$ for a class of test functions.
  \begin{definition}\label{def.goodtestfunct}
    $\mathfrak{G}_T(X)$ is the set of all functions $f$ in $L^\infty_TC_b(\Rd)$ such that for a.e. $r\in[0,T]$ and a.s. $\omega$, $\E^{\cff_s} f_r(X_r(\omega))=U_{s,r}f_r(X_s(\omega))$  for every $s\in[0,r]$.
  \end{definition}
  \begin{proposition}\label{prop.K}
    Suppose that the hypotheses of Proposition \ref{prop:defA} are satisfied. 
    Then, $\caa^X_t[f]=\int_{[0,t]}f(r,X_r)dr$ a.s. for every $t\in[0,T]$ and every $f\in\fgg_T(X)$.
  \end{proposition}
  \begin{proof}
    By Fubini's theorem and the definition of $\fgg_T(X)$, for every $s\le t$, we have
    \begin{equation*}
      \E^{\cff_s}\int_{[s,t]}f_r(X_r)dr=\int_{[s,t]}\E^{\cff_s}f_r(X_r)dr=\int_{[s,t]}U_{s,r}f_r(X_s)dr \quad\textrm{a.s.}
    \end{equation*}
    It follows that
    \begin{equation*}
      \|\int_{[s,t]}f_r(X_r)dr-\int_{[s,t]}U_{s,r}f_r(X_s)dr\|_{L_m}\le 2|t-s|\esssup_{r\in[s,t]}\|f_r\|_\infty.
    \end{equation*}
    Hence, by Proposition \ref{prop:defA} and the uniqueness part of Theorem \ref{lem:Sew2}, $\caa^X_t[f]=\int_{[0,t]} f_r(X_r)dr $ a.s. for every $t\ge0$. 
    % This shows property \ref{pa1}.
  \end{proof}
 \begin{remark}\label{rem.intf} Let $f\in L^q_T\chh$ be fixed.

    (i) Since the process $\caa^X[f]$ is constructed from Theorem \ref{lem:Sew2}, it is the unique $\{\cff_t\}$-adapted and $L_m$-integrable process $\varphi$ such that $\varphi_0=0$, 
    \begin{equation*}
      \|\varphi_t- \varphi_s\|_{L_m}\lesssim |t-s|^{\frac12+\varepsilon}
      \quad\textrm{and}\quad
      \|\E^{\cff_s}(\varphi_t- \varphi_s-\int_{[s,t]}U_{s,r}f_r(X_s)dr)\|_{L_m}\lesssim|t-s|^{1+\varepsilon}
    \end{equation*}
    for every  $(s,t)$ in $[0,T]^2_\le$, for some $\varepsilon>0$. This provides a characterization of the process $\caa^X[f]$ for a given $f$.

    (ii) By choosing $m\ge2$ sufficiently large in \eqref{est:Ast} and applying Kolmogorov continuity theorem, we can find a modification of $\caa^X[f]$ which is a.s. $\tau$-H\"older continuous on $[0,T]$ for every $\tau\in(0,1+\frac \nu2-\frac1q)$.
    In the remaining of the section, we will always work with this continuous version of $\caa^X[f]$.
  \end{remark}
  \begin{corollary}\label{cor:Acont}
    Assume that the hypotheses of Proposition \ref{prop:defA} hold. Let $\tau$ be any number in $(\frac12,1+\frac \nu2-\frac1q)$.
    There exists a number $m_{\tau}\ge2$ such that
    for every $f\in L^q_T\chh$ and every sequence $\{f_n\}_n\subset \fgg_T(X)$ convergent to $f$ in $L^q_T\chh$, 
    \begin{equation}
      \lim_{n}\lt\|\big\|\int_{[0,\cdot]} f_{n}(r,X_r)dr-\caa^X_\cdot[f]\big\|_{C^\tau([0,T])} \rt\|_{L_m}=0
    \end{equation}
    for every $m\ge m_{\tau}$.
  \end{corollary}
  \begin{proof}
    From \eqref{est:Ast}, we have
    \begin{align*}
      \|\caa_t^X[f_{n}-f]-\caa_s^X[f_{n}-f]\|_{L_m}\lesssim \|f_n-f\|_{L^q_T\chh}|t-s|^{1+\frac \nu2-\frac1q}\,.
    \end{align*}
    for every $s\le t$. 
    Applying Garsia-Rodemich-Rumsey inequality, there exists $m_\tau\ge2$ such that for every $m\ge m_\tau$, we have
    \begin{align*}
      \|\|\caa_\cdot^X[f_{n}-f]\|_{C^\tau([0,T])}\| _{L_m}\lesssim \|f_n-f\|_{L^q_T\chh}.
    \end{align*}
    Observing that $\caa^X_t[f_n]=\int_{[0,t]}f_n(r,X_r)dr$ (by Proposition \ref{prop.K}), the previous estimate implies the result.
  \end{proof}
  
  It worths noting that the process $\caa^X_\cdot[f]$ depends on the trajectory of $X$, the element $f$ and the transition law of $(X,\{\cff_t\})$ (via the transition propagator $U$). Compare to the situation when $f$ is a continuous bounded function, $\int_0^\cdot f_r(X_r)dr$ only depends on the first two factors. 
  In the last part of the current section, we investigate the continuity dependence of  process $\caa^X_{\cdot}[f]$ with respect to the factors: the transition propagator corresponding to $(X,\{\cff_t\})$, the trajectory of $X$ and the element $f$. 
  % In the remaining of the current section, we establish continuity property of $\caa^X[f]$ with respect to these factors. 

  Let $X,Y$ be two Markov processes with respect to $\{\cff_t\}$ and let $U^X=\{U^X_{s,t}\}_{s\le t}$, $U^Y=\{U^Y_{s,t}\}_{s\le t}$ be two families of bounded operators on $\chh$. Assume that $(X,U^X,\chh)$ and $(Y,U^Y,\chh)$ satisfy Condition \ref{con:set1} with the same $\nu\le 0$.
   We denote by $\|U^X-U^Y\|_{\chh,\nu}$ the smallest number $M$ such that
    \begin{equation*}
      \|(U^X_{s,t}-U^Y_{s,t})\phi\|_\infty\le M \|\phi\|_\chh|t-s|^{\frac{\nu}2}
    \end{equation*}
  for every  $\phi\in\chh$ and $0\le s< t\le T$. From the bound \eqref{con:U1}, we see that $\|U^X-U^Y\|_{\chh,\nu}$ is necessarily finite. For every $\nu_1\le 0$ and every $h>0$, we define
  \begin{equation}\label{con:U3}
    \rho^X_{\nu_1}(h):=\sup_{\phi:\|\phi\|_\chh\le 1}\sup_{0\le s<t\le T} \sup_{|x-y|\le h}\frac{|U^X_{s,t}\phi(x)-U^X_{s,t}\phi(y) |}{|t-s|^{\frac{\nu_1}2}}\,.
  \end{equation}
  We note that by triangle inequality 
  \[
    \rho^X_{\nu_1}(h)\le 2\sup_{\phi:\|\phi\|_\chh\le 1}\sup_{0\le s<t\le T}\frac{\|U^X_{s,t}\phi\|_\infty}{|t-s|^{\frac{\nu_1}2}}
  \] 
  so that $\rho^X_{\nu_1}(h)$ is finite whenever $\nu_1\le \nu$.
  \begin{proposition}\label{prop:Lip}
    We assume that $(X,U^X,\chh)$ and $(Y,U^Y,\chh)$ satisfy the Condition \ref{con:set1} with the same $\nu\le 0$. Let $q\in(2,\infty]$ satisfy inequality \eqref{con:nuq} and let $\nu_1\le 0$ satisfy $\frac{\nu_1}2-\frac1q>-\frac12$.
    We assume that $\rho^X_{\nu_1}(h)$ is finite for every $h>0$.
    Then, for every $f,g\in L^q_T\chh$ and for every $(s,t)\in[0,T]^2_\le$,
    \begin{multline*}
      \|\caa^X_t[g]-\caa^X_s[g]-\caa^Y_t[f]+\caa^Y_s[f]\|_{L_m}
      \\\lesssim\(\|f-g\|_{L^q_T\chh}\|U^X\|_{\chh,\nu}+\|f\|_{L^q_T\chh}D_{\nu,\nu_1}(X,Y)\)|t-s|^{1+\frac{\nu\wedge \nu_1}2-\frac1q}\,,
    \end{multline*}
    where
    \begin{equation*}
      D_{\nu,\nu_1}(X,Y)=\|U^X-U^Y\|_{\chh,\nu}+\sup_{r\in[0,T]}\|\rho^X_{\nu_1}(|X_r-Y_r|)\|_{L_m}.
    \end{equation*}
   \end{proposition}
  \begin{proof}
    We apply Theorem \ref{lem:Sew2} to
    \begin{equation*}
      A_{s,t}:=A_{s,t}^X[g]-A_{s,t}^Y[f]= \int_{[s,t]} U^X_{s,r}g_r(X_s)dr-\int_{[s,t]} U^Y_{s,r}f_r(Y_s)dr\,.
    \end{equation*}
    Note that
    \begin{align*}
      A_{s,t}
      &=\int_{[s,t]}U_{s,r}^X(g_r-f_r)(X_s)dr
      +\int_{[s,t]} [U_{s,r}^Xf_r(X_s)-U_{s,r}^Xf_r(Y_s)]dr
      \\&\quad+\int_{[s,t]} [U_{s,r}^Xf_r(Y_s)-U_{s,r}^Yf_r(Y_s)]dr\,.
    \end{align*}
    We use \eqref{con:U1}, \eqref{est:supa} and \eqref{con:U3} to estimate these terms, which yields
    \begin{align*}
      |A_{s,t}|&\lesssim \|U^X\|_{\chh,\nu}\|g-f\|_{L^q_T\chh}|t-s|^{1+\frac \nu2-\frac1q}
      + \rho^X_{\nu_1}(|X_s-Y_s|)\|f\|_{L^q_T\chh}|t-s|^{1+\frac{\nu_1}2-\frac1q}
      \\&\quad+ \|U^X-U^Y\|_{\chh,\nu}\|f\|_{L^q_T\chh}|t-s|^{1+\frac{\nu}2-\frac1q}\,.
     \end{align*}
    This implies that
    \[
      \|A_{s,t}\|_{L_m}\lesssim \(\|f-g\|_{L^q_T\chh}\|U^X\|_{\chh,\nu}+\|f\|_{L^q_T\chh}D_{\nu,\nu_1}(X,Y)\)|t-s|^{1+\frac{\nu\wedge \nu_1}2-\frac1q}.
    \]
    Similar to \eqref{id.Amarkov}, we also have $\E^{\cff_s} \delta A_{s,u,t}=0$ for every triplet $s\le u\le t$. Since $1+\frac{\nu\wedge \nu_1}2-\frac1q>\frac12$ by our assumptions, Theorem \ref{lem:Sew2} is applicable and \eqref{est:A1'} yields the stated inequality.
  \end{proof}

% section distributive_functionals_of_martingales (end)

\section{Stochastic flows} % (fold)
\label{sec:stochastic_flows}
  Let $\alpha$ be a fixed number in $(0,1)$ and $b$ be a function in $[L^\infty([0,T]; C_b^\alpha(\Rd))]^{d}$. Given $x\in \Rd$, consider the stochastic differential equation (SDE)
  \begin{equation}\label{eqn:FGP}
    dX_t=b(t,X_t)dt+dW_t\,, \quad t\in[0,T]\,,\quad X_0=x\,,
  \end{equation}
  where $(W,\{\cff_t\})$ is a standard Brownian motion in $\Rd$.
  Let $X_t^x$ denote the solution to \eqref{eqn:FGP} starting from $X_0=x$ at time $t=0$ which is adapted to $\{\cff_t\}$.
  It is shown in \cite[Theorem 5(ii)]{MR2593276} that equation \eqref{eqn:FGP} has a $C^{1+\alpha'}$ - stochastic flow for any $\alpha'\in(0,\alpha)$. In particular, the process $Y_t^{ij}=\partial_{x_i} X_t^{j,x}$ is well-defined. (For simplicity, we omit the dependence of $x$ in $Y$.) Formally, $Y$ satisfies the following equations
  \begin{equation}\label{eqn:Y}
    Y_t^{ij}=\delta_{ij}+\sum_{k=1}^d\int_0^t \partial_k b^j(r,X^x_r)Y_r^{ik}dr\,,\quad\forall i,j\in\{1,\dots,d\}\,,
  \end{equation}
  where $\delta_{ij}$ is the Kronecker delta symbol.
  Since $\nabla b(r,\cdot)$ is only a distribution, the composition  $\nabla  b(r,X_r^x)$ is a priori ill-posed. Hence, equation \eqref{eqn:Y} is not mathematically rigorous and an equation describing the dynamic of $Y$ is missing in the literature.
  Filling this gap is the main purpose of the current section.
 
  Indeed, we can make sense of the process
  \begin{equation}\label{def.v.flow}
    V^{kj}_t=V^{kj}_t(b,X) =\caa^X_t[\partial_kb^j]%= \int_0^t \partial_k b^j(r,X_r)dr 
  \end{equation}
  for every $t\ge0$ and $j,k\in\{1,\dots,d\}$ by applying Proposition \ref{prop:defA}. Moreover, it turns out that the process $V^{kj}$ has a modification which belongs to $C^{\frac{1+\alpha'}2}([0,T])$ a.s. for every $\alpha'\in(0,\alpha)$. Then, a rigorous interpretation of the system of equations in \eqref{eqn:Y} is
  \begin{equation}\label{eqn:RSY}
    Y_t^{ij}=\delta_{ij}+\sum_{k=1}^d\int_0^t Y_r^{ki}dV^{kj}_r(b,X) \,,\quad\forall i,j\in\{1,\dots,d\},
  \end{equation}
  where the integrals on the right-hand side are Young integrals (\cite{young}). Our main result in the current section can be stated in the following theorem.
  \begin{theorem}\label{thm:eqnDX}
    Let $x\in\Rd$ be fixed but arbitrary. With probability one, the process $\partial_{x_i}X^{j,x} $ is the unique solution to the system of Young-type equations \eqref{eqn:RSY}. In addition,  the map $t\to \nabla X^x_t$ is a.s. $\frac{1+\alpha'}2$ - H\"older continuous for every $\alpha'\in(0,\alpha)$.
  \end{theorem}

  To show the above stated result, we first show that the process $V$ given by \eqref{def.v.flow} is well-defined and has a $\frac{1+\alpha'}2$-H\"older continuous modification. Then we show that $\partial_{x_i}X^{j,x} $ satisfies equation \eqref{eqn:RSY} using smooth approximations of $b$. The first step amounts in verifying the hypothesis of Proposition \ref{prop:defA}, Condition \ref{con:set1}.
  To pass through the limit in the second step, we rely on
   stability properties of Young-type differential equations (see e.g. \cite{MR3505229,MR2397797}) and of the functional $\caa^X$ (Corollary \ref{cor:Acont}).
 
  To verify Condition \ref{con:set1}, we study the transition propagator of $X$, which is denoted by $U=\{U_{s,t}\}_{s\le t}$. Each $U_{s,t}$, with $s<t$, takes values in $C_b(\Rd)$ and is defined first on $C^{2+\alpha}_b(\Rd)$, then extended to larger classes of bounded uniformly continuous functions and H\"older-Besov distributions.
  The definition of $U_{s,t}$ as an operator on $C^{2+\alpha}_b(\Rd)$ is a direct consequence of the regularity theory of Krylov and Priola in \cite{MR2748616} for the Kolmogorov backward parabolic differential equations associated to the SDE \eqref{eqn:FGP}. To extend the domain of $U_{s,t}$ to the larger classes of uniformly continuous functions and H\"older-Besov distributions, we rely on a maximum principle shown in Lemma \ref{lem.maximum} and a priori estimates in Lemma \ref{lem:sch.flow} showing that solutions to Kolmogorov backward equations are more regular away from the terminal time.  

  Let us proceed with more detail. For each $t\in[0,T]$, $f\in L^\infty([0,T];C^\alpha_b(\Rd))$ and $g\in C^{2+\alpha}_b(\Rd)$, consider the Kolmogorov backward equation
  \begin{equation}\label{eqn:BK}
    \(\partial_s +\frac12\Delta +b\cdot\nabla \)u(s,x)=f\quad\forall (s,x)\in[0, t]\times\Rd\,,  \quad u(t,\cdot)=g(\cdot)\,.
  \end{equation}
  Since the coefficients $b,f$ are only measurable in time, the notion of solutions to \eqref{eqn:BK} is non-standard. As in \cite{MR2593276,MR2748616}, a function $u^t:[0,t]\times \Rd\to\R$ is a solution to \eqref{eqn:BK} if $u^t$ belongs to $L^\infty([0,t];C^{2+\alpha}_b(\Rd))$ and satisfies 
  \begin{equation}\label{def.KPsol}
    u_t^t(x)=g(x)
    \quad\textrm{and}\quad
    u_s^t(x)-u_r^t(x)=\int_r^s [-L^b_\theta u^t_\theta(x)+f_\theta(x)]d \theta
  \end{equation}
  for every $0\le r\le s\le t$, $x\in\Rd$. Here for every $\theta\in[0,T]$, $L^b_\theta=\frac12 \Delta+b_\theta\cdot\nabla$.

  It is shown in \cite[Theorem 2.8 and Remark 2.9]{MR2748616} that equation \eqref{eqn:BK} has unique solution in $L^\infty([0,t];C^{2+\alpha}_b(\Rd))$ for every $f\in L^\infty([0,T];C^\alpha_b(\Rd))$ and $g\in C^{2+\alpha}_b(\Rd)$. In addition, for this solution, we have
  \begin{equation}\label{est.KP}
    \sup_{s\in[0,t]}\|u_s^t\|_{C^{2+\alpha}_b}\le C(\alpha,d)(\|f\|_{L^\infty C^\alpha_b}+\|g\|_{C^{2+\alpha}_b}).
  \end{equation}

  We define for every $s\le t$ and $x\in\Rd$, $U_{s,t}g(x)=u^t_s(x)$, where $u^t$ is the solution \eqref{eqn:BK} with $f\equiv0$. Then $\{U_{s,t}\}_{s\le t}$ is a propagator on $C^{2+\alpha}_b(\Rd)$.
  This means that $U_{s,t}$ is a bounded linear operator on $C^{2+\alpha}_b(\Rd)$ and 
  \begin{equation}\label{def.propagator}
    U_{s,t}g(x)=U_{s,r}U_{r,t}g(x)
  \end{equation} 
  for every $s\le r\le t$, $x\in\Rd$ and $g\in C^{2+\alpha}_b(\Rd)$. 
  The boundedness of $U_{s,t}$ as a linear operator on $C^{2+\alpha}_b(\Rd)$ follows from the bound \eqref{est.KP} and the identity \eqref{def.propagator} follows from the fact that equation \eqref{eqn:BK} is uniquely solvable with terminal data in $C^{2+\alpha}_b(\Rd)$.

  At this point, we reason that $\{U_{s,t}\}$ can be extended to a propagator on $BUC(\Rd)$, the Banach space of all bounded uniformly continuous real functions on $\Rd$ equipped with the supremum norm $\|\cdot\|_\infty$. This requires a maximum principle for equation \eqref{eqn:BK}.
  In \cite[pages 10,11]{MR2593276} and in \cite[Theorem 4.1]{MR2748616}, a maximum principle is shown for second-order differential operators with negative potentials. These results can not be applied to equation \eqref{eqn:BK} with a null potential. Nevertheless, the following It\^o formula, shown in \cite{MR2593276}, is useful for our considerations.
  \begin{lemma}[{\cite[Lemma 3]{MR2593276}}]\label{lem.itoFPG}
    Let $u:[0,T]\times\Rd\to\R$ be a function in $L^\infty([0,T];C^{2+\alpha}_b(\Rd))$ such that
    \[
      u_t(x)-u_s(x)=\int_s^t v_r(x)dr
    \]
    for every $s\le t$, $x\in\Rd$ with $v\in L^\infty([0,T];C^\alpha_b(\Rd))$. Let $(X_t)_{t\ge0}$ be a continuous adapted process of the form
    \[
      X_t=X_0+\int_0^t b_sds+W_t
    \]
    where $b$ is a progressively measurable process, $b$ is integrable in $t$ with probability one. Then
    \begin{align*}
      u_t(X_t)=u_s(X_s)+\int_s^t(v_r+b_r\cdot\nabla u_r+\frac12 \Delta u_r)(X_r)dr+\int_s^t\nabla u_r(X_r)\cdot dW_r.
    \end{align*}
  \end{lemma}
    \begin{lemma}[Maximum principle]\label{lem.maximum} For every $f\in L^\infty([0,T];C^{\alpha}_b(\Rd))$, $g\in C^{2+\alpha}_b(\Rd)$ and every $s\le t$,
    \begin{equation}\label{id.FK}
      u^t_s(x)=\E g(X^x_{s,t})-\E\int_s^tf_r(X^x_{s,r})dr,
    \end{equation}
    and
    \begin{equation}\label{est.maxPrinci}
      \|u^t_s\|_\infty\le \|g\|_\infty+(t-s)\|f\|_{L^\infty C_b}.
    \end{equation}
    Here, $u^t$ is the solution to \eqref{eqn:BK}.
  \end{lemma}
  \begin{proof}
    Let $s\le t$ be fixed
    % , $f$ be in $L^\infty([0,T];C^{2+\alpha}_b(\Rd))$, $g$ be in $C^{2+\alpha}_b(\Rd)$ 
    and $(X^x_{s,r})_{r\ge s}$ be the solution to \eqref{eqn:FGP} starting from time $s$ at $X^x_{s,s}=x$. 
    Let $u^t$ be the solution to \eqref{eqn:BK}. 
    Applying the It\^o formula in Lemma \ref{lem.itoFPG} to $r\mapsto u^t_r(X^x_{s,r})$, using the relation \eqref{def.KPsol}, we have
    \begin{equation*}
      g(X^x_{s,t})=u^t_s(x)+\int_s^tf_r(X^x_{s,r})d r +\int_s^t\nabla u^t_r(X^x_{s,r})\cdot dW_r.
    \end{equation*}
    By \eqref{est.KP}, the stochastic integral above is a square integrable martingale.
    Taking expectation, we obtain \eqref{id.FK}, which is a probabilistic presentation for the solution to \eqref{eqn:BK}.
    % By approximation, \eqref{}
    The formula \eqref{id.FK} implies the estimate \eqref{est.maxPrinci}.
     % for every $s\le t$ and $g\in C^{2+\alpha}_b(\Rd)$.
    % By approximation, \eqref{est.maxPrinci} holds for every $s\le t$ and $g\in BUC(\Rd)$.
  \end{proof}
  As a consequence of \eqref{id.FK}, we have for every $g\in C^{2+\alpha}_b(\Rd)$,
  \begin{equation}\label{id.UFK}
    U_{s,t}g(x)=\E g(X^x_{s,t}).
  \end{equation}
  This implies that for every $s\le t$ and $g\in C^{2+\alpha}_b(\Rd)$,
  \begin{equation}\label{est.Umax}
    \|U_{s,t}g\|_\infty\le \|g\|_\infty.
  \end{equation}
  We note that $C^\gamma_b(\Rd)$ is dense in $BUC(\Rd)$ for every $\gamma>0$.
  By approximation, the above inequality holds for every $s\le t$ and $g\in BUC(\Rd)$. This means that $U_{s,t}$ can be extended to a bounded operator on $BUC(\Rd)$ for each $s\le t$. The family $\{U_{s,t}\}$ then forms a $C_0$-propagator on $BUC(\Rd)$. The precise meaning of this statement is the content of the following result.
  \begin{proposition}\label{prop.Uc0propa}
    The two-parameter family $\{U_{s,t}\}_{s\le t}$ satisfies
    \begin{enumerate}[label={\upshape(\alph*)}]
      \item\label{u1} $U_{s,t}$ is a contraction on $BUC(\Rd)$ for every $s\le t$,
      \item\label{u2} $U_{s,t}=U_{s,r}U_{r,t}$ as operators on $BUC(\Rd)$ for every $s\le r\le t$,
      \item\label{u3} $\lim_{h\to0}\sup_{s,t:|t-s|\le h}\|U_{s,t}g-g\|_\infty=0$ for every $g\in BUC(\Rd)$.
    \end{enumerate}
  \end{proposition}
  \begin{proof}
    \ref{u1} is shown in the discussion prior to the proposition. Let $g$ be in $BUC(\Rd)$ and $\{g^n\}_n$ be a sequence in $C^\alpha_b(\Rd)$ which converges to $g$ in $BUC(\Rd)$. From \eqref{est.Umax}, it follows that $\lim_nU_{s,t}g^n=U_{s,t}g$ in $BUC(\Rd)$. From \eqref{def.propagator}, we have $U_{s,t}g^n=U_{s,r}U_{r,t}g^n$ for each $n$. Sending $n\to\infty$, using the continuity of $U_{s,r}$ and $U_{r,t}$ on $BUC(\Rd)$, yields $U_{s,t}g=U_{s,r}U_{r,t}g$, which implies \ref{u2}. To show \ref{u3}, let $u^{t,n}$ be the solution to \eqref{eqn:BK} with $f\equiv0$ and terminal condition $u^{t,n}_t=g^n$. We note that by triangle inequality, contraction property \ref{u1} and definition of $U$,
    \begin{align*}
      \|U_{s,t}g-g\|_\infty
      &\le\|U_{s,t}g^n-g^n\|_\infty+\|U_{s,t}g^n-U_{s,t}g\|_\infty+\|g^n-g\|_\infty
      \\&\le\|u^{t,n}_s -g^n\|_\infty+2\|g^n-g\|_\infty.
    \end{align*} 
    It follows directly from \eqref{def.KPsol} and \eqref{est.KP}
    % that $s\mapsto u^{t,n}_s(x)$ is continuous for each $x\in\Rd$ and 
    that
    \begin{align*}
      \|u^{t,n}_s-g^n\|_\infty
      \le\int_s^t\|L^b_\theta u^{t,n}_\theta\|_\infty d \theta
      \lesssim (t-s)\|u^{t,n}\|_{L^\infty C^{2}_b}\lesssim (t-s)\|g^n\|_{C^{2+\alpha}_b} .
    \end{align*}
    Combining the previous estimates we obtain that
    \[
      \limsup_{h\to0}\sup_{|t-s|\le h}\|U_{s,t}g-g\|_\infty\le2\|g^n-g\|_\infty
    \]
    for each $n$. Finally, sending $n\to\infty$ we obtain \ref{u3}.
  \end{proof}
  \begin{remark}
    Let $\{Q_t\}_{t\ge0}$ be a family of bounded operators on $BUC(\Rd)$.
    It is straightforward to verify that $\{U_{s,t}\}_{s\le t}:=\{Q_{t-s}\}_{s\le t}$ satisfies the properties \ref{u1}-\ref{u3} listed above if and only if $\{Q_t\}_{t\ge0}$ is a contraction $C_0$-semigroup. This justifies our terminology of $C_0$-propagator.
  \end{remark}

  We show that $\{U_{s,t}\}$ is the transition propagator of $X^x$, the solution of \eqref{eqn:FGP}, in the following sense. 
  \begin{lemma}\label{lem.Upropa} For every $g$ in $BUC(\Rd)$ and every $s\le t$
    \begin{equation*}
      \E^{\cff_s}g(X^x_{t})=U_{s,t}g(X^x_{s}).
    \end{equation*}
  \end{lemma}
  \begin{proof}
    In view of \eqref{est.Umax}, it suffices to show the result for $g$ in $C^{2+\alpha}_b(\Rd)$. Applying It\^o formula in Lemma \ref{lem.itoFPG} to $r\mapsto U_{r,t}g(X^x_{r})$
    \begin{align*}
      g(X^x_{t})=U_{s,t}g(X^x_{s})+\int_s^t\nabla U_{r,t}(X^x_{r})\cdot dW_r
    \end{align*}
    Taking conditional expectation with respect to $\cff_s$, noting that $\nabla U_{r,t}$ is uniformly bounded (by \eqref{est.KP}), we obtain the claim.
  \end{proof}
  
  Next, we show that $U_{s,t}$ maps $\C^{\beta-1}(\Rd)$ with $\beta\in(0,1]$, to $BUC(\Rd)$ provided that $s<t$. For concreteness, we first describe this procedure for the heat propagator. Let $\{P_{s,t}\}$ be the propagator associated to $\frac 12 \Delta$. This means that for each $s<t$, $g\in BUC(\Rd)$ and $x\in\Rd$, $P_{s,s}g=g$ and
  \[
    P_{s,t}g(x):=\E g(W_t-W_s+x)=(2 \pi(t-s))^{-\frac d2}\int_\Rd e^{-\frac{|x-y|^2}{2(t-s)}}g(y)dy .
  \]
  Let $s<t$ be fixed. We aim to extend the domain of $P_{s,t}$ to H\"older-Besov space $\C^{\beta-1}(\Rd)$ without appealing to the explicit formula of the kernel. This approach has the advantage that it can be carried out for $U_{s,t}$ once certain Schauder estimates are available. We proceed by approximation. Let $g$ be in $\C^{\beta-1}(\Rd)$ and let $\{g^n\}_n$ be a mollifying sequence of $g$. This means that $g^n(x)=\wei{g(x-\cdot),\rho_{n}(\cdot)}$, $\rho_{n}(y)=n^{d}\rho(ny)$, $\rho$ is a smooth non-negative function $\rho$ which equals $1$ on $\{x\in\Rd:|x|\le1\}$ and vanishes outside $\{x\in\Rd:|x|\le 2\}$.
  H\"older-Besov spaces are unfortunately not separable. One cannot expect that $\{g^n\}_n$ converges to $g$ in $\C^{\beta-1}(\Rd)$. Nevertheless, it is true that $\lim g^n=g$ in $\C^{\beta'-1}$ for every $\beta'<\beta$. From \eqref{P.schau}, we have
  \begin{equation}\label{tmp.Pgnk}
    \|P_{s,t}g^n-P_{s,t}g^k\|_\infty\lesssim|t-s|^{\frac {\beta'}2-\frac 12}\|g^n- g^k\|_{\C^{\beta'-1}}
  \end{equation}
  for every $n,k$.  This implies that $\lim_n P_{s,t}g^n$ exists in $BUC(\Rd)$ and we denote the resulting limit by $P_{s,t}g$, which is a function in $BUC(\Rd)$. 
  % In terms of Schwarz pairing, $P_{s,t}g^n(x)=\wei{p_{t-s}(x-\cdot),g^n}$ so we have $P_{s,t}g(x)=\wei{p_{t-s}(x-\cdot),g}$. 
  Using \eqref{P.schau} once again, we see that
  \[
    \|P_{s,t}g^n\|_\infty
    \lesssim |t-s|^{\frac \beta2-\frac12}\|g^n\|_{\C^{\beta-1}}
    \lesssim |t-s|^{\frac \beta2-\frac12}\|g\|_{\C^{\beta-1}}.
  \]
  Sending $n\to\infty$, we obtain
  \[
    \|P_{s,t}g\|_\infty
    \lesssim |t-s|^{\frac \beta2-\frac12}\|g\|_{\C^{\beta-1}}.
  \]
  The extension of $U_{s,t}$ on $\C^{\beta-1}(\Rd)$ is carried out through the same procedure. To obtain bounds for $U_{s,t}$ similar to \eqref{tmp.Pgnk}, we rely on some a priori estimates for the solutions of \eqref{eqn:BK} obtained through the mild formulation.

  \begin{lemma}[Mild formulation]
    For every $f\in L^\infty([0,T];C^\alpha_b(\Rd))$, $g\in C^{2+\alpha}_b(\Rd)$ and every $s\le t$, $x\in\Rd$
    \begin{equation}\label{id.mild}
      u^t_s(x)= P_{s,t}g(x)+\int_s^t P_{s,r}(b_r\cdot \nabla u^t_r-f_r)(x)dr.
    \end{equation}
  \end{lemma}
  \begin{proof}
    Let $(W_r)_{r\ge0}$ be a standard Brownian motion in $\Rd$ and let $s\le t$ be fixed. Define $X_r=W_r-W_s+x$ for every $r\ge s$. We note that $(W_r-W_s)_{r\ge s}$ is a standard Brownian motion starting at time $s$. Applying the It\^o formula in Lemma \ref{lem.itoFPG} and using \eqref{def.KPsol}, we have
    \begin{align*}
      u^t_t(X_t)=u^t_s(X_s)+\int_s^t(f_r-b_r\cdot\nabla u^t_r)(X_r)dr+\int_s^t\nabla u^t_r(X_r)\cdot dW_r.  
    \end{align*}
    By \eqref{est.KP}, the stochastic integral above is square integrable. Hence, we can take expectation in the above formula and use the identities $u^t_t(X_t)=g(W_t-W_r+x)$, $u^t_s(X_s)=u^t_s(x)$ to obtain that
    \[
      u^t_s(x)=\E g(W_t-W_s+x)+\E\int_s^t (b_r\cdot \nabla u^t_r-f_r)(W_r-W_s+x)dr.
    \] 
    By Fubini's theorem, we can interchange the expectation and the integration in $dr$ in the last term of the above identity. The identity \eqref{id.mild} follows by definition of $P$.    
  \end{proof}
  \begin{lemma}\label{lem:sch.flow}
    Let $g$ be in $C^{2+\alpha}_b(\Rd)$, $f$ be in $L^\infty([0,t]; C^\alpha_b(\Rd))$ and $u^t$ be the solution to \eqref{eqn:BK} with terminal condition $g$ at time $t$ and right-hand side $f$.
    Let $\beta$ be a fixed number in $(0,1]$.
    We define for each $s\in[0,t]$,
    \begin{equation}\label{def.F}
      F^{(\beta)}_s=(t-s)^{1-\frac \beta2} \int_s^t(r-s)^{-\frac12}\|f_r\|_\infty dr
      \quad\textrm{and}\quad
      F^{(\beta),*}_s=\sup_{r\in[s,t]}F^{(\beta)}_r.
    \end{equation}
    Then, we have for every $s\in[0,t)$
    \begin{equation}\label{est.gradU}
      \|\nabla u^t_{s}\|_{\infty}\le C |t-s|^{\frac \beta2-1} \left(\|g\|_{\C^{\beta-1}}+F^{(\beta),*}_s\right)
    \end{equation}
    and
    \begin{equation}\label{est.supU}
      \| u^t_{s}\|_{\infty}\le C|t-s|^{\frac \beta2-\frac12} (\|g\|_{\C^{\beta-1}}+F^{(\beta)}_s)+C(t-s)^{\frac \beta2}F^{(\beta),*}_s
    \end{equation}
    for a constant $C=C(\|b\|_{L^\infty C_b},T,\beta)$.
    In particular, 
    \begin{equation}\label{est.Uschauder}
      \|\nabla U_{s,t}g\|_\infty\le C|t-s|^{\frac \beta2-1}\|g\|_{\C^{\beta-1}}
      \quad\textrm{and} \quad 
      \|U_{s,t}g\|_{\infty}\le C|t-s|^{\frac \beta2-\frac12}\|g\|_{\C^{\beta-1}}
    \end{equation}
    for every $s<t$.
    \end{lemma}
  \begin{proof} From \eqref{id.mild}, we have
    \begin{equation*}
      \|\nabla u^t_s\|_\infty\le\|\nabla P_{s,t}g\|_\infty+\int_s^t\|\nabla P_{s,r}f_r\|_\infty dr+ \int_s^t\|\nabla P_{s,r}(b_r\cdot\nabla u^t_r)\|_\infty dr.
    \end{equation*}
    From \eqref{P.schau} and direct estimations, we know that
    \[
      \|\nabla P_{s,t}g\|_\infty\lesssim|t-s|^{\frac \beta2-1}\|g\|_{\C^{\beta-1}}
      ,\quad
      \|\nabla P_{s,r}f_r\|_\infty\lesssim|r-s|^{-\frac12}\|f_r\|_\infty
    \]
    and
    \[
      \|\nabla P_{s,r}(b_r\cdot\nabla u^t_r) \|_\infty
      \lesssim|r-s|^{-\frac12}\|b_r\cdot\nabla u^t_r \|_\infty
      \lesssim|r-s|^{-\frac12}\|b\|_{L^\infty C_b} \|\nabla u^t_r \|_\infty.
    \]
    It follows that
    \begin{align*}
      \|\nabla u^t_s\|_\infty\lesssim|t-s|^{\frac \beta2-1}(\|g\|_{\C^{\beta-1}}+F^{(\beta)}_s) +\|b\|_{L^\infty C_b} \int_s^t|r-s|^{-\frac12}\|\nabla u^t_r\|_\infty dr.
    \end{align*}
    Let $\lambda$ be a positive constant to be chosen later. For every $s\in[0,t]$, we define
    \[
      m_s=e^{-\lambda(t-s)} (t-s)^{1-\frac \beta2}\|\nabla u^t_s\|_\infty
      \quad\textrm{and}\quad
      m^*_s=\sup_{r\in[s,t]} m_r,
    \]
    which are finite by \eqref{est.KP}.
    The previous estimate implies that
    \begin{align*}
      m_s\lesssim\|g\|_{\C^{\beta-1}}+F^{(\beta)}_s+m^*_s\|b\|_{L^\infty C_b}(t-s)^{1-\frac \beta2}\int_s^t|r-s|^{-\frac12}|t-r|^{\frac \beta2-1}e^{-\lambda(r-s)}dr .
    \end{align*}
    To estimate the integral on the right-hand side, we split it into two regions
    \begin{align*}
      \int_s^{\frac{s+t}{2} }|r-s|^{-\frac12}|t-r|^{\frac \beta2-1}e^{-\lambda(r-s)}dr
      &\lesssim|t-s|^{\frac\beta2-1}\int_s^{\infty}(r-s)^{-\frac12}e^{-\lambda(r-s)}dr
      \\&\lesssim |t-s|^{\frac \beta2-1}\lambda^{-\frac12}
    \end{align*}
    and
    \begin{align*}
      \int_{\frac{s+t}{2}}^t|r-s|^{-\frac12}|t-r|^{\frac \beta2-1}e^{-\lambda(r-s)}dr
      &\lesssim |t-s|^{-\frac12}e^{-\frac \lambda2 (t-s)}\int_{\frac{s+t}{2}}^t(t-r)^{\frac \beta2-1}dr
      \\&\lesssim |t-s|^{\frac \beta2-\frac12}e^{-\frac \lambda2(t-s)}.
    \end{align*}
    Noting that $e^{-\frac \lambda2(t-s)}\lesssim \lambda^{-\frac12}(t-s)^{-\frac12}$, the previous two bounds yields
    \[
      \int_s^t|r-s|^{-\frac12}|t-r|^{\frac \beta2-1}e^{-\lambda(r-s)}dr\lesssim |t-s|^{\frac \beta2-1}\lambda^{-\frac12}.
    \]
    Hence, we have
    \[
      m_s\lesssim\|g\|_{\C^{\beta-1}}+F^{(\beta)}_s+m^*_s\|b\|_{L^\infty C_b} \lambda^{-\frac12}
    \]
    and 
    \[
      m^*_s\lesssim\|g\|_{\C^{\beta-1}}+F^{(\beta),*}_s+m^*_s\|b\|_{L^\infty C_b} \lambda^{-\frac12}
    \]
    for every $s\in[0,t]$.
    From here, choosing $\lambda$ sufficiently large, we obtain that $m^*_s\le C(\|g\|_{\C^{\beta-1}}+F^{(\beta),*}_s)$ for every $s\in[0,t]$ for a constant $C>0$ depending only on $\beta$ and $\|b\|_{L^\infty C_b}$. This implies \eqref{est.gradU}.

    To show \eqref{est.supU}, we derive from \eqref{id.mild} that
    \[
      \|u^t_s\|_{\infty}\le\|P_{s,t}g\|_\infty+\int_s^t\|P_{s,r}f_r\|_\infty dr+ \int_s^t\|P_{s,r}(b_r\cdot\nabla u^t_r)\|_\infty dr.
    \]
    We estimate the first term by
    \[
      \|P_{s,t}g\|_\infty\lesssim|t-s|^{\frac \beta2-\frac12}\|g\|_{\C^{\beta-1}},
    \]
    the second term by
    \[
      \int_s^t\|P_{s,r}f_r\|_\infty dr\le\int_s^t\|f_r\|_{\infty}dr\le(t-s)^{\frac \beta2-\frac12}F^{(\beta)}_s ,
    \]
    and the last term by
    \begin{align*}
      \|P_{s,r}(b_r\cdot\nabla u^t_r)\|_\infty\le\|b_r\|_\infty\|\nabla u^t_r\|_\infty
      \le C(\|b\|_{L^\infty C_b})(\|g\|_{\C^{\beta-1}}+F^{(\beta),*}_s)|t-r|^{\frac \beta2-1}
    \end{align*}
    in which we have used \eqref{est.gradU} to obtain the last inequality. Combining these bounds together, we obtain \eqref{est.supU}.
  \end{proof}
  % Next, we show that $U_{s,t}$ maps $\C^{\alpha-1}$ to $BUC(\Rd)$. 
  \begin{lemma}\label{lem.UCbeta}
    For every $s<t$, the operator $U_{s,t}:BUC(\Rd)\to BUC(\Rd)$ can be extended uniquely to a bounded linear operator $U_{s,t}:\C^{\beta-1}(\Rd)\to BUC(\Rd)$.
    In addition, there exists a positive constant $C=C(\|b\|_{L^\infty C_b},T,\beta)$ such that
    \begin{equation}\label{est.Ua1}
      \|U_{s,t}g\|_\infty\le C \|g\|_{\C^{\beta-1}}|t-s|^{\frac{\beta-1}2}
    \end{equation}
    for every $g\in\C^{\beta-1}(\Rd)$.
  \end{lemma}
  \begin{proof}
    Let $\{g^n\}_n$ be a sequence in $C^{3}_b(\Rd)$ such that $g^n$ converges to $g$ in $\C^{\beta'-1}$ for some $\beta'\in(0,\beta)$ and that $\sup_n\|g^n\|_{\C^{\beta-1}}\lesssim\|g\|_{\C^{\beta-1}}$. 
    Let $s,t$ be fixed but arbitrary in $[0,T]$, $s<t$. 
    From \eqref{est.Uschauder}, we have for every $n,k$
    \[
      \|U_{s,t}g^n-U_{s,t}g^k\|_\infty\lesssim |t-s|^{\frac{\beta'}2-\frac12}\|g^n-g^k\|_{\C^{\beta'-1}}.
    \]
    This implies that $\lim_n U_{s,t}g^n$ exists in $BUC(\Rd)$ and the limit is denoted by $U_{s,t}g$. It is standard to  check that $U_{s,t}g$ is independent from the choices of the approximating sequence $\{g^n\}_n$ and that $U_{s,t}:\C^{\beta-1}\to BUC(\Rd)$ is linear. In view of the above estimate, such extension is necessary unique in the sense that if $\bar U_{s,t}:\C^{\beta-1}(\Rd)\to BUC(\Rd)$ is another continuous extension of $U_{s,t}:BUC(\Rd)\to BUC(\Rd)$, then $U_{s,t}=\bar U_{s,t}$ on $\C^{\beta-1}(\Rd)$. In addition, from \eqref{est.Uschauder}, we have
    \[
      \|U_{s,t}g\|_\infty=\lim_n\|U_{s,t}g^n\|_\infty\lesssim|t-s|^{\frac \beta2-\frac12}\limsup_n\|g^n\|_{\C^{\beta-1}}\lesssim|t-s|^{\frac \beta2-\frac12}\|g\|_{\C^{\beta-1}},
    \]
    which shows \eqref{est.Ua1}.
   \end{proof}
  \begin{proposition}\label{prop.U31}
    For every $\beta$ in $(0,1)$,
    $(X,U,\C^{\beta-1}(\Rd))$ satisfies Condition \ref{con:set1} with $\nu=\beta-1$.
  \end{proposition}
  \begin{proof}
    Let $\chh=\C^{\beta-1}(\Rd)$ and $\nu=\beta-1$.
    It is evident that Lemma \ref{lem.UCbeta} implies property \ref{set1c} of Condition \ref{con:set1}. Let $h$ be in $\C^{\beta-1}(\Rd)$ and let $\{h^n\}_n$ be a sequence in $C^3_b(\Rd)$ which converges to $h$ in $\C^{\beta'-1}(\Rd)$ for some $\beta'\in(0,\beta)$. For each $s\ge0$ and $n$, it follows from \eqref{id.UFK} that  $t\mapsto U_{s,t}h^n(x)$ is measurable. Since $\lim_n U_{s,t}h^n=U_{s,t}h$ in $BUC(\Rd)$ for every $t>s$, this implies the property \ref{set1b} of Condition \ref{con:set1}. Finally, since for $r<t$, $U_{r,t}h$ belongs to $BUC(\Rd)$ (by Lemma \ref{lem.UCbeta}), we derive from Lemma \ref{lem.Upropa} and Proposition \ref{prop.Uc0propa}\ref{u2} that $\E^{\cff_s}U_{r,t}h=U_{s,r}U_{r,t}h=U_{s,t}h$ for every $s\le r<t$. This shows property \ref{set1EU} of Condition \ref{con:set1}.
  \end{proof}

   Note that $\nabla b$ belongs to $[L^\infty_T\C^{\alpha-1}(\Rd;\Rd)]^d$ and $(\nu,q)=(\alpha-1,\infty)$ satisfies condition \eqref{con:nuq}. Applying Proposition \ref{prop:defA}, we can define the process $V$ in \eqref{def.v.flow}. The estimate \eqref{est:Ast} implies that
  \begin{equation*}
    \|V_t(b,X)-V_s(b,X)\|_{L_m}\lesssim |t-s|^{\frac{\alpha+1}2}
  \end{equation*}
  for every $s<t$ and $m\ge2$. By choosing $m$ sufficiently large and applying Kolmogorov continuity theorem, we can find a continuous version of $V$ which is a.s. $\frac{\alpha'+1}2$-H\"older continuous  for every $\alpha'\in(0,\alpha)$.

  Let $\{b^n\}_n$ be a sequence of functions in $L^\infty([0,T];C^{4}_b(\Rd;\Rd))$ which converges to $b$ in $L^\infty([0,T];C^{\beta}(\Rd;\Rd))$ for every $\beta\in(0,\alpha)$. Let $X^n, U^n$ denote respectively the solution to \eqref{eqn:FGP} and the transition semigroup defined in \eqref{eqn:BK} with $b^n$ in place of $b$.
  From \cite[Theorem 5]{MR2593276}, for every $m\ge2$, we have the following stability estimates
  \begin{gather}\label{est:FGP}
    \lim_n\sup_{x\in\Rd}\E\sup_{r\in[0,T]}|X^n_r(x)-X_r(x)|^m=0\,,
    % \\\sup_n\sup_{x\in\Rd}\E\sup_{r\in[0,T]}|\nabla X^n_{r}(x)|^m<\infty
    \\\lim_n\sup_{x\in\Rd}\E\sup_{r\in[0,T]}|\nabla X^n_{r}(x)-\nabla X_{r}(x)|^m=0\,.
    \label{est:FGP2}
  \end{gather}

  \begin{lemma}\label{lem:Vconvg}
    For every $\alpha'\in(0,\alpha)$,  $\{\int_{[0,\cdot]} \nabla b^n_r(X^n_r)dr\}_n$ converges to $V(b,X) $ in $C^{\frac{1+\alpha'}2}([0,T])$ in probability.
  \end{lemma}
  \begin{proof}
    Let $j,k$ be fixed indices in $\{1,\dots,d\}$. Put $ f=\partial_k b^j\in L^\infty_T \C^{\alpha-1}(\Rd)$, $f^n=\partial_k b^{n,j}$ (the $x_k$-partial derivative of the $j$-th component of $b^n$). 
    Let $\alpha'$ and $\beta$ be positive numbers such that $\alpha'<\beta<\alpha$.
    We note that each $f^n$ belongs to $\fgg_T(X)$ (by Lemma \ref{lem.Upropa} and recalling the Definition \ref{def.goodtestfunct}) and that $\{f^n\}_n$ converges to $f$ in $L^\infty_T\C^{\beta-1}$. In addition, $(X,U^X,\C^{\beta-1}(\Rd))$ satisfies Condition \ref{con:set1} with  $\nu=\beta-1$ (Proposition \ref{prop.U31}). Hence,  we can apply Corollary \ref{cor:Acont} to see that $\{\int_{[0,\cdot]} f^n_r(X_r)dr\}_n$ converges to $\caa_\cdot^X[f]$ in $C^{(1+\alpha')/2}([0,T])$ in probability.
    It suffices to show that $\{\int_{[0,\cdot]} f^n_r(X_r)dr-\int_{[0,\cdot]} f^n_r(X^n_r)dr\}_n$ converges to $0$ in $C^{({1+\alpha'})/2}([0,T])$ in probability. 
    We are going to apply Proposition \ref{prop:Lip} which requires an estimate for $D_{\alpha-1,\beta-1}(X,X^n)$.    
    The following bounds are sufficient for this purpose
    \begin{gather}\label{tmp:uf1}
      \|[U^n_{s,t}-U_{s,t}]g\|_{\infty}\lesssim \|b-b^n\|_{L^\infty C_b}\|g\|_{\C^{\alpha-1}}|t-s|^{\frac \alpha2}
    \end{gather}
    and
    \begin{equation}\label{tmp:uf2}
      |U_{s,t}g(x)-U_{s,t}g(y)|
      \lesssim |t-s|^{-\frac12+\frac{\beta}2}\|g\|_{ \C^{\alpha-1}}|x-y|^{\alpha- \beta}
    \end{equation}
    for every $g\in \C^{\alpha-1}(\Rd)$, every $x,y\in\Rd$ and every $s<t$. 
    
    Reasoning as in Lemma \ref{lem.UCbeta}, it suffices to show \eqref{tmp:uf1} and \eqref{tmp:uf2} for $g\in C^3_b(\Rd)$.
    To obtain \eqref{tmp:uf1}, note that the function $w_s(x):=[U^n_{s,t}-U_{s,t}]g(x) $ satisfies $w_t(x)=0$ and
    \[
      w_s(x)-w_r(x)=\int_r^s[-L^bw_\theta(x)+(b_\theta-b^n_\theta)\nabla U^n_{\theta,t}g(x)]d \theta
    \]
    for every $0\le r\le s\le t$, $x\in\Rd$. In other words, $w$ is the solution to \eqref{eqn:BK} with right-hand side $f_\theta=(b_\theta-b^n_\theta)\cdot\nabla U^n_{\theta,t}g$ and terminal datum $0$ at time $t$. 
    To estimate $f$, we apply the first inequality in \eqref{est.Uschauder} to obtain 
    \begin{align*}
      \|f_r\|_\infty 
      &\lesssim\|b-b^n\|_{L^\infty C_b}\|\nabla U^n_{r,t}g\|_\infty 
      \lesssim\|b-b^n\|_{L^\infty C_b}\|g\|_{\C^{\alpha-1}}|t-r|^{\frac \alpha2-1}.
    \end{align*}
    % We recall the definition of $F$ in \eqref{def.F}.
    Recall the definition of $F^{(\alpha)}$ in Lemma \ref{lem:sch.flow}.  It follows that
    \begin{align*}
      F^{(\alpha)}_s
      &:=(t-s)^{1-\frac \alpha2}\int_s^t|r-s|^{-\frac12}\|f_r\|_\infty dr
      \lesssim\|b-b^n\|_{L^\infty C_b}\|g\|_{\C^{\alpha-1}}|t-s|^{\frac12}
    \end{align*}
    and hence $F^{(\alpha),*}_s\lesssim\|b-b^n\|_{L^\infty C_b}\|g\|_{\C^{\alpha-1}}|t-s|^{\frac12}$.
    Applying the estimate \eqref{est.supU}, we have
    \begin{equation*}
      \|w_s\|_\infty\lesssim \|b-b^n\|_{L^\infty C_b}\|g\|_{\C^{\alpha-1}}|t-s|^{\frac \alpha2}
    \end{equation*}
    which is equivalent to \eqref{tmp:uf1}. The estimate \eqref{tmp:uf2} is obtained by interpolating the two inequalities in \eqref{est.Uschauder}.

    In the notation of Proposition \ref{prop:Lip}, the inequalities \eqref{tmp:uf1} and \eqref{tmp:uf2} imply that
    \begin{equation*}
      \|U^{X}-U^{X_n}\|_{\C^{\alpha-1},\alpha-1}\lesssim \|b-b^n\|_{L^\infty C_b}
      \quad\textrm{and}\quad
      \rho^{X}_{\beta-1}(h) \lesssim h^{\alpha- \beta}.
    \end{equation*}
    Hence,
    \[
      D_{\alpha-1,\beta-1}(X,X^n)\lesssim \|b-b^n\|_{L^\infty C_b}+\sup_{r\in[0,T]}\|X_r-X^n_r\|_{L_{(\alpha- \beta) m}}^{\alpha- \beta},
    \]
    which converges to $0$ as $n\to\infty$ by \eqref{est:FGP}.
     From Proposition \ref{prop.U31}, we know  that $(X,U^X,\C^{\alpha-1}(\Rd))$ and $(X^n,U^n,\C^{\alpha-1}(\Rd))$ satisfy Condition \ref{con:set1} with $\nu=\alpha-1$. Hence, Proposition \ref{prop:Lip} is applicable and yields
    \begin{align*}
      &\|\int_{[s,t]} f^n_r(X_r)dr-\int_{[s,t]}f^n_r(X^n_r)dr\|_{L_m}
      \lesssim\|f^n\|_{L^\infty \C^{\beta-1}}D_{\alpha-1,\beta-1}(X,X^n) |t-s|^{\frac{1+\beta}2}.
    \end{align*}
    Here, we have used the identities $\caa_t^X[f^n]=\int_{[0,t]}f^n_r(X_r)dr$ and $\caa_t^{X^n}[f^n]=\int_{[0,t]}f^n_r(X^n_r)dr$ from Proposition \ref{prop.K}.
     The above inequality combined with Garsia-Rodemich-Rumsey inequality (\cite{garsiarodemich}) and the fact that $$\lim_n D_{\alpha-1,\beta-1}(X,X^n)=0$$ implies that
    \[
      \lim_n\left[\int_{[0,\cdot]} f^n_r(X_r)dr-\int_{[0,\cdot]} f^n_r(X^n_r)dr\right]=0
    \] in $C^{({1+\alpha'})/2}([0,T])$ in probability.  This completes the proof.
  \end{proof}
  \begin{proof}[Proof of Theorem \ref{thm:eqnDX}]
    In what follows, we fix a trajectory of $V$ which is $\frac{1+\alpha'}2$-H\"older continuous for some $\alpha'\in(0,\alpha)$. The system \eqref{eqn:RSY} is of Young-type, hence it  has a unique solution $Y$ which is $\frac{1+\alpha'}2$-H\"older continuous (see e.g. \cite[Theorems 2.9 and 2.16]{MR3505229} or \cite{MR2397797}). 
    On the other hand, by direct differentiations, it is evident that $Y^{n,ij}:=\partial_i X^{n,j}$ satisfies
    \begin{equation*}
      Y_t^{n,ij}=\delta_{ij}+\sum_{k=1}^d\int_0^t Y_r^{n,ki}dV^{n,kj}_r\,,\quad\forall i,j\in\{1,\dots,d\},
    \end{equation*}
    where $V^n=V(b^n,X^n)=\int_{[0,\cdot]} \nabla b^n_r(X^{n}_r)dr $. 
    Sending $n\to\infty$ in the above identity, applying \eqref{est:FGP2}, Lemma \ref{lem:Vconvg} and stability of Young differential equations (see \cite[Theorem 2.16]{MR3505229} or \cite[Theorem 4]{MR2397797}), we see that $\nabla X$ satisfies equation \eqref{eqn:RSY}. The result follows.
  \end{proof}

% section stochastic_flows (end)
\section{Chaos expansions} % (fold)
\label{sec:chaos_expansion}
  Let $(\Omega,\cff,\P)$ be a complete probability space equipped with a filtration $\{\cff_t\}$ such that $\cff_0$ contains $\P$-null sets.
  As in Section \ref{sec:additive_functionals}, let $T>0$ be a fixed time horizon and let $(\chh,\|\cdot\|_\chh)$ be a normed vector space, which contains $C_b(\R^d)$ and is a subset of $\css'(\R^d)$. For each $q\in[1,\infty]$, $L_T^q\chh$ denotes the Bochner space $L^q([0,T];\chh)$.
  Let $U=\{U_{s,t}\}_{0\le s\le t}$ be a two-parameter family of bounded operators on $\chh$. We recall that the set $\fgg_T(X)$ defined in Definition \ref{def.goodtestfunct}.
  Let $\sigma$ be a continuous bounded function on $\R_+\times\R^d$ and $b$ be an element in $L^{q}_T\chh$. We also assume that for every $t\ge0$, $\sigma(t,\cdot)$ is H\"older continuous with some positive exponent.
  In the current section, we consider chaos expansions for pathwise solutions to the stochastic differential equation
  \begin{equation}\label{SDE}
      dX_t=b(t,X_t)dt+\sigma(t,X_t)dW_t\,, \quad X_0=x\,.
  \end{equation}  
  The results obtained herein extend those of \cite{veretennikov1976explicit,veretennikov1981strong} and \cite{MR1481650} in which SDEs with more regular coefficients were considered.
  Since $b(t,\cdot)$ is allowed to be a distribution, a solution to \eqref{SDE} requires a non-standard definition.
  A triplet $(W,X,\{\cff_t\})$ is a pathwise solution to the stochastic differential equation \eqref{SDE} on $[0,T]$ if
  \begin{enumerate}[label={\upshape(\roman*)}]
    \item $W$ is a Brownian motion with respect to $\{\cff_t\}$,
    \item $X$ is adapted with respect to $\{\cff_t\}$,
    \item there is a sequence $\{b_n\}_n$ of bounded uniformly continuous functions on $[0,T]\times\Rd$ converging to $b$ in $L^q\chh$ such that
    \begin{equation}\label{SDEn}
      X_t=x+\mathrm{ucp-}\lim_{n}\int_0^t b_n(r,X_r)dr+\int_0^t \sigma(r,X_r)dW_r \quad\forall\, t\in[0,T]\,.
    \end{equation}
  \end{enumerate}
  In the above, $\mathrm{ucp-}\lim$ means convergence uniformly in $t$ over $[0,T]$ in probability.
  Condition (iii) is inspired by the work \cite{MR1964949} of Bass and Chen.
   \begin{definition}\label{def.reg.sol}
    A pathwise solution $(W,X,\{\cff_t\})$ is called regular if $(X,\{\cff_t\})$ is a Markov process, 
    $\fgg_T(X)$ contains the set of bounded uniformly continuous functions on $[0,T]\times\Rd$ 
    and $(X,U,\chh)$ satisfies Condition \ref{con:set1} with $\frac \nu2>\frac1q-\frac1{2}$.
  \end{definition}
  A trivial example of regular pathwise solution is the SDE in Section \ref{sec:stochastic_flows}.
  Other examples which we have in mind are the SDEs considered in \cite{MR3500267,MR3785598} where $b$ is a controlled distribution in $[C([0,T];\C^\nu(\Rd))]^d$ with $\nu>-\frac23$.
  In what follows, we always assume that pathwise solutions are regular.
  In such case, we can define the linear map $\caa^X$ as in Proposition \ref{prop:defA}. Furthermore, for every $f\in L^q_T\chh$, we will always work with the continuous version of $\caa^X[f]$ chosen in Remark \ref{rem.intf}(iii). 
  From Corollary \ref{cor:Acont}, we see that $\{\int_{[0,\cdot]} b_n(r,X_r)dr\}_n$ converges to $\caa^X[b]$ in $C^\tau([0,T])$ in probability for some (in fact any) $\tau\in(\frac12,1+\frac \nu2-\frac1q)$. Consequently, the equation \eqref{SDEn} can be written as
  \begin{equation}\label{id:X}
    X_t=x+\caa^X_t[b]+\int_0^t \sigma(r,X_r)dW_r \quad\forall t\ge0.
  \end{equation}
  In particular, we see that $X$ is a.s. $\alpha$-H\"older continuous for every $\alpha\in(0,\frac12)$.
  % It follows that $X$ is a Dirichlet process (a sum of a martingale and a zero energy process).
  
  We define the following formal differential symbols
  \begin{equation*}
    L_t:= \sum_{i,j=1}^d\frac12 (\sigma^\dagger\sigma )^{i,j}(t,\cdot)\partial^2_{x_ix_j}
    \quad\textrm{and}\quad L^b_t=L_t+\sum_{j=1}^d b_j(t,\cdot)\partial_{x_j}\,.
  \end{equation*}
  and introduce the notation
  \begin{equation*}
    \partial^\sigma_j g(r,x)=(\sigma^\dagger\nabla)_jg(r,x) =\sum_{i=1}^d\sigma^{i,j}(r,x)\partial_{x_i}g(r,x)
  \end{equation*}
  and $\nabla^\sigma=(\partial^\sigma_1,\dots,\partial^\sigma_d)$. The following change-of-variable formula is a slight extension of the classical It\^o formula.
  \begin{lemma}[It\^o formula]\label{lem.ItoYoung}
    For every $u$ in $C^2_b([0,T]\times\Rd)$ and $t\in[0,T]$, we have
    \begin{multline}\label{eqn:cito}
      u(t,X_t)=u(0,x)+\int_0^t(\partial_r+L_r)u(r,X_r)dr+\int_0^t\nabla u(r,X_r)\cdot \caa^X[b](dr)
      \\+\int_0^t \nabla^\sigma u(r,X_r)\cdot dW_r\,,
    \end{multline} 
    where the integration with respect to $\caa^X[b](dr)$ is in Young sense.
  \end{lemma}
  \begin{proof}
    Without loss of generality, assume $t=1$ in \eqref{eqn:cito}.
    For brevity, we put $\caa=\caa^X[b]$ and $M=\int_0^\cdot \sigma(r,X_r)dW_r$. We know that $\caa$ is a.s. H\"older continuous with some exponent $\tau>\frac12$ and that $M$ is an $L_m$-martingale for every $m\ge2$. In addition, $M$ satisfies condition \eqref{con.M6} and $[M]_t=\int_0^t(\sigma^\dagger \sigma)(r,X_r)dr$.
    Heuristically, to show \eqref{eqn:cito}, we use Taylor expansion and the facts that $[\caa]=[\caa,M]=0$. 
    Consider first the case $u\in C^3_b([0,T]\times\Rd)$. 
    For every $s<t$, put $A_{s,t}=u(t,X_t)-u(s,X_s)$. Obviously, $\cii_1[A]=u(1,X_1)-u(0,X_0)$. From Taylor expansion, we have
    \begin{align*}
      A_{s,t}
      &=\partial_tu(s,X_s)(t-s)+\wei{\nabla u(s,X_s),X_t-X_s}
      \\&\quad
      +\frac12\wei{\nabla^2u(s,X_s),(X_t-X_s)\otimes(X_t-X_s)}+R(t,X_t,s,X_s)
      \\&=:A^{(1)}_{s,t}+A^{(2)}_{s,t}+A^{(3)}_{s,t}+A^{(4)}_{s,t},
    \end{align*}
    where the Taylor remainder $R(t,x,s,y)$ satisfies
    \[
      |R(t,x,s,y)|\le\|u\|_{C^3_b}(|t-s|^2+|x-y|^3).
    \]
    It is clear that $\cii_1[A^{(1)}]=\int_0^1 \partial_t u(r,X_r)dr$. We write
    \[
      A^{(2)}_{s,t}=\wei{\nabla u(s,X_s),A_t-A_s}+\wei{\nabla u(s,X_s),M_t-M_s}.
    \]
    The action of $\cii$ to each term on the right-hand side above is computed using Young's theory \cite{young} and It\^o stochastic theory respectively, which yields
    \[
      \cii_1[A^{(2)}]=\int_0^1 \nabla u(r,X_r)\cdot dA_r+\int_0^1\wei{\nabla u(r,X_r),\sigma(r,X_r)dW_r}.
    \]
    For $A^{(3)}$, we define $A^{(3a)}_{s,t}=\frac12\wei{\nabla^2u(s,X_s),(M_t-M_s)\otimes(M_t-M_s)}$ and $A^{(3b)}_{s,t}=A^{3}_{s,t}-A^{(3a)}_{s,t}$. As in Example \ref{ex3}, we have
    \[
      \cii_1[A^{(3a)}]
      =\frac12\int_0^1\wei{\nabla^2u(r,X_r),d[M]_r}
      =\frac12\int_0^1\wei{\nabla^2u(r,X_r),(\sigma^\dagger \sigma)(r,X_r)}dr.
    \]
    We note that $\|A^{(3b)}_{s,t}\|_{L_2}\lesssim|t-s|^{1+\varepsilon}$ and $\|A^{4}_{s,t}\|_{L_2}\lesssim|t-s|^{1+\varepsilon}$  for some $\varepsilon>0$. Applying Proposition \ref{prop:AA}, we see that $\cii[A^{(3b)}]=\cii[A^{(4)}]=0$. This implies \eqref{eqn:cito} with $u$ in $C^3_b([0,1]\times\Rd)$.

    In the case $u\in C^{2}_b([0,T]\times\Rd)$, take a mollifying sequence $\{u^n\}_n$ in $C^3_b([0,T]\times\Rd)$ such that $\partial_t u^n,\nabla u^n,\nabla ^2 u^n$ converge uniformly over compact sets respectively to $\partial_t u,\nabla u,\nabla^2 u$. The It\^o formula \eqref{eqn:cito} for $u$ is obtained by passing through the limit from the It\^o formula for each $u^n$. In such procedure, the convergence of the Young integrals is an almost sure convergence, the convergence of the It\^o stochastic integrals is in probability.
  \end{proof}
  
  To proceed further, we need to introduce an additional structure on $\chh$. 
  A normed vector space $\cvv\subset\css'(\R^d)$ is called an $\chh$-multiplier if the multiplication $$\cdot:\cvv\times\chh\to\chh$$ is a continuous bilinear map and is an extension of classical multiplication operation.
  An extension of classical multiplication operation means that if $(g,f)\in\cvv\times\chh\cap C_b(\R^d)\times C_b(\R^d) $ then $g\cdot f= gf$, the product of two continuous functions.

  \begin{definition}
    A normed vector space $\chh\subset\css'(\R^d)$ is admissible if it contains $C_b(\R^d) $ and $C_b^1(\R^d)$ is an $\chh$-multiplier. 
  \end{definition}
  We note that if $\chh$ is admissible, for every $t\in[0,T]$, $[\partial_t+L^b_t](f_t)$ belongs to $\chh$ for every $f$ in $C^2_b([0,T]\times\R^d)$.
  
  An interesting choice of the function $u$ in \eqref{eqn:cito} is $u(r,x)=U_{r,t}\phi(x)$, where $t$ is a fixed terminal time and $\phi$ is a regular function.
  Heuristically, such $u$ is a solution to the Kolmogorov backward equation
  \[
    (\partial_t+L_r^b)u(r,\cdot)=0, \quad u(t,\cdot)=\phi(\cdot).
  \] 
  However, unlike the cases treated in \cite{MR2748616}, a satisfactory theory for such equation is not yet available in our general setup.
 We therefore proceed with an approximation and impose sufficient conditions for its convergence. In the following, we denote $C^+=\cup_{\alpha\in(0,1)}C^\alpha_b(\Rd)$.
  Let $\{b_n\}_n$ be a sequence of bounded uniformly continuous functions on $[0,T]\times\Rd$. For each $n$, $\phi\in C^+$ and $t\in[0,T]$, a function $u:[0,t]\times \Rd\to\R$, let $U^n_{r,t} \phi(r,x)$ denote a solution of the parabolic backward equation 
  \begin{equation}\label{eqn.lbn}
    (\partial_t+L+b_n\cdot\nabla)u(r,x)=0,\quad u(t,x)=\phi(x) \quad
    \forall (s,x)\in[0,t]\times\Rd.
  \end{equation} 
  \begin{condition}\label{con:set2}
    There exists a sequence of bounded uniformly continuous functions $\{b_n\}$ converging to $b$ in $L^q\chh$ such that for every $t'<t$ and every $\phi\in C^+$
    \begin{enumerate}[label={\upshape(\roman*)}]
      \item\label{con:fifi} $\nabla U_{t',t}\phi$ is a function in $C^+$;
      \item\label{con:Ulimst} $\lim_{s\uparrow t}U_{s,t}\phi(x)=\phi(x)$ uniformly in $x$ over compact sets of $\Rd$;
      \item\label{con.unfiexist} $U^n \phi$ exists, $(s,x)\to U^n_{s,t}\phi(x)$ belongs to $C^2_b([0,t']\times\Rd)$ and $\|U^n_{s,t}\phi\|_\infty\le\|\phi\|_\infty$;
      \item\label{con.unfi} $(s,x)\to U^n_{s,t}\phi(x)$ converges to $(s,x)\to U_{s,t}\phi(x) $ uniformly over compact sets of $[0,t']\times\R^d$ and in $L^2([0,t'];L^\infty(\R^d)) $;
      \item\label{con:UUn} $(s,x)\to\nabla U^n_{s,t}\phi(x)$ converges to $(s,x)\to\nabla U_{s,t}\phi(x) $ in $L^2([0,t'];L^\infty(\R^d;\R^d)) $;
      \item\label{con:Un} $(s,x)\to 1_{[0,t']}(s) (b-b_n)(s,x)\cdot \nabla U^n_{s,t}\phi(x) $ converges to 0 in $L^q_T\chh$.
    \end{enumerate}
  \end{condition}
  \begin{remark}
    Condition \ref{con:set2} implies that $U_{s,t}$ is a contraction map on $BUC(\Rd)$ for every $s<t$. Indeed, for every $\phi\in C^+$, $U_{s,t}\phi$ belongs to $BUC(\Rd)$ (by \ref{con:set2}\ref{con:fifi}) and satisfies (by \ref{con:set2}\ref{con.unfiexist},\ref{con.unfi})
    \begin{equation*}
      \|U_{s,t}\phi\|_\infty=\lim_{n}\|U^n_{s,t}\phi\|_\infty\le\|\phi\|_\infty\,.
    \end{equation*}
    By a density argument, the previous inequality also holds for $\phi$ in $BUC(\Rd)$.
  \end{remark}

  \begin{theorem}
    Assuming Condition \ref{con:set2}. Let $\phi$ be a function in $C^+$.
    Then for every $s<t$, 
    \begin{equation}\label{id:CO}
      \phi(X_t)=U_{s,t}\phi(X_s)+\sum_{j=1}^d \int_s^t \partial^\sigma_j U_{r,t}\phi(X_r)dW^j_r \,.
    \end{equation}
  \end{theorem}
  \begin{proof}
    Let $\{b_n\}_n$ be the sequence given in Condition \ref{con:set2} and $t'$ be a fixed number in $[s,t)$.  We note that
    \begin{equation*}
      (\partial_r+L_r)U^n_{r,t}\phi(x)=-b_n(r,x)\cdot\nabla U^n_{r,t}\phi(x) \quad\forall (r,x)\in[s,t']\times\Rd\,,
    \end{equation*}
    and that 
    \[
      \int_0^\cdot b_n(r,X_r)\cdot\nabla U^n_{r,t}\phi(X_r)dr=\caa[b_n\cdot\nabla U^n_{\cdot,t}\phi]
    \]
    by uniform continuity of $b_n(r,\cdot)\cdot\nabla U^n_{r,t}\phi(\cdot)$, Definition \ref{def.reg.sol} and Proposition \ref{prop.K}.
    Hence, applying the  It\^o formula \eqref{eqn:cito} for $u(r,x)=U^n_{r,t}\phi(x)$ where $r\in[s,t']$,
    we get
    \begin{multline}\label{tmp:i}
      U^n_{t',t} \phi(X_{t'})=U^n_{s,t}\phi(X_s)+\caa^X_{t'}[(b-b_n)\cdot\nabla U^n_{\cdot,t}\phi]-\caa^X_{s}[(b-b_n)\cdot\nabla U^n_{\cdot,t}\phi]
      \\+\sum_{j=1}^d \int_s^{t'}  \partial^\sigma_j U^n_{r,t}\phi(X_r)dW^j_r\,.
    \end{multline}
    Using Condition \ref{con:set2}\ref{con:Un} and Proposition \ref{prop:defA}, we have
    \begin{align*}
      \lim_n\caa^X_{t'}[(b-b_n)\cdot\nabla U^n_{\cdot,t}\phi]=\lim_n\caa^X_{s}[(b-b_n)\cdot\nabla U^n_{\cdot,t}\phi]=0
    \end{align*}
    in probability.
    In addition, using Condition \ref{con:set2}\ref{con:UUn}, it is straightforward to verify that
    \begin{equation*}
      \lim_n\|\int_s^{t'} \sigma^{i,j}(r,X_r)(\partial_iU^n_{r,t}\phi(X_r)-\partial_iU_{r,t}\phi(X_r))dW^j_r\|_{L^2(\Omega)}=0 \,.  
    \end{equation*} 
    Hence, we send $n\to\infty$ in \eqref{tmp:i} to obtain
    \begin{equation}\label{tmp:itt}
      U_{t',t}\phi(X_{t'})=U_{s,t}\phi(X_s)+\sum_{j=1}^d \int_s^{t'}  \partial^\sigma_j U_{r,t}\phi(X_r)dW^j_r\,.
    \end{equation}
    Using  Condition \ref{con:set2}\ref{con:Ulimst} and the fact that $X$ is continuous, we see that $$\lim_{t'\uparrow t}U_{t',t}\phi(X_{t'})=\phi(X_t).$$
    On the other hand, we rise the relation \eqref{tmp:itt} to second power, take conditional expectation with respect to $\cff_s$ and use It\^o isometry to obtain that
    \[
      \E^{\cff_s}|U_{t',t}\phi(X_{t'})|^2=\E^{\cff_s}|U_{s,t}\phi(X_s)|^2+\int_s^{t'}\E^{\cff_s}|\nabla^\sigma U_{r,t}\phi(X_r)|^2dr.
    \]
    This implies that
    \begin{equation*}
       \int_s^{t'} \E^{\cff_s} |\nabla^\sigma U_{r,t}\phi(X_r)|^2dr
       \le \E^{\cff_s} |U_{t',t}\phi(X_{t'})|^2
       \le \|\phi\|_\infty^2
    \end{equation*}
    for every $t'<t$. Hence, by martingale convergence theorem, the limit 
    \begin{equation*}
      \lim_{t'\uparrow t}\sum_{j=1}^d \int_s^{t'}  \partial^\sigma_j U_{r,t}\phi(X_r)dW^j_r
    \end{equation*}
    exists a.s. and in $L_2$, which we denote by $\sum_{j=1}^d \int_s^{t}  \partial^\sigma_j U_{r,t}\phi(X_r)dW^j_r$.
    We can send $t'$ to $t$ from below in \eqref{tmp:itt} to obtain formula \eqref{id:CO}.
  \end{proof}
  For each positive integer $n$, denote $[d]^n=\{1,\dots,d\}^n$. An element $\bj$ in $[d]^n$ is an $n$-tuple $\bj=(j_1,\dots,j_n)$ such that $j_i\in\{1,\dots,d\}$ for every $i=1,\dots,n$. For each $\bj\in[d]^n$, let $I^\bj_{s,t}$ denote the $n$-fold iterated integral with respect to $W^{j_1},\dots,W^{j_n}$ over the interval $[s,t]$. That is for every $f\in L^2([s,t]^n) $
  \begin{equation*}
    I^\bj_{s,t}(f)=\int_{[s,t]^n_{>}}f(r_1,\dots,r_n)dW^{j_n}_{r_n}\cdots dW^{j_1}_{r_1}
  \end{equation*}
  where $[s,t]^n_>=\{(r_1,\dots,r_n)\in[s,t]^n:t>r_1>r_2>\cdots>r_n>s\}$.
  
  From Condition \ref{con:set2}\ref{con:fifi}, we see that if $\phi$ is a function in $C^+$, then so is $\partial_j^\sigma U_{r,t}\phi $ for every $r< t$. Hence, we can apply \eqref{id:CO} for $\partial_j^\sigma U_{r,t}\phi $ to obtain
  \begin{equation*}
    \partial^\sigma_j U_{r,t}\phi(X_r)=U_{s,r}\partial^\sigma_j U_{r,t}\phi(X_s)+\sum_{k\in[d]}\int_s^r \partial^\sigma_k U_{r_1,r}\partial^\sigma_j U_{r,t}\phi(X_{r_1})dW^k_{r_1}
  \end{equation*}
  and so
  \begin{align}
    \phi(X_t)=U_{s,t}\phi(X_s)&+\sum_{j\in [d]}\int_s^tU_{s,r}\partial^\sigma_j U_{r,t}\phi(X_s)dW^j_r
    \nonumber\\&\quad+\sum_{(j,k)\in[d]^2}\int_s^t\int_s^r\partial^\sigma_j U_{r_1,r}\partial^\sigma_k U_{r,t}\phi(X_{r_1})dW^k_{r_1}dW^j_r\,.
  \end{align}
  It is evident that this procedure can be iterated. 
  \begin{theorem} Assuming Condition \ref{con:set2}. Let $\phi$ be a function in $C^+$. For every integer $n\ge1$ and every $0\le s\le t$
    \begin{equation}\label{chaos:n}
      \phi(X_t)=U_{s,t}\phi(X_s)+ \sum_{k=1}^n\sum_{\bj\in[d]^k} I^\bj_{s,t}(f_{s,t}^\bj)+\sum_{\bj\in[d]^{n+1}} I^\bj_{s,t}(g_{s,t}^{\bj})
    \end{equation}
    where for every $\bj=(j_1,\dots,j_n)\in[d]^n$ and every $(s_1,\dots,s_n)\in[s,t]^n_> $
    \begin{equation}\label{def:g}
      g^\bj_{s,t}(s_1,\dots,s_n)=\partial^\sigma_{j_n} U_{s_n,s_{n-1}}\cdots \partial^\sigma_{j_1} U_{s_1,t}\phi(X_{s_n})\,,
    \end{equation}
    and
    \begin{equation}\label{def:f}
      f^\bj_{s,t}(s_1,\dots,s_n) =\E \(g^\bj_{s,t}(s_1,\dots,s_n)|\cff_s\)=U_{s,s_n}\partial^\sigma_{j_n} U_{s_n,s_{n-1}}\cdots \partial^\sigma_{j_1} U_{s_1,t}\phi(X_s)\,.
    \end{equation}
  \end{theorem}
  \begin{proof}
    Straightforward.
  \end{proof}
  Let $\{\cff^W_t\}$ be the augmented filtration generated by $W$ which satisfies the usual conditions.
  \begin{theorem}\label{thm:phistr} Assuming Condition \ref{con:set2}.
    For every function $\phi$ in $C^+$, $\phi(X_t)$ is $\cff^W_t$-measurable if and only if
    \begin{equation}\label{con:VK}
      \lim_{n\to\infty}\sum_{\bj\in[d]^n} \int_{(s_1,\dots,s_n)\in[0,t]^n_>}U_{0,s_n}[\partial^\sigma_{j_n} U_{s_n,s_{n-1}}\cdots \partial^\sigma_{j_1} U_{s_1,t}\phi(x)]^2ds_1\dots ds_n=0\,.
    \end{equation}
    In such case, for every $s\le t$, we have
    \begin{equation}\label{chaos:infinity}
      \phi(X_t)=U_{s,t}\phi(X_s)+ \sum_{k=1}^\infty\sum_{\bj\in[d]^k} I^\bj_{s,t}(f_{s,t}^\bj)\,,
    \end{equation}
    where $f^\bj$'s are defined in \eqref{def:f} and the series converges in $L^2(\Omega)$. 
  \end{theorem}
  \begin{proof}
    The argument is essentially that of \cite{veretennikov1981strong,veretennikov1976explicit}. The formula \eqref{chaos:n} with $s=0$ gives
    \begin{equation}\label{tmp.phix}
      \phi(X_t)=U_{0,t}\phi(x)+ \sum_{k=1}^n\sum_{\bj\in[d]^k} I^\bj_{0,t}(f_{0,t}^\bj)+\sum_{\bj\in[d]^{n+1}} I^\bj_{0,t}(g_{0,t}^{\bj}).
    \end{equation}
    We note from \eqref{def:f} that each $f_{0,t}^\bj$ is deterministic.
    We recall that the space $L^2(\Omega,\cff^W_t)$ admits the Wiener-It\^o chaos decomposition
    \[
      L^2(\Omega,\cff^W_t)=\bigoplus_{n=0}^\infty\chh_n
    \]
    where for each $n\ge0$, $\chh_n$ is the closed linear space spanned by $\{I_{0,t}^\bj(f):f\in L^2([0,t]^n),\bj\in[d]^n\}$, see \cite[Theorems 1.1.1 and 1.1.2]{MR2200233}. Each $\chh_n$ is called the Wiener chaos of order $n$. In particular, $I^\bj_{0,t}(f_{0,t}^\bj)$ belongs to $\chh_n$ for each $\bj\in[d]^n$.

    Suppose that $\phi(X_t)$ is $\cff^W_t$-measurable. Since $\phi$ is bounded, it is obvious that $\phi(X_t)$ is square integrable. Then by Wiener-It\^o decomposition, we have
    \begin{equation}\label{tmp.pchaos}
      \phi(X_t)=\sum_{k=0}^\infty G_k
    \end{equation}
    where each $G_k$ belongs to $\chh_k$. Comparing \eqref{tmp.pchaos} with \eqref{tmp.phix} and taking into account orthogonality reveals that $G_0=U_{0,t}\phi(x)$, $G_k=\sum_{\bj\in[d]^k} I^\bj_{0,t}(f_{0,t}^\bj)$ for $1\le k\le n$ and
    \[
      \sum_{\bj\in[d]^{n+1}} I^\bj_{0,t}(g_{0,t}^{\bj})=\sum_{k=n+1}^\infty G_k.
    \]
    Since \eqref{tmp.pchaos} holds in $L^2(\Omega,\cff^W_t)$, this implies that
    \[
      \lim_n\E|\sum_{\bj\in[d]^{n+1}} I^\bj_{0,t}(g_{0,t}^{\bj})|^2=\lim_n\E|\sum_{k=n+1}^\infty G_k|= 0
    \]
    Taking into account \eqref{def:g}, we see that the above limit is equivalent to \eqref{con:VK}.

    On the other hand, if \eqref{con:VK} holds, then 
    \[
      \lim_n\E|\sum_{\bj\in[d]^{n+1}} I^\bj_{0,t}(g_{0,t}^{\bj})|^2=0.
    \]
    One can send $n\to\infty$ in \eqref{chaos:n} to obtain 
    \begin{equation*}
      \phi(X_t)=U_{0,t}\phi(x)+ \sum_{k=1}^\infty\sum_{\bj\in[d]^k} I^\bj_{0,t}(f_{0,t}^\bj)\,.
    \end{equation*}
    where the series converges in $L^2(\Omega)$.
    Since each term on the right-hand side of the above identity is $\cff^W_t$-measurable, it is evident that $\phi(X_t)$ is $\cff^W_t$-measurable.
    The identity \eqref{chaos:infinity} is obtained analogously to the previous one.
  \end{proof}
  We recall that a pathwise solution $(W,X,\{\cff_t\})$ is called strong if $X_t$ is $\cff^W_t$-measurable for every $t\ge0$.
  \begin{corollary}
    A regular pathwise solution $(W,X,\{\cff_t\})$ is a strong solution if and only if \eqref{con:VK} holds for every function $\phi$ in $C^+$ and every $t\ge0$.
   \end{corollary}
  \begin{proof}
    Suppose that  \eqref{con:VK} holds for every function in $\phi\in C^+$ and every $t\ge0$.
    Let $\{\phi_n\}$ be a sequence of functions in $C^+$ which converges to the identity map uniformly over compact sets. Then, for every $t\ge0$, $\phi_n(X_t)$ is $\cff^W_t$-measurable. This implies that $X_t$ is $\cff^W_t$-measurable for every $t\ge0$ and hence $(X,W)$ is a strong solution. The other direction is straightforward from Theorem \ref{thm:phistr}.
  \end{proof}
% section chaos_expansion (end)

\section{SDEs driven by fractional Brownian motions} % (fold)
\label{sec:sdes_driven_by_fractional_brownian_motions}
  Let $B^H$ be a standard fractional Brownian motion in $\Rd$ with Hurst parameter $H\in(0,\frac12) $. This means that $B^H$ is a Gaussian process in $\Rd$ with $B_0=0$ and covariance given by
  \begin{equation}\label{cov.fbm}
    \E(B^H_t\otimes B^H_s)=\frac12(t^{2H}+s^{2H}-|t-s|^{2H}) I_d.
    \quad\textrm{for all}\quad s,t\ge0.
  \end{equation}
  We consider the stochastic differential equation driven by fractional Brownian motion
  \begin{equation}\label{eqn:fSDE}
    X_t=x+\int_0^tb(r,X_r)dr+B^H_t \quad\forall t\in[0,T]\,,
  \end{equation}
  where $x\in\Rd$, $b$ is a Borel function in $[L^q([0,T];L^p(\Rd))]^d$, $p,q\in[1,\infty]$. In particular, when $p=q=\infty$, $b$ is uniformly bounded over $[0,T]\times\R^d$. 
  Since $b$ is only measurable in time, the integral $\int_0^t b(r,X_r)$ should be interpreted in Lebesgue sense. 
  We will always view each element of $[L^q([0,T];L^p(\Rd))]^d$ as a measurable function instead of an equivalent class of measurable functions. From this point of view, there is no ambiguity for writing $b(r,X_r)$. 
  
  We will show in the current section that two adapted solutions defined on the same filtered probability space coincide provided that
  \begin{equation}\label{con:fbmpq}
    H\frac dp+\frac1q<\frac12-H\,.
  \end{equation}
  Hereafter, we use the convention $1/\infty=0$. Existence and uniqueness in law of weak solutions were obtained earlier by Nualart and Ouknine \cite{MR2073441} under the condition
  \begin{equation}\label{con:weakfbm}
    H\frac dp+\frac1q<\frac12\,.
  \end{equation}
  In fact, the authors of \cite{MR2073441} consider \eqref{eqn:fSDE} in one dimension, however, their arguments, which rely on Girsanov transformation, can be easily extended to multi-dimensions, as we will explain below.
  Dimension one is special, because of the validity of a comparison principle for \eqref{eqn:fSDE}. In this case, pathwise uniqueness holds under condition \eqref{con:weakfbm}.
  To obtain pathwise uniqueness in high dimensions, we rely on stochastic sewing lemma, Theorem \ref{lem:Sew2}, which leads to the more restricted condition \eqref{con:fbmpq}.

  Equation \eqref{eqn:fSDE} was considered earlier in \cite{MR1934157,MR2073441,banos2015strong} and in \cite{MR2117951} when $H=1/2$.
  % Compare to \cite{MR2073441}, condition \eqref{con:weakfbm} is more general. 
  Nualart and Ouknine in \cite{MR1934157} obtain strong existence and pathwise uniqueness for \eqref{eqn:fSDE} in one dimension, where $b$ is a measurable function with linear growth. 
  Modulo a localization procedure, this result corresponds to $p=q=\infty$ in our case. 
  In \cite{MR2073441}, the same authors extend their previous result in \cite{MR1934157} to allow drifts in $L^q([0,T];L^p(\R))$, with $H\frac1p+\frac1q<\frac12$. 
  
  Ba\~nos et.al. \cite{banos2015strong} obtained existence of a unique strong solution for \eqref{eqn:fSDE} in multi-dimensions, $b\in L^1(\Rd;L^\infty([0,T],\Rd))\cap L^\infty(\Rd;L^\infty([0,T];\Rd)) $ and $H<\frac1{2(2+d)}$. 
  This is also a particular case of the results stated here.
  The approaches of \cite{MR2073441,banos2015strong}, which rely on Girsanov transformation, comparison principle and Malliavin calculus,  are different from the one presented here. 
  Although the condition \eqref{con:fbmpq} is not optimal, it appears to be new and is the best available for stochastic differential equations driven by fractional Brownian motions in multi-dimensions. We point the readers to the end of the current section for further discussions.

  To state the main results precisely, let us recall some relevant definitions.   
  For each $t\ge0$, we denote by $\cff^{B^H}_t$ the $\sigma$-field generated by $\{B^H_r:r\in[0,t]\}$ and the sets of null probability. In other words, $\{\cff^{B_H}_t\}$ is the augmented filtration generated by $B^H$.
  One can find a Wiener process $W$ such that
  \begin{equation}\label{id:fbm}
    B^H_t=\int_0^t K_H(t,s)dW_s\,,\quad \forall t\ge0\,.
  \end{equation}
  where  $K_H$ is the square integrable kernel given by 
  \begin{align*}
    K_H(t,s)=c_H\left[\left(\frac ts\right)^{H-\frac12}(t-s)^{H-\frac12}
    -\left(H-\frac12\right)s^{\frac12-H}\int_s^tu^{H-\frac32}(u-s)^{H-\frac12}du
    \right],
  \end{align*}
  for every $s<t$ and $c_H=\lt(\frac{2H \Gamma(\frac32-H)}{\Gamma(H+\frac12)\Gamma(2-2H)}\rt)^{1/2}$. A procedure to find $W$ is described in \cite[Section 5.1.3]{MR2200233}. For the formula of $K_H$ above, we refer to \cite[Proposition 5.1.3]{MR2200233}. In addition, $W$ can be chosen so that $B^H$ and $W$ generate the same filtration (\cite[page 280]{MR2200233}).
  % $F(a,b,c,z)$ being the Gauss hypergeometric function. 

  Let $(\cff_t)_{t\ge0}$ be a right-continuous filtration such that $\cff_0$ contains the $\P$-null sets. 
  Following \cite{MR1934157}, we say that $B^H$ is an $\{\cff_t\}$-fractional Brownian motion if there exists an $\{\cff_t\}$-Wiener process $W$ such that \eqref{id:fbm} holds.
  The previous paragraph can be summarized in the statement that every standard fractional Brownian motion $B^H$ is an $\{\cff^{B^H}_t\}$-fractional Brownian motion.

  An essential property of $\{\cff_t\}$-fractional Brownian motion which we exploit is the local nondeterminism property. Namely, there exists a positive constant $c$ such that
  \begin{equation}\label{NDP}
    \sigma_H(s,t)\ge c|t-s|^H \quad\forall s<t\,,
  \end{equation}
  where 
  \begin{equation}\label{def:sigma}
    \sigma^2_H(s,t):=\frac1d\E\left(|B^H_t-\E^{\cff_s}B^H_t|^2\right) =\int_s^t|K_H(t,r)
    |^2dr \,.
  \end{equation}
  This property was used in \cite{MR2073441} without proof.
  We give a simple proof of \eqref{NDP}. Without loss of generality, assume $d=1$. The fractional Brownian motion $B^H$ admits a moving average representation which was introduced by Mandelbrot and Van Ness \cite{MR242239}
  \[
    B^H_t=C(H)\int_{-\infty}^t[(t-r)_+^{H-\frac12}-(-r)^{H-\frac12}]d\widetilde W_r
    \quad\textrm{for every}\quad t\ge0,
  \]
  where $C(H)$ is a suitable normalizing constant and $\widetilde W$ is a two-sided Brownian motion. For each $t\ge0$, let $\widetilde\cff_t$ the the $\sigma$-algebra generated by $\{\widetilde W_r:r\le t\}$ and the sets of null probability. 
  Clearly $B^H_t$ is $\widetilde \cff_t$-measurable and hence $\cff^{B^H}_t\subset \widetilde \cff_t$. It follows that
  \begin{align*}
    \int_s^t|K_H(t,r)
    |^2dr=\E(B^H_t-\E^{\cff^{B^H}_s}B^H_s)^2\ge \E(B^H_t-\E^{\widetilde \cff_s}B^H_t)^2.
  \end{align*}
  Using the moving average representation, it is straightforward to see that
  \[
    \E(B^H_t-\E^{\widetilde \cff_s}B^H_t)^2=\E\left(C(H)\int_s^t(t-r)^{H-\frac12}d\widetilde W_r\right)^2= C_1(H)|t-s|^{2H}
  \]
  where $C_1(H)=(C(H)(H+\frac12)^{-1})^2$ is a positive number. From here we obtain \eqref{NDP}.

  By a pathwise solution to \eqref{eqn:fSDE} on $(\Omega,\cff,\P)$, we mean a triplet $(B^H,X,\{\cff_t\}_{t\in[0,T]})$ such that
  \begin{enumerate}[(i)]
    \item $\{\cff_t\}$ is a right-continuous filtration, $\cff_0$ contains $\P$-null sets, $X$ and $B^H$ are $\{\cff_t\}$-adapted,
    \item $B^H$ is an $\{\cff_t\}$-fractional Brownian motion,
    \item $X$ and $B^H$ satisfy \eqref{eqn:fSDE} a.s. for every $t\in[0,T]$.
  \end{enumerate}
  A pathwise solution is called a strong solution if $X$ is adapted to the augmented filtration generated by $B^H$, i.e. $\{\cff^{B^H}_t\}$. Since $B^H$ and $W$ generate the same filtration, $X$ is a strong solution iff it is adapted to the augmented filtration generated by $W$.
  % Here, $\{\cff_t\}_{t\ge0}$ is a right-continuous filtration and $\cff_0$ contains the sets of null probability.

  A weak solution to \eqref{eqn:fSDE} is a tuple $((\Omega,\cff,\P),B^H,X,\{\cff_t\}_{t\in[0,T]})$ such that $(\Omega,\cff,\P)$ is a complete probability space and $(B^H,X,\{\cff_t\}_{t\in[0,T]})$ is a pathwise solution to \eqref{eqn:fSDE} on $(\Omega,\cff,\P)$.

  The following theorems are our main results in the current section.
  \begin{theorem}[Nualart-Ouknine] \label{thm:weakfsde}
    Suppose that \eqref{con:weakfbm} holds. Then equation \eqref{eqn:fSDE} has a weak solution. Any two weak solutions have the same probability law. 
  \end{theorem}
  \begin{theorem}[Pathwise uniqueness]\label{thm:puni}
    Suppose that \eqref{con:fbmpq} holds. Let $(B^H,X,\{\cff_t\})$ and $(B^H,\bar X,\{\cff_t\})$ be two pathwise solutions to \eqref{eqn:fSDE} starting from $x$ on $(\Omega,\cff,\P)$. Then $X$ and $\bar X$ are indistinguishable. 
  \end{theorem}
  \begin{theorem}[Strong existence]\label{thm:fbm}
    Suppose that \eqref{con:fbmpq} holds. Then \eqref{eqn:fSDE} has a unique strong solution.
 \end{theorem}
  Theorem \ref{thm:weakfsde} is an easy extension of \cite{MR1934157,MR2073441} using Girsanov transformation. Theorem \ref{thm:puni} relies on Lipschitz property of the map $\psi\mapsto\int b(r,B^H_r+\psi_r)dr $, which is derived using Theorem \ref{lem:Sew2}.
  The notion of $\{\cff_t\}$-fractional Brownian motion and formula \eqref{id:fbm} allow us to view equation \eqref{eqn:fSDE} as a stochastic integral equation driven by a Wiener process. Having established weak existence and pathwise uniqueness for equation \eqref{eqn:fSDE} in Theorems \ref{thm:weakfsde} and \ref{thm:puni}, an application of Yamada-Watanabe principle \cite{MR0278420} yields Theorem \ref{thm:fbm}.

  \subsection*{Weak existence and uniqueness in probability law} % (fold)
  \label{sub:weak_existence_and_uniqueness_in_law}
  
  % subsection weak_existence_and_uniqueness_in_law (end)
  Let us now discuss in more detail. We begin with the following lemma, which extends \cite[Lemma 3.2]{MR2073441}.

  \begin{lemma}\label{lem:khas}
    Let $B^H$ be an $\{\cff_t\}$-fractional Brownian motion.
    Let $(p,q)\in[1,\infty]^2$ satisfy $H\frac dp +\frac1q<1$.
    For every non-negative measurable function $h$ in $ L^q([0,T]; L^p(\Rd)) $, we have
    \begin{align}\label{est:eh}
      \E^{\cff_s}\int_s^Th(t,B^H_t)dt\lesssim \|h\|_{L^qL^p}(T-s)^{1-H\frac dp-\frac1q}
    \end{align}
    and
    \begin{align}\label{est:exph}
      \E\exp\(\int_0^Th(t,B^H_t)dt \)<\infty\,,
    \end{align}
    where $q'=\frac q{q-1}$ is the H\"older conjugate of $q$.
  \end{lemma}
  \begin{proof}
    Note that $q'<\infty$ because $q>1$.
    By Tonelli's theorem, we have
    \begin{equation*}
      \E^{\cff_s}\int_s^Th(t,B^H_t)dt=\int_s^T\E^{\cff_s}h(t,B^H_t)dt=\int_s^T\E^{\cff_s}h(t,\E^{\cff_s}B^H_t+(B^H_t-\E^{\cff_s}B^H_t) )dt\,.
    \end{equation*}
    Using the definition of $\{\cff_t\}$-fractional Brownian motion, \eqref{id:fbm} and \eqref{def:sigma}, we see that $B^H_t-\E^{\cff_s}B^H_t=\int_s^t K_H(t,r)dW_r$, which has a centered normal distribution with covariance matrix $\sigma^2_H(s,t)I_d$ and is independent from $\cff_s$. Hence, we obtain
    \begin{equation*}
      \E^{\cff_s}\int_s^Th(t,B^H_t)dt
      =\int_s^TP_{\sigma^2_H(s,t)}[h(t,\cdot)](\E^{\cff_s}B^H_t)dt\,,
    \end{equation*}
    where $\{P_\sigma\}_{\sigma\ge0}$ is the heat semigroup associated to the kernel $(2 \pi \sigma)^{-\frac d2}e^{-|x-y|^2/(2 \sigma)}$.
    For each $\sigma>0$, $P_\sigma$ maps $L^p$ to $C_b(\R^d)$ and for every $\phi\in L^p(\Rd)$, we have
  \begin{equation*}
    \|P_{\sigma}\phi\|_\infty\lesssim \sigma^{-\frac d{2p}}\|\phi\|_{L^p}\,.
  \end{equation*}
    This bound follows from the H\"older inequality $\|p_\sigma * \phi\|_{\infty}\le\|p_\sigma\|_{L^{p'}}\|\phi\|_{L^p}$, where $p_\sigma$ is the Gaussian density of $P_\sigma$ and $\frac1{p'}+\frac1p=1$.
    In conjunction with \eqref{NDP} and H\"older inequality, we have
    \begin{align}
      \E^{\cff_s}\int_s^Th(t,B^H_t)dt
      &\lesssim\int_s^T|t-s|^{-H\frac dp}\|h(t,\cdot)\|_{L^p(\Rd)}dt
      \nonumber\\&\lesssim \|h\|_{L^qL^p}\(\int_s^T|t-s|^{-q'H\frac dp}dt \)^{\frac1{q'}}.
      \label{tmp.eh}
    \end{align}
    This yields \eqref{est:eh}.
    Then, by Taylor's expansion, we have
    \begin{align*}
      \E\exp\(\int_0^Th(t,B^H_t)dt \)=1+\sum_{n=1}^\infty I_n\,,
    \end{align*}
    where
    \begin{equation*}
      I_n=\E\int_0^T\int_{t_1}^T\cdots\int_{t_{n-1}}^T h(t_1,B^H_{t_1})\cdots h(t_n,B^H_{t_n})dt_n\cdots dt_1\,.
    \end{equation*}
    By conditioning successively on $\cff_{t_{n-1}},\dots, \cff_{t_1}$ and using \eqref{tmp.eh} we obtain
    \begin{multline*}
      I_n
      \le C^n\|h\|^n_{L^qL^p}
      \\\times\(\int_0^T\int_{t_1}^T\cdots\int_{t_{n-1}}^T (t_n-t_{n-1})^{-q'H\frac dp}\cdots (t_2-t_{1})^{-q'H\frac dp}t_1^{-q'H\frac dp} dt_n\cdots dt_1 \)^{\frac1{q'}}.
    \end{multline*}
    Integrating successively starting from $dt_n$ to $dt_1$, using the identity 
    \[
      \int_s^T(T-t)^{x-1}(t-s)^{y-1}dt=(T-s)^{x+y-1}\frac{\Gamma(x)\Gamma(y)}{\Gamma(x+y)} 
      \quad\textrm{for every}\quad x,y>0,
    \]
    it is straightforward to obtain
    \begin{align*}
      I_n\le \frac{C^n\|h\|^n_{L^qL^p} T^{n(1-H\frac dp-\frac1q)}}{\Gamma(n(1-H\frac dp-\frac1q)+1) }
    \end{align*}
    for some constant $C>0$ depending on $p,q,H$. Thus we have
    \begin{equation*}
      \E\exp\(\int_0^Th(t,B^H_t)dt \)\le \sum_{n=0}^\infty  \frac{C^n\|h\|^n_{L^qL^p} T^{n(1-H\frac dp-\frac1q)}}{\Gamma(n(1-H\frac dp-\frac1q)+1) }\,,
    \end{equation*}
    which implies finiteness of the exponential moment. 
  \end{proof}
  In the above proof, using the identity $\E\(\int_0^T h(t,B^H_t)dt \)^n=n!I_n $, we obtain the following estimate
  \begin{equation}\label{est:momenth}
    \E\(\int_0^T h(t,B^H_t)dt \)^n\le \frac{n!C^n\|h\|^n_{L^qL^p}T^{n(1-H\frac dp-\frac1q)}}{\Gamma(n(1-H\frac dp-\frac1q)+1)}
  \end{equation}
  for every positive integer $n$.
  \begin{proposition}[Girsanov transformation]\label{Prop.Gir}
    Let $B^H$ be an $\{\cff_t\}$-fractional Brownian motion in $\Rd$ with Hurst parameter $H\in(0,1/2)$.
    Consider the shifted process
    \[
      \tilde B^H_t=B^H_t+\int_0^t h_s ds
    \]
    where $(h_s)_{s\in[0,T]}$ is an $\{\cff_t\}$-adapted process in $\Rd$ with integrable trajectories. We define
    \begin{align*}
      v_s
      % =K_H^{-1}\(\int_0^\cdot b(r,X_r)dr \)(s)
      =\frac1{\Gamma(\frac12-H)} s^{H-\frac12}\int_0^s(s-r)^{-\frac12-H}r^{\frac12-H}h_rdr\,.
    \end{align*}
    Assume that
    \begin{enumerate}[label={\upshape(\roman*)}]
      \item\label{g1} $v$ belongs to $L^2([0,T];\Rd)$, almost surely, 
      \item\label{g2} $\E(\xi_T)=1$ where
      \[
        \xi_T=\exp\left\{-\int_0^Tv_s\cdot dW_s-\frac12\int_0^T\left|v_s\right|^2ds \right\}.
      \]
    \end{enumerate}
    Then the shifted process $\tilde B^H$ is an $\{\cff^{B^H}_t\}$-fractional Brownian motion with Hurst parameter $H$ under the new probability $\tilde\P$ defined by $d\tilde\P=\xi_Td\P$.
  \end{proposition}
  \begin{proof}
    This result is derived from the Girsanov transformation of multidimensional Brownian motion $W$ via the representation \eqref{id:fbm}. The specific form of $v$ is chosen so that $\int_0^t K_H(t,s)v_sds=\int_0^th_rdr$ (see \cite[Eq. (13)]{MR1934157}). The proof follows the same arguments as \cite[Theorem 2]{MR1934157} with obvious modifications to multidimensional setting. 
  \end{proof}
  \begin{proof}[Proof of Theorem \ref{thm:weakfsde}]
    This result can be obtained following the approach of \cite{MR2073441}. The only difference here is that we work on arbitrary finite dimension.

    In \cite{MR2073441}, a weak solution solution is constructed using Girsanov transformation, which relies on the validity of the following inequality (Novikov's condition)
    \begin{equation}\label{tmp:expv}
      \E\exp\(\theta\int_0^T|v_s|^2ds \)<\infty
    \end{equation}
    where $\theta$ is any positive number  and
    \begin{align*}
      v_s
      % =K_H^{-1}\(\int_0^\cdot b(r,X_r)dr \)(s)
      =\frac1{\Gamma(\frac12-H)} s^{H-\frac12}\int_0^s(s-r)^{-\frac12-H}r^{\frac12-H}b(r,B^H_r)dr\,.
    \end{align*}
    See \cite[Lemma 3.3 and equation (3.11)]{MR2073441}. 
    These formulas are unchanged in multi-dimensions (see Proposition \ref{Prop.Gir}).
    As in \cite{MR1934157,MR2073441}, using the fact that $H<\frac12$ and Young's inequality, we have
    \begin{align}\label{tmp.vb}
      \int_0^T|v_s|^2ds\lesssim\int_0^T \left|\int_0^s(s-r)^{-\frac12-H}b(r,B^H_r)dr\right|^2ds
      \lesssim \int_0^T|b(s,B^H_s)|^2ds\,.
    \end{align}
    Hence, to obtain \eqref{tmp:expv}, it suffices to show that
    \begin{equation}\label{tmp:expb}
      \E\exp\(\theta\int_0^T|b(r,B^H_r)|^{2}dr\)<\infty
    \end{equation}
    for any $\theta>0$.
    Observe that the function $h=|b|^{2}$ belongs to $L^{q/2}([0,T];L^{p/2}(\R^d))$.
    The condition \eqref{con:weakfbm} ensures that
    \begin{equation*}
      H\frac d{p/2}+\frac1{q/2}<1\,.
    \end{equation*}
    Hence, we can apply Lemma \ref{lem:khas} to obtain \eqref{tmp:expb}. From here, the arguments used in the proofs of \cite[Theorems 3.3 and 3.4]{MR2073441} are applicable, which yield existence and uniqueness in law of weak solutions to \eqref{eqn:fSDE}.
  \end{proof}
  \begin{remark}
    In \eqref{tmp.vb}, one could use Hardy-Littlewood inequality (see \cite[Theorem 1, p. 119]{MR0290095}) to obtain 
    \[
      \int_0^T|v_s|^2ds \lesssim \left(\int_0^T|b(s,B^H_s)|^{\frac1{1-H}}ds\right)^{2(1-H)}\,,
    \]
    which is an improvement over \eqref{tmp.vb}. If this inequality is used instead of \eqref{tmp.vb} in the previous proof, the integrability condition \eqref{tmp:expb} should be replaced by
    \[
      \E\exp\left\{\theta\(\int_0^T|b(r,B^H_r)|^{\frac1{1-H}}dr\)^{2(1-H)}\right\}<\infty.
    \]
    However, verifying this condition using Taylor expansion and the moment estimate \eqref{est:momenth} also leads to condition \eqref{con:weakfbm}.
  \end{remark}
  
  For later purposes, we state the following result which is analogous to \cite[Lemma 3.3]{MR2073441}.
  \begin{lemma}\label{lem:GirX}
    Let $(B^H,X,\{\cff_t\}_{t\in[0,T]})$ be a pathwise solution of \eqref{eqn:fSDE} defined on a probability space $(\Omega,\cff,\P)$. Set
    \begin{equation}\label{def.v}
      v_s=\frac1{\Gamma(\frac12-H)} s^{H-\frac12}\int_0^s(s-r)^{-\frac12-H}r^{\frac12-H}b(r,X_r)dr\,.
    \end{equation}
    Assume that \eqref{con:weakfbm} holds. Then $v$ belongs to $L^2([0,T];\Rd)$ almost surely and
    \begin{equation}\label{def.xit}
      \xi_T=\exp\(-\int_0^Tv_s \cdot dW_s-\frac12\int_0^T|v_s|^2ds \)
    \end{equation}
    defines a random variable such that the measure $\tilde \P$ given by $d\tilde\P=\xi_T d\P$ is a probability measure equivalent to $\P$. Moreover, for every $\theta\in\R$, there exists a constant $K_\theta$ depending on $\theta,T,K,H,p,q$ and on $\|b\|_{L^qL^p}$ such that
    \begin{equation*}
      \E^{\tilde\P}\xi^\theta_T+\E^\P \xi^\theta_T\le K_\theta\,.
    \end{equation*}
  \end{lemma}
  \begin{proof}
    In \cite[Lemma 3.3]{MR2073441}, this result relies on Girsanov transformation and the validity of inequality \eqref{est:exph}. Hence it can be carried over in a multi-dimensional setting.
  \end{proof}
  \begin{proposition}\label{prop.fx}
    Let $(B^H,X,\{\cff_t\}_{t\in[0,T]})$ be a pathwise solution of \eqref{eqn:fSDE} defined on a probability space $(\Omega,\cff,\P)$ and $(p_1,q_1)$ be in $[1,\infty]^2$ such that $H\frac d{p_1}+\frac1{q_1}<1$. Assume that \eqref{con:weakfbm} holds. Then for every Borel function $f$ in $L^{q_1}([0,T];L^{p_1}(\Rd))$, every integer $m\ge2$ and $s,t\in[0,T]$,
    \begin{equation*}
      \|\int_s^t f_r(X_r)dr\|_{L_m}\le C(T,m,p,q,\|b\|_{L^qL^p})\|f\|_{L^{q_1}L^{p_1}} |t-s|^{1-H\frac d{p_1} -\frac1{q_1} }\,.
    \end{equation*}
  \end{proposition}
  \begin{proof}
    Define $v$ and $ \xi_T$ by \eqref{def.v} and \eqref{def.xit}, respectively. By Lemma \ref{lem:GirX}, the process $v$ satisfies conditions \ref{g1} and \ref{g2} of Proposition \ref{Prop.Gir}. Hence, under the measure $\tilde \P$ given by $d\tilde\P=\xi_Td\P$, the process $X$ is a fractional Brownian motion with Hurst parameter $H$. 
    By Cauchy-Schwarz inequality, we have
    \begin{align*}
      \E^{\P}|\int_s^tf_r(X_r)dr |^m
      &=\E^{\tilde\P}|\int_s^tf_r(X_r)dr|^m \xi_T^{-1}
      \\&\le \(\E^{\tilde\P}\xi_T^{-2}\)^{\frac12 }\(\E^{\tilde\P}|\int_s^tf_r(X_r)dr|^{2m }\)^{\frac12}\,.
    \end{align*}
    We now apply Lemma \ref{lem:GirX} and \eqref{est:momenth} to obtain the result.
  \end{proof}
  As an application, we derive a regularity property for pathwise solutions.
  \begin{proposition}\label{prop:regweaksoln}
    Let $(B^H,X,\{\cff_t\}_{t\in[0,T]})$ be a pathwise solution of \eqref{eqn:fSDE} defined on a probability space $(\Omega,\cff,\P)$. Put $\psi=X-B^H$ and assume that \eqref{con:weakfbm} holds. Then for every $m\ge2$ and $s,t\in[0,T]$,
    \begin{equation*}
      \|\psi_t- \psi_s\|_{L_m}\le C(T,m,p,q,\|b\|_{L^qL^p}) |t-s|^{1-H\frac dp-\frac1q}\,.
    \end{equation*}
  \end{proposition}
  \begin{proof}
    Note that $\psi_t- \psi_s=\int_s^t b_r(X_r)dr$. The result is a direct application of Proposition \ref{prop.fx}.
  \end{proof}
  \subsection*{Pathwise uniqueness and strong existence} % (fold)
  \label{sub:pathwise_uniqueness}
  
  % subsection pathwise_uniqueness (end)
  We denote
  \begin{equation}\label{def.tauH}
    \tau_H(p,q)=1-H\frac dp-\frac1q\,.
  \end{equation}
  It is useful to observe that the conditions \eqref{con:weakfbm} and \eqref{con:fbmpq} are respectively equivalent to
  \begin{equation}
    \tau_H(p,q)>\frac12 
    \quad\textrm{and}\quad \tau_H(p,q)>H+\frac12\,.
  \end{equation}
  We recall that $\{P_\sigma\}_{\sigma\ge0}$ is the heat semigroup and $\sigma_H$ is defined in \eqref{def:sigma}. The following lemma will be needed.
  \begin{lemma}
    For every $f\in L^q([0,T];L^p(\Rd) )$ and every $s<t$, we have
    \begin{align}
      \label{pf2}&\int_s^t\|\nabla P_{\sigma^2_H(s,r)}f_r\|_\infty dr\lesssim\|f\|_{L^qL^p} |t-s|^{\tau_H(p,q)-H}\,,\quad&\textrm{if }& \tau_H(p,q)>H\,,
      \\\label{pf7}&\int_s^t\|\nabla^2P_{\sigma^2_H(s,r)}f_r\|_\infty dr\lesssim\|f\|_{L^qL^p} |t-s|^{\tau_H(p,q)-2H}\,,\quad&\textrm{if }& \tau_H(p,q)>2H\,.
    \end{align}
  \end{lemma}
  \begin{proof}
    We recall that for every $\sigma>0$, $P_\sigma$ maps $L^p(\Rd)$ to $C^2_b(\R^d)$. In addition, for every $\phi\in L^p(\Rd)$, applying H\"older inequality, we have
    \begin{gather}
      \label{Sch:f2}\|\nabla P_{\sigma}\phi\|_\infty\lesssim \sigma^{-\frac d{2p}-\frac12}\|\phi\|_{L^p}
      \quad\textrm{and}\quad
      % \\\label{Sch:f2}\textrm{and}\quad
      \|\nabla^2 P_{\sigma}\phi\|_\infty\lesssim \sigma^{-\frac d{2p}-1}\|\phi\|_{L^p}\,.
    \end{gather}
    Applying the former inequality in \eqref{Sch:f2} and H\"older inequality yields
    \begin{align*}
      \int_s^t\|\nabla P_{\sigma^2_H(s,r)}f_r\|_\infty dr
      &\lesssim\int_s^t|\sigma_H(s,r)|^{-\frac dp-1}\|f_r\|_{L^p}dr
      \\&\lesssim\|f\|_{L^qL^p} \(\int_s^t|\sigma_H(s,r)|^{-q'\frac dp-q'}dr\)^{\frac1{q'}}\,,
    \end{align*}
    where $q'=\frac q{q-1}$, the H\"older conjugate of $q$.
    Using \eqref{NDP}, it is evident that
    \begin{equation*}
        \(\int_s^t|\sigma_H(s,r)|^{-q'\frac dp-q'}dr\)^{\frac1{q'}}\lesssim \(\int_s^t|r-s|^{-Hq'\frac dp-Hq'}dr\)^{\frac1{q'}}
        \lesssim |t-s|^{1-H\frac dp-\frac1q-H}\,.
    \end{equation*}
    The above two inequalities imply \eqref{pf2}. The estimate \eqref{pf7} is obtained analogously by applying H\"older inequality and the later inequality in \eqref{Sch:f2}.
  \end{proof}
  \begin{remark}
    In the previous proof, we have used approximations of $L^{q_1}L^{p_1}$-functions by smooth functions, which requires that $q_1,p_1$ are finite.
  \end{remark}
  Let us now fix a filtered probability space $(\Omega,\cff,\P,\{\cff_t\}_{t\in[0,T]})$ such that $\cff_0$ contains $\P$-null sets. On this probability space, $B^H$ is an $\{\cff_t\}$-fractional Brownian motion. The following result is a kind of division property.

  \begin{lemma}\label{lem.yyv}
    Let $(B^H,X,\{\cff_t\})$ and $(B^H,\bar X,\{\cff_t\})$ be two pathwise solutions to \eqref{eqn:fSDE}.
    Let  $f$ be in $L^{q_1}([0,T];L^{p_1}(\Rd))$, $q_1,p_1\in[1,\infty)$. Assume that \eqref{con:weakfbm} holds, $\tau_H(p_1,q_1)>H+\frac12$  and $\tau_H(p,q)+\tau_H(p_1,q_1)>1+2H$.
    Putting $\psi=X-B^H$ and $\bar \psi=\bar X-B^H$, there exist modifications of $\psi,\bar \psi$ and a H\"older continuous process $V=V(f,\psi,\bar \psi)$ such that with probability one,
    \begin{equation}\label{eqn:yyv}
      \int_0^tf_r(X_r)dr-\int_0^tf_r(\bar X_r)dr=\int_0^t(\psi_r-\bar \psi_r)\cdot dV_r\,, \quad\forall t\in[0,T].
    \end{equation}
    In the above, $\int_0^\cdot (\psi_r-\bar \psi_r)\cdot d V_r $ is a (well-defined) Young integral. 
  \end{lemma}
  \begin{proof}
    Herein, we put $\tau=\tau_H(p,q)$ and $\tau_1=\tau_H(p_1,q_1)$.
    Let $\{f^n\}$ be a sequence of functions in $C^2_b([0,T];\Rd)$ convergent to $f$ in $L^{q_1}([0,T];L^{p_1}(\Rd))$. Then, it is evident that
    \begin{equation}\label{tmp,fnv}
      \int_0^tf_r^n(X_r)dr-\int_0^tf_r^n(\bar X_r)dr=\int_0^t(\psi_r-\bar \psi_r)\cdot dV_r[f^n]\,, \quad\forall t\in[0,T],
    \end{equation}
    where
    \begin{align*}
      V_t[f^n]=\int_0^t\int_0^1 \nabla f_r^n(B^H_r+\theta \psi_r+(1- \theta)\bar\psi_r)d \theta dr\,.
    \end{align*}
    Our strategy is to pass through the limit $n\to\infty$ in \eqref{tmp,fnv}.
    From Proposition \ref{prop.fx} and Garsia-Rodemich-Rumsey inequality, we see that left-hand side of \eqref{tmp,fnv} converges in probability uniformly on $\left[0,T\right]$ to
    \begin{align*}
      \int_0^tf_r(B^H_r+\psi_r)dr-\int_0^tf_r(B^H_r+\bar \psi_r)dr\,.
    \end{align*}
    From Proposition \ref{prop:regweaksoln}, choosing $m$ sufficiently large and applying Kolmogorov continuity theorem, we see that $\psi$ and $\bar \psi$ have continuous modifications in $C^{\frac12+\varepsilon}([0,T];\Rd)$ for some $\varepsilon>0$. Hence, by continuity of Young integrations, to pass through the limit $n\to\infty$ in the right-hand side of \eqref{tmp,fnv}, it suffices to show that $V[f^n]$ converges to a process $V[f]$ in $C^{\frac12+\varepsilon}([0,T])$ in probability for some $\varepsilon>0$. This is accomplished below via the stochastic sewing lemma, which is similar to Proposition \ref{prop:defA} and Corollary \ref{cor:Acont}.

    We first construct an auxiliary process.
    For each $g\in L^{q_1}([0,T];L^{p_1}(\Rd))$, consider
    \begin{equation*}
      A_{s,t}[g]=\int_s^t\int_0^1 \nabla (P_{\sigma^2_H(s,r)} g_r)(\E^{\cff_s}B^H_r+\theta\psi_s+(1- \theta)\bar\psi_s)d \theta dr\,.
    \end{equation*} 
    Let $m\ge2$ be fixed. From \eqref{pf2}, we have
    \begin{align*}
      \|A_{s,t}[g]\|_{L_m}\lesssim\|g\|_{L^{q_1}L^{p_1}}|t-s|^{\tau_1-H}
      \,.
    \end{align*}
    To simplify the notation, we denote $z_s^\theta=\theta \psi_s+(1- \theta)\bar \psi_s $. Then
    \begin{align*}
      &\E^{\cff_s}\delta A_{s,u,t}[g]
      \\&=\E^{\cff_s} \int_u^t\int_0^1 \(\nabla (P_{\sigma^2_H(u,r)} g_r)(\E^{\cff_u}B^H_r+z^\theta_s)-\nabla (P_{\sigma^2_H(u,r)} g_r)(\E^{\cff_u}B^H_r+z^\theta_u)\)d \theta dr\,.
    \end{align*}
    Applying \eqref{pf7}, we have
    \begin{align*}
      \|\E^{\cff_s}\delta A_{s,u,t}[g]\|_{L_m}\lesssim\|g\|_{L^{q_1}L^{p_1}}(\|\psi\|_{C^\tau L_m}+\|\bar \psi\|_{C^\tau L_m}) |t-s|^{\tau_1+\tau-2H}\,.
    \end{align*}
    (Recall that $\|\cdot\|_{C^\tau L_m}$ is defined at the beginning of Section \ref{sec:additive_functionals}.)
    From our assumptions, $\tau_1-H>\frac12$ and $\tau_1+\tau-2H>1$, so by Theorem \ref{lem:Sew2}, there exists an adapted process in $C^{\tau_H(p_1,q_1)-H}L_m$, denoted by $\caa[g]$, such that
    \begin{align}\label{tmp.v1}
      \|\caa_t[g]-\caa_s[g]\|_{L_m}\lesssim\|g\|_{L^{q_1}L^{p_1}}(1+\|\psi\|_{C^\tau L_m}+\|\bar \psi\|_{C^\tau L_m})|t-s|^{\tau_1-H}\,,
    \end{align}
    % and
    % \begin{align}\label{tmp.v2}
    %   \|\E^{\cff_s}(\caa_t[g]-\caa_s[g])-A_{s,t}[g]\|_{L_m}\lesssim\|g\|_{L^{q_1}L^{p_1}}(\|\psi\|_{C^\tau L_m}+\|\bar \psi\|_{C^\tau L_m})|t-s|^{\tau_1+\tau-2H}\,,
    % \end{align}
    for every $s\le t$ in $[0,T]$. By choosing $m$ sufficiently large and applying Kolmogorov continuity theorem, we see that $\caa[g]$ has continuous modification in $C^\alpha([0,T])$ for every $\alpha<\tau_1-H$.

    We claim that $V[f^n]=\caa[f^n]$. Indeed, since $\nabla f^n$ is bounded and Lipschitz, applying Proposition \ref{prop:regweaksoln}, it is straightforward to verify that for every $s\le t$,
    \begin{align*}
      \|V_t[f^n]-V_s[f^n]\|_{L_m}\le\|\nabla f^n\|_{C_b([0,T]\times\Rd)}|t-s|
    \end{align*}
    and
    \begin{align*}
      &\|\E^{\cff_s}(V_t[f^n]-V_s[f^n])-A_{s,t}[f^n]\|_{L_m}
      \\&=\|\E^{\cff_s}\int_s^t\int_0^1\left[\nabla f^n_r(B^H_r+z^\theta_r)-\nabla f^n_r(B^H_r+z^\theta_s)\right]d \theta dr\|_{L_m}
      \\&\le\|f^n\|_{C^2_b}(\|\psi\|_{C^{\tau} L_m}+\|\bar \psi\|_{C^{\tau} L_m})|t-s|^{1+\tau}.
    \end{align*}
    Hence, by uniqueness of Theorem \ref{lem:Sew2}, we must have $V_t[f^n]=\caa_t[f^n]$ a.s. for every $t\in[0,T]$.

    We now define $V[f]$ as the continuous modification of $\caa[f]$. Using the estimate \eqref{tmp.v1}, we see that $\lim_nV[f^n]=V[f]$ in $C^{\tau_1-H}_TL_m$. By choosing $m$ sufficiently large and applying Garsia-Rodemich-Rumsey inequality, it follows that the sequence $\{V[f^n]\}$ converges to $V[f]$ in $C^{\frac12+\varepsilon}([0,T];\Rd)$ in probability for some $\varepsilon>0$. This completes the proof. 
  \end{proof}
  
  We now have enough material to show pathwise uniqueness for \eqref{eqn:fSDE}.
  \begin{proof}[Proof of Theorem \ref{thm:puni}]
    Suppose $(B^H,X,\{\cff_t\}_{t\in[0,T]}),(B^H,\bar X,\{\cff_t\}_{t\in[0,T]})$ are two pathwise solutions to \eqref{eqn:fSDE} defined on the same probability space $(\Omega,\cff,\P)$ such that $X_0=\bar X_0=x$. From Proposition \ref{prop:regweaksoln}, using Kolmogorov continuity theorem, we see that $\psi,\bar \psi$ are a.s. H\"older continuous with exponent $\frac12+\varepsilon$ for some $\varepsilon>0$.
    Then with probability one, 
    \begin{equation*}
      \psi_t-\bar \psi_t=\int_0^t b_r(B^H_r+\psi_r)dr-\int_0^tb_r(B^H_r+\bar\psi_r)dr
      \quad\forall t\in[0,T]\,.
    \end{equation*}
    Since $[0,T]$ has finite length, we can assume without loss of generality that $q<\infty$. We consider two cases. 

    \textit{Case 1: $p<\infty$.}
    From Lemma \ref{lem.yyv}, (for a.s. $\omega$) we can rewrite the above equation to
    \begin{equation}\label{tmp.yy}
      \psi^i_t-\bar \psi^i_t=\int_0^t(\psi_r-\bar \psi_r)\cdot dV^i_r\,, \quad\forall t\in[0,T]\,,\forall i=1,\dots,d\,,
    \end{equation} 
    where for each $i$, $\psi^i,\bar \psi^i$ are respectively the $i$-th components of $\psi,\bar \psi$, $V^i=V(b^i,\psi,\bar \psi)$ is the process defined in Lemma \ref{lem.yyv}. 
    Note that each integral on the right-hand side \eqref{tmp.yy} is a Young integral. Equation \eqref{tmp.yy} is a Young differential equation, which has uniqueness (\cite{MR2397797}). 
    This implies that $\psi=\bar \psi$ and hence $X=\bar X$.  

    \textit{Case 2: $p=\infty$.} For each $n$, we denote $\chi_n(x)=\mathbf{1}_{\{|x|\le n\}}$. Choose $p_1\in[1,\infty)$ such that $H\frac d{p_1}+\frac1q<\frac12-H$. 
    Let $\Omega^*\in\cff$ be such that on $\Omega^*$, $\psi,\bar \psi$ are H\"older continuous with exponent $\frac12+\varepsilon$ for some $\varepsilon>0$ and
    \begin{equation*}
      \int_0^t[b \chi_n](r,X_r)dr-\int_0^t[b \chi_n](r,\bar X_r)dr=\int_0^t(\psi_r-\bar \psi_r)\cdot dV^{(n)}_r \quad\forall t\in[0,T],\forall n\ge1.
     \end{equation*} 
    Here, $V^{(n)}=(V^{(n),1},\dots,V^{(n),d})$ and for each $j=1,\dots,d$, $V^{(n),j}=V(b^j \chi_n,\psi,\bar \psi)$ is the process constructed in Lemma \ref{lem.yyv}. It is clear that we can find $\varepsilon>0$ so that $\P(\Omega^*)=1$.
    
    Let us fix an arbitrary $\omega$ in $ \Omega^*$. We define
    \[
      \sigma_n(\omega)=\inf\{t\in[0,T]:|\psi_t(\omega)|>n \textrm{ or }|\bar \psi_t(\omega)|>n\}\,.
    \]
    For every $t\in[0,\sigma_n(\omega)]$, we have
    \begin{align*}
      \psi_t(\omega)-\bar \psi_t(\omega)
      &=\int_0^t [b \chi_n](r,X_r(\omega))dr-[b \chi_n](r,\bar X_r(\omega))dr
      \\&=\int_0^t(\psi_r(\omega)-\bar \psi_r(\omega))\cdot dV^{(n)}_r(\omega)\,.
    \end{align*}
    As in the first case, uniqueness of Young differential equations implies that $\psi_t(\omega)=\bar \psi_t(\omega)$ for every $t\in[0,\sigma_n(\omega)]$. It is obvious that $\lim_n \sigma_n(\omega)=T$. Hence, we conclude that $\psi(\omega)=\bar \psi(\omega)$ on $[0,T]$, which implies $X=\bar X$.   
  \end{proof}
  \begin{proof}[Proof of Theorem \ref{thm:fbm}]
    Uniqueness of strong solutions is a consequence of pathwise uniqueness obtained in Theorem \ref{thm:puni}. To show existence of a strong solution, we rely on Yamada-Watanabe's result \cite[Corollary 1]{MR0278420}. 
    We recall the identity \eqref{id:fbm} and observe that $(B^H,X,\{\cff_t\})$ is a solution to \eqref{eqn:fSDE} if and only if $(W,X,\{\cff_t\})$ satisfies (i) $W$ is an $\{\cff_t\}$-Wiener process, (ii) $X$ is $\{\cff_t\}$-adapted and (iii) with probability one
    \begin{equation}\label{eqn:KW}
      X_t=x+\int_0^tb(r,X_r)dr+\int_0^tK_H(t,s)dW_s\,,\quad\forall t\in[0,T]\,.
    \end{equation}
    The results of Theorems \ref{thm:weakfsde} and \ref{thm:fbm} are transfered in obvious ways to equation \eqref{eqn:KW}.
    The argument of \cite{MR0278420} is then applicable to equation \eqref{eqn:KW} as it relies only on properties of Wiener processes, weak existence and pathwise uniqueness. In particular, it does not depend on a particular form of the equation. This shows existence of a strong solution to \eqref{eqn:KW}. Since $B^H$ and $W$ generate the same filtration, a strong solution to \eqref{eqn:KW} is also a strong solution to \eqref{eqn:fSDE}. This completes the proof.
  \end{proof}
  
  When $H=1/2$, a well-known result of Krylov and R\"ockner \cite{MR2117951} states that \eqref{eqn:fSDE} has unique strong solution if 
  \begin{equation}\label{kry}
    \frac dp+\frac2q<1 \quad\Leftrightarrow \quad \tau_{1/2}(p,q)>\frac12\,.
  \end{equation}
  Our condition \eqref{con:fbmpq} is therefore not optimal because it forces $p,q\to\infty$ when $H\to\frac12^-$.
  The fraction $\frac12$ which appears in \eqref{kry} is the same as the one in \eqref{con:dA2} and seems to be fixed for any $H<\frac12$ in order to apply Girsanov transformation.
  It is then natural to expect the following result.
  \begin{conjecture}
    Equation \eqref{eqn:fSDE} has a unique strong solution provided that
    \begin{equation}
      \tau_H(p,q)>\frac12 \,,
    \end{equation}
    which is indeed equivalent to \eqref{con:weakfbm}.
  \end{conjecture}
  \noindent The validity of this conjecture is out of reach at the moment of writing and is an interesting problem by its own. 
% section sdes_driven_by_fractional_brownian_motions (end)
\section{Averaging along fractional Brownian motions} % (fold)
\label{sec:averaging_along_fractional_brown}
  Let $B^H$ be an $\{\cff_t\}$-fractional Brownian motion in $\Rd$ with $H\in(0,1)$ and $f$  be an element in $L^q_T\C^\nu(\Rd)$ for some $\nu<1$. 
  We recall Definition \ref{def.Besov} of $\C^\nu(\Rd)$ for $\gamma\le0$. For $\nu\in(0,1)$, $\C^\nu(\Rd)$ is the same as the space of bounded H\"older continuous functions $C^\nu_b(\Rd)$. The space $L^q_T\C^\nu(\Rd)$ is defined at the beginning of Section \ref{sec:additive_functionals} with $\chh=\C^\nu(\Rd)$. 
  In the current section, we investigate the almost sure joint-H\"older regularity of the following map
  \begin{equation*}
    (t,x)\mapsto\int_0^t f(r,B^H_r+x)dr\,.
  \end{equation*}
  % Since $f(r,\cdot)$ is a distribution, the integratio
  In particular, we seek for almost sure estimates of the type
  \begin{equation}\label{tmp:jholder}
    |\int_s^t f(r,B^H_r+x)dr-\int_s^tf(r,B^H_r+y)dr|\lesssim |t-s|^{\gamma}|x-y|^\alpha
  \end{equation}
  uniformly for $s,t,x,y$ in compact sets, for some suitable exponents $\gamma, \alpha>0$.
  This type of estimates plays a key role in the work \cite{MR3505229} of Catellier and Gubinelli in obtaining path-by-path uniqueness for stochastic differential equations of the type
  \begin{equation}\label{tmp.CG}
    dX_t=b(t,X_t)dt+dB^H_t\,, \quad b\in L^\infty_T\C^\nu\,.
  \end{equation}
  They observe that the remainder process $\psi=X-B^H$ satisfies a non-linear Young integral equation in which the driving space-time vector field is $(t,x)\mapsto \int_0^tb(r,B^H_r+x)dr$. The a.s. estimate \eqref{tmp:jholder} then provides necessary regularity for this space-time vector field so that the equation for $\psi$ has a unique solution.
  Let us point out that the approach of \cite{MR3505229} is an extension of the earlier work \cite{MR2377011} of Davie.
  Non-linear Young integral equations are also considered in \cite{MR3581224} although connections with equation \eqref{tmp.CG} and Davie's work were not observed there.

  Well-posedness for non-linear Young differential equations are well-understood (see \cite[Section 2]{MR3505229} or \cite[Section 3]{MR3581224}). Hence we will not consider them here. We rather provide another way of obtaining estimates of the type \eqref{tmp:jholder}, using stochastic sewing lemma (more specifically Theorem \ref{lem:Sew2}).
  Our approach allows for time-dependent vector fields, which are not considered in \cite{MR3505229}.
  Comparing the two papers, most of estimates in \cite{MR3505229} can be recovered from ours by simply setting the parameter $q$  to be $\infty$. In addition, our treatment provides explicit H\"older exponents in space and time simultaneously.

  Let $\{\cff_t\}_{t\ge0}$ be the natural filtration generated by $B^H$. 
  Let us consider for each $s<t$ and $f\in L^q_T\C^\nu(\Rd)$, the increment process
  \begin{equation*}
    A_{s,t}^x[f]=\int_{[s,t]}[P_{\sigma^2_H(s,r)}f_r](\E^{\cff_s}B^H_r+x)dr\,,
  \end{equation*}
  where $\{P_\sigma\}_{\sigma\ge0}$ is the heat semigroup associated to the kernel $(2 \pi \sigma)^{-\frac d2}e^{-|x-y|^2/(2 \sigma)}$, $\sigma_H$ is defined in \eqref{def:sigma}. If $f$ is a bounded Borel function, we can also write $A_{s,t}^x[f]=\int_{[s,t]}\E^{\cff_s}f_r(B^H_r+x)dr$.
    We recall that for every $\sigma>0$, $P_\sigma$ maps $\C^\nu(\Rd)$ to $C^2_b(\R^d)$. In addition, for every $\alpha\in(0,1]$ and for every $\phi\in \C^\nu(\Rd)$, we have (see  e.g. \cite[Prop. 2.4]{MR3785598})
  \begin{gather}
    \label{Sch:f3}\|P_{\sigma}\phi\|_\infty\lesssim \sigma^{\frac \nu2}\|\phi\|_{\C^\nu}
    \quad\textrm{and}\quad
    \|P_{\sigma}\phi\|_{\C^\alpha}\lesssim \sigma^{\frac {\nu- \alpha}2}\|\phi\|_{\C^\nu}\,.
  \end{gather}
  In analogy with \eqref{pf2} and \eqref{pf7}, it is straightforward to obtain following estimates
  \begin{gather}
    \label{pf3}\int_{[s,t]}\|P_{\sigma^2_H(s,r)}f_r\|_\infty dr\lesssim\|f\|_{L^q\C^\nu} |t-s|^{1+H\nu-\frac1q}\,,
    \\\label{pf4}\int_{[s,t]}\|P_{\sigma^2_H(s,r)}f_r\|_{\C^\alpha} dr\lesssim\|f\|_{L^q\C^\nu} |t-s|^{1+H(\nu- \alpha)-\frac1q}\,,
  \end{gather}
  for every $f\in L^q_T \C^\nu(\Rd)$ and every $s<t$, provided that $1+H(\nu- \alpha)-\frac1q>0$. These estimates are summarized in the following result.
  \begin{lemma}\label{lem:fAnu} Let $s<u<t$, $f\in L^q_T\C^\nu(\Rd)$ and $x,y\in\Rd$ be fixed. Suppose that $\alpha\in(0,1]$ and $1+H(\nu- \alpha)-\frac1q>0$, then
    \begin{gather}
      \label{est:PH1nu}|A^x_{s,t}[f]|\lesssim \|f\|_{L^q_T\C^\nu}|t-s|^{1+H \nu-\frac1q}
    \end{gather}
    and
    \begin{equation}
      \label{est:PH3nu}|A^{x}_{s,t}[f]-A^{y}_{s,t}[f]|\lesssim\|f\|_{L^q_T\C^\nu}|x-y|^\alpha|t-s|^{1+H(\nu- \alpha)-\frac1q}\,.
    \end{equation}
  \end{lemma}
  We define
  \begin{equation}
    \gamma_H(\nu,q)=1+H \nu-\frac1q\,.
  \end{equation}
  The following result is an analogue of Proposition \ref{prop:defA}.
    \begin{proposition}\label{prop:defAfbm}
    Let $m\ge2$ and $x\in\Rd$ be fixed. Suppose that
    \begin{equation}\label{con:nuqfbm}
      \gamma_H(\nu,q)>\frac12 \quad\Leftrightarrow \quad \nu>\frac1H\(\frac1q-\frac12\)\,.
    \end{equation} 
    There exists a linear map $\caa^x$ from $L^q_T\C^\nu(\Rd)$ to $C^{1+H \nu-\frac1q}_TL_m$ such that 
    \begin{enumerate}[label={\upshape(\alph*)}]
      \item\label{pa0fbm} For every $f\in L^q_T\C^\nu(\Rd)$ and every $0\le s\le t\le T$,
      \begin{equation}\label{est:Astfbm}
        \|\caa^x_t[f]-\caa^x_s[f]\|_{L_m}\lesssim \|f1_{[s,t]}\|_{L^q_T\C^\nu}|t-s|^{1+H\nu-\frac1q}\,.
      \end{equation}
      
      \item\label{pa2fbm} For every $f\in L^q_T\C^\nu(\Rd)$ and every $t\ge0$, $\caa_t[f]$ is the limit in $L_m$ of the Riemann sum
      \begin{equation*}
        \sum_i \int_{[{t_i},{t_{i+1}}]}P_{\sigma^2_H(t_i,r)}f_r(\E^{\cff_{t_i}}B^H_{r}+x)dr
      \end{equation*}
      as $\max_i|t_{i+1}-t_i|\to0$. Here $\{t_i\}_i$ is any partition of $[0,t]$.
      \item\label{pa1fbm} If $f$ is a  bounded measurable function, then $\caa_t^x[f]=\int_{[0,t]} f_r(B^H_r+x)dr $ a.s. for every $t\ge0$. 
    \end{enumerate}
  \end{proposition}
  \begin{proof}
    Note that $\E^{\cff_s}\delta A_{s,u,t}^x[f]=0$ for every triplet $s<u<t$. Similar to Proposition \ref{prop:defA}, we can obtain \ref{pa0fbm} and \ref{pa2} by using \eqref{est:PH1nu} and Theorem \ref{lem:Sew2}. Suppose that $f$ is a bounded measurable function. Then we have $A_{s,t}^x[f]=\int_{[s,t]}\E^{\cff_s}f_r(B^H_r+x)dr$ and
    \[
      \|\int_{[s,t]}f_r(X_r)dr-A_{s,t}^x[f]\|_{L_m}\le 2\sup_{r,x}|f(r,x)||t-s|.
    \]
    Similar to Proposition \ref{prop.K}, we obtain \ref{pa1fbm} by uniqueness of Theorem \ref{lem:Sew2}.
  \end{proof}
  
  Hereafter, we fix an element $f$ in $L^q_T\C^\nu(\Rd)$ and write $\int_0^tf_r(B^H_r+x)dr $ for $\caa_t^x[f]$ whenever the hypotheses of Proposition \ref{prop:defAfbm} are met.
  \begin{proposition}
    For every $\alpha\in(0,1]$ such that
    \begin{equation}
      \gamma_H(\nu,q)>\frac12+H \alpha \quad\Leftrightarrow \quad \alpha<\nu-\frac1H\(\frac1q-\frac12\)\,,
    \end{equation}
    then 
    \begin{equation}\label{est:f2}
      \|\int_s^tf_r(B^H_r+x)dr-\int_s^tf_r(B^H_r+y)dr \|_{L_m}\lesssim\|f\|_{L^q_T\C^\nu} |t-s|^{\gamma_H(\nu,q)-H \alpha}|x-y|^\alpha.
    \end{equation}
  \end{proposition}
  \begin{proof}
    This is a direct consequence of \eqref{est:PH3nu} and Theorem \ref{lem:Sew2}. 
  \end{proof}
  By choosing $m$ sufficiently large and applying the Kolmogorov continuity theorem, there exists a continuous modification of the random field $\{\int_0^tf_r(B^H_r+x)\}_{t,x}$. We then apply the multiparameter Garsia-Rodemich-Rumsey inequality in \cite{hule2012} (choosing $m$ sufficiently large) to see that there are arbitrarily small positive constants $\varepsilon_1,\varepsilon_2$ such that with probability one,
  \begin{multline}\label{est:fxy}
    |\int_s^tf_r(B^H_r+x)dr-\int_s^tf_r(B^H_r+y)dr|\lesssim\|f\|_{L^q_T\C^\nu} |t-s|^{\gamma_H(\nu,q)-H \alpha- \varepsilon_1}|x-y|^{\alpha- \varepsilon_2}
  \end{multline}
  for uniformly in $s,t\in[0,T]$ and $x,y$ in a compact set. 
  This procedure provides a method to obtain almost sure estimates of the type \eqref{tmp:jholder}. It can be applied to most of the moment estimates stated below.

  \begin{proposition}
    Suppose that $q\in[1,\infty)$ and
    \begin{equation}\label{con:n1fbm}
      \gamma_H(\nu,q)>\frac12+H \quad\Leftrightarrow \quad \nu>1+\frac1H\(\frac1q-\frac12\)\,,
    \end{equation}
    then the map $x\to\int_0^t f_r(B^H_r+x) $ is differentiable
    and for every fixed $t,x$
    \begin{equation}\label{id:gradf}
      \nabla \int_0^t f_r(B^H_r+x)dr=\int_0^t\nabla f_r(B^H_r+x)dr \quad \textrm{a.s.}
    \end{equation}
    In addition, the map $(t,x)\to \nabla\int_0^t f_r(B^H_r+x) $ has a continuous version. For every $\alpha\in(0,1]$ such that
    \begin{equation}
      \gamma_H(\nu,q)>\frac12+H+H \alpha \quad\Leftrightarrow \quad \alpha<\nu-1-\frac1H\(\frac1q-\frac12\)\,,  
    \end{equation} we have the following estimates for every $s<t$ and every $x,y\in\Rd$
    \begin{equation}\label{est:gradf1}
      \|\nabla\int_s^tf_r(B^H_r+x)\|_{L_m}\lesssim\|f\|_{L^q_T\C^\nu} |t-s|^{\gamma_H(\nu,q)-H}
    \end{equation}
    and
    \begin{equation}\label{est:gradf2}
      \|\nabla\int_s^tf_r(B^H_r+x)dr-\nabla\int_s^tf_r(B^H_r+y)dr\|_{L_m}\lesssim\|f\|_{L^q_T\C^\nu} |t-s|^{\gamma_H(\nu,q)-H-H \alpha}|x-y|^\alpha\,.
    \end{equation}
  \end{proposition}
  \begin{proof}
    We observe that $\nabla f$ belongs to $L^q_T\C^{\nu-1}(\Rd)$ with $\|\nabla f\|_{L^q_T\C^{\nu-1}}\le\|f\|_{L^q_T\C^\nu}$ and that condition \eqref{con:n1fbm} is equivalent to $\gamma_H(\nu-1,q)>\frac12$. By Proposition \ref{prop:defAfbm}, the random field $\{\int_0^t\nabla f_r(B^H_r+x)dr\}_{t,x} $ can be defined and has a continuous modification.
    Since $q<\infty$, we can choose a sequence $\{f^n\}$ in $\cee_T^2$ which converges to $f$ in $L^q_T\C^\nu$ (\cite[Lemma 1.2.19]{MR3617205}). Then from Proposition \ref{prop:defAfbm}, we have for each $t\in[0,T]$,
    \begin{equation*}
      \int_0^t\nabla f_r(B^H_r+x)dr=\lim_n\int_{[0,t]}\nabla f_r^n(B^H_r+x)dr \quad\textrm{in} \quad L_m
    \end{equation*}
    and for each $n$
    \begin{equation*}
      \int_{[0,t]}\nabla f_r^n(B^H_r+x)dr=\nabla\int_{[0,t]} f_r^n(B^H_r+x)dr\,.
    \end{equation*}
    The later identity is due to the fact that each $f_n$ is a finite linear combination of $C^2_b(\Rd)$-functions. This shows that $x\to\int_0^tf_r(B^H_r+x)dr$ is differentiable and \eqref{id:gradf} holds. The estimates \eqref{est:gradf1} and \eqref{est:gradf2} follow from identity \eqref{id:gradf} and estimates \eqref{est:Astfbm}, \eqref{est:f2} with $f$ replaced by $\nabla f$.  
  \end{proof}

% section averaging_along_fractional_brown (end)
% \bibliographystyle{abbrv}
% \bibliographystyle{nar}
% \bibliographystyle{plain}
% \bibliography{../bibliography/processes}
% \bibliography{../../bibliography/bib}

\section*{Acknowledgment} % (fold)
\label{sec:acknowledgment}
 The author thanks Leonid Mytnik, Oleg Butkovsky, Siva Athreya and Giuseppe Cannizzaro for interesting and motivating discussions.  He also thanks the editor, an associated editor and several anonymous referees for numerous constructive suggestions which help improve the presentation of the article.
\end{document}